\documentclass[10pt]{article}
\RequirePackage{amsthm,amsmath,amsfonts,amssymb}
\RequirePackage[numbers]{natbib}
\usepackage{float}

\usepackage{graphicx}
\usepackage{prodint}
\usepackage{hyperref}
\hypersetup{colorlinks=true,linkcolor=blue,filecolor=magenta,urlcolor=cyan}
\usepackage{authblk}
\theoremstyle{plain}
\theoremstyle{definition}
\newtheorem{theorem}{Theorem}[section]
\newtheorem{corollary}{Corollary}[theorem]
\newtheorem{lemma}[theorem]{Lemma}
\newtheorem{proposition}[theorem]{Proposition}
\newtheorem{definition}{Definition}[section]

\renewcommand{\Pr}{\mathbf{P}}

\newcommand{\mymultA}{.4}
\newcommand{\mymultB}{.35}

\newcommand{\RED}[1]{{#1}}
\newcommand{\BLUE}[1]{{#1}}

\newcommand{\OLDPROOF}[1]{}

\title{``Game, Set, Match'' \\
Double Delight Watching a Grand Slam Tennis Match}
\author{Edsel A.\ Pe\~na\footnote{E.\ Pe\~na is Professor and Chair, Department of Statistics, University of South Carolina. {\em Email: pena@stat.sc.edu}} \quad Dip Das\footnote{D.\ Das is Graduate Student in the Department of Statistics, University of South Carolina. {\em Email: dipd@email.sc.edu}} \quad\ Yuexuan Wu\footnote{Y.\ Wu is Assistant Professor, Department of Statistics, University of South Carolina. {\em Email: yuexuan@sc.edu.}}}
\date{}

\begin{document}

\maketitle
	
\abstract{In this pedagogically-oriented paper, probabilistic properties of tennis scoring systems are comprehensively examined and compared with best-of-$K$ systems. A simplistic and idealized mathematical model, where each player has his/her own probability of winning his/her service point and which remains invariant for the duration of the match, and where outcomes of points played are independent of each other, is assumed. Given these service success probabilities, probabilities of winning a game tie-breaker, a game, a set tie-breaker's tie-breaker, a set tie-breaker, a set, and the match are obtained. Since tennis scoring systems are unique, probability calculations require decomposing big and complicated problems into smaller and simpler constituent problems, solving these sub-problems, then combining to obtain the solution to the big problem. The problems that arise from tennis scoring systems offer excellent pedagogical venues for teaching probability, in particular, the use of the Theorem of Total Probability and the Iterated Rules for Mean, Variance, and Covariance. There are also many interesting questions in tennis, foremost of which is whether a tennis match under this assumption will {\em actually} end with probability one; whether when two players of `equal abilities' play a match, the first server possesses an advantage; \BLUE{or whether who serves first affects other aspects of the match such as the total number of points played}. Some of these intriguing questions are addressed in this work. Tennis scoring systems are technically statistical decision systems to determine the better player. However, since such a decision system is based on just a {\em finite} number of points played, erroneous decisions could arise, such as the inferior player winning the match. Of interest, therefore, is to compare different systems in terms of the probability of the better player winning, as well as the duration of the match in terms of the number of points played. Such comparisons are done by introducing a measure of a scoring system's efficiency. A best-of-$(2L+1)$-games system, with a tie-breaker consisting of games where the winner is the first to achieve an advantage of two games, appears comparable to the tennis' best-of-five-sets system when $L=22$, though the number of points played may drastically differ between the two systems, especially when both players are very outstanding (or, equivalently, very mediocre) servers against each other. \BLUE{A best-of-60-games system with a 10-point tie-breaker system was also compared to the tennis best-of-five sets match system.} A myriad of statistical questions  for future work are also described.}

\medskip


\maketitle

\section{Double Joys of Watching a Tennis Match}
\label{sec: US Open}

{\em Have you recently watched a Grand Slam tennis match; a game in the Baseball World Series; or an American football game in the College Football Playoff (CFP) National Championship, wherein the match reaches a tie-breaker stage: a 6-6 score in tennis on the fifth set; a tied score after nine innings in baseball; or a tied score after four quarters in football? Suddenly, looking at your watch, you start wondering whether the match might go on forever and never finish, especially after it has reached a tie-breaker score of 16-16 points in tennis;  still a tied score after 15 innings in baseball; or,  still a tied score after 3 overtime cycles in football?}

These thoughts entered the mind of the senior and lead author, who had the wonderful delight of watching a Round-of-16 US Open tennis match between the eventual winner of the 2025 US Open Tennis Championship, Carlos Alcaraz of Spain, and Arthur Rinderknech of France, at Arthur Ashe Stadium in Flushing, New York City last August 31, 2025 (see Figure \ref{fig: Alcaraz serve}). However, it was more than just {\em a} delight, since a statistician watching a top-level tennis match also delights in the myriad mathematical, probabilistic, and statistical aspects of the match, such as will a tie-breaker {\em ever} end? In addition, a tennis match contains both asymmetries and symmetries. The player who serves in a point or game usually has an advantage for that point or game, an asymmetry; but, because of the tennis scoring system, there is also an inherent symmetry in the match. The senior author, while watching the match, aside from the question of a {\em never-ending} match, started to wonder as well whether the player who serves first, even if the two players are hypothetically of equal abilities (hypothetically, since of `equal abilities' may never happen in reality), has a higher probability of winning the match. As such, it was for the senior author a {\em double} delight watching the match!
\begin{figure}[h]
\begin{center}
\includegraphics[width=\mymultB\textwidth,height=\mymultB\textwidth]{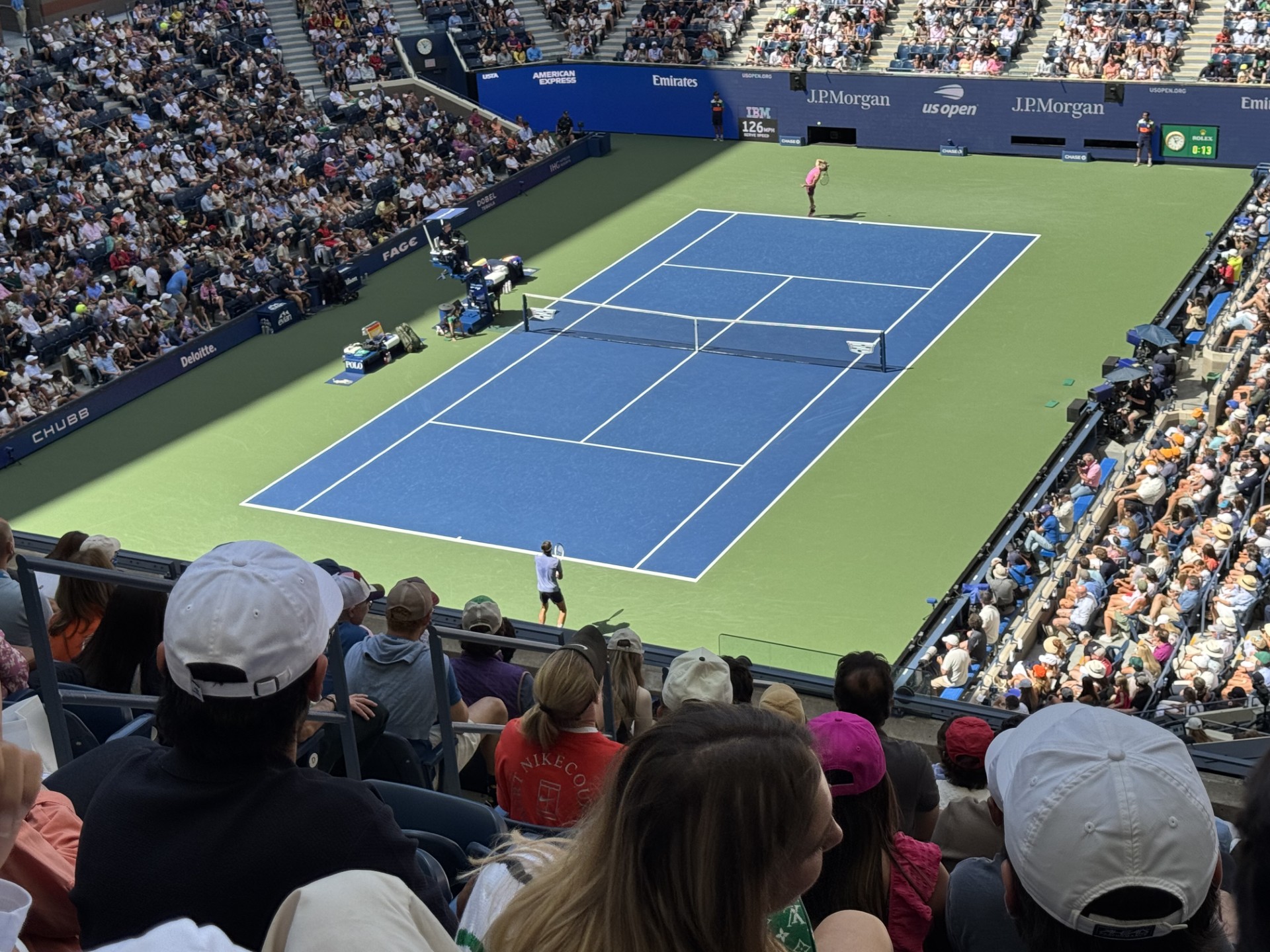}
\caption{Carlos Alcaraz of Spain serving against Arthur Rinderknech of France in a Round-of-16 match at the US Open in Arthur Ashe Stadium, Flushing, New York City, August 31, 2025.}
\label{fig: Alcaraz serve}
\end{center}
\end{figure}

The scoring system in tennis is unique. See the United States Tennis Association website 
{\url{https://www.usta.com/en/home.html}}
for tennis scoring rules. A match is usually either a best-of-three sets or a best-of-five sets. In the US Open, the men's matches are best-of-five sets, whereas the women's matches are best-of-three sets. The first to serve in the first set of a match is determined randomly (well, a coin is tossed, and the winner of the coin toss decides whether to serve first or to receive first). Each set consists of a sequence of games. A game consists of points and one player serves for the duration of the game. A game ends if one player reaches four points (with the scoring being 0 (Love), 15, 30, 40, game), provided that there is at least an advantage of two points. If both players reach three points each (a 40:40 score, called a {\em deuce}), a game tie-breaker ensues where the winner is the first to have an advantage of two points. The server for the game continues serving throughout, including the game tie-breaker (deuce). 
If the player serving in a game wins the game, we say that the server held serve, whereas if the serving player loses the game, we say that the server was broken.
The players alternate serving the games of a set. The set ends if a player wins six games, provided there is an advantage of at least two games. If the players win five games each, then the first to win seven games, while the other player remains at five wins, wins the set. If the players win six games each, then a set tie-breaker is played. In the set tie-breaker where the players play for points, the first to serve is the player who was the receiver in the last game played. Then the players alternate serves according to a $ABBAABBAA...$ or a $BAABBAABB...$ sequence (we refer to two players as players $A$ and $B$). The first player to reach $K$ points, where $K = 7$ if the set is not the last set, and $K=10$ if the set is the deciding set as in the US Open, wins the set tie-breaker, provided that there is an advantage of at least two points. If both players win $(K-1)$ points each, then a set tie-breaker's tie-breaker ensues, with the serving sequence still followed. The first player to achieve an advantage of 2 points in the set tie-breaker's tie-breaker wins the set tie-breaker, and hence the set. 
After the first set, in the succeeding sets, the player who received in the last game, provided there was no set tie-breaker, serves first in the set, but if there was a set tie-breaker, the player who received first in the set tie-breaker, serves first in the succeeding set.

The senior author, having taught probability courses to undergraduate and graduate students, from the most basic to the most advanced, over many years, recognized that a tennis match provides an excellent specimen to demonstrate instructive probability calculations, at least in an idealized setting. In a tennis match, there are several questions that could be asked, such as:
%
(i) How do we know that a tennis match will actually end with probability one? This is an interesting question to pose to students since their immediate response is either {\em obviously!} or {\em why even ask this question?}, since, of course, everyone has always seen a tennis match completed; but the instructor could ask the student to provide a {\em mathematical proof} that indeed it will end with probability one;
(ii) Is the tennis scoring system {\em fair} in the sense that if the players are of equal abilities, then each has a 50\% chance of winning, or is it the case that the player who serves first possesses an advantage? This is a non-trivial question. See, for example, Kingston \cite{Kingston1976}, which investigates this issue theoretically; and Magnus and Klaasen \cite{Magnus_Klaasen_1999}, which examines this issue based on empirical data, as well as George \cite{George1973}, which discusses optimal strategies in service in tennis;
(iii) How efficient is the tennis scoring system for determining the better player, and how should efficiency be measured? How does a tennis scoring system compare with other systems, such as a best-of-$K$-games system? Or, is a best-of-five-sets match more efficient than a best-of-three-sets match at identifying the better player? 
(iv) In a set tie-breaker, what is the impact of $K$ being odd (e.g., $K=7$), or $K$ being even (e.g., $K=10$), as is the case now with the US Open system?
(v) On average, how many sets will be played in the match; how many games will be played in a set; and how many points will be played in a game, a set, or in the match?
(vi) There are also many statistical questions that arise, though these will not be addressed in the current paper. For example, based on the results of a tennis match, what are the best estimates of the probabilities of the players winning their service points? Another question is, if one has a prior joint distribution on the service probabilities of the players, after the match, what is the posterior joint distribution of these service probabilities, and what is the posterior probability that the better player indeed won the match?
%

Answers to these questions are not immediately apparent, hence these questions offer situations where a student needs to learn how to decompose a problem into its constituent parts, which are easier to solve, and then to combine these intermediate solutions to answer the more complicated questions. This is what we do in this paper, which we believe will be an instructive pedagogical paper, both for undergraduate and graduate students. Aside from these pedagogical aspects, the results we present provide us with insights about scoring systems in tennis, and which could be extended to scoring systems in other sports or in competitions involving two players, two teams, two organizations, etc., or even in quality control settings.

For our tennis setting, as stated earlier, we refer to the two players as players $A$ and $B$. We denote by $p_A$ the probability that when $A$ serves the point (abbreviated $A$'s service point), he/she wins the point; while $p_B$ is the probability that when $B$ serves, he/she wins the point. We shall then let $q_A = 1 - p_A$ and $q_B = 1 - p_B$. Generically, we let $V$ be the Bernoulli random variable indicating that player $A$ wins his/her service point; while $W$ indicates that player $B$ wins his/her service point: $V\sim B(1,p_A)$ and $W \sim B(1,p_B)$, where $B(K,p)$ is the binomial distribution with $K$ trials and success probability $p$, so $B(1,p) = Ber(p)$ is a Bernoulli distribution with parameter $p$. $V_1, V_2, \ldots$ and $W_1, W_2, \ldots$ then represent the sequences of outcomes for points served by $A$ and for points served by $B$, respectively, with the assumption that $V_i$s and $W_j$s are stochastically independent random variables. We combine these with the assumption that the $V_i$s are identically distributed $B(1,p_A)$, i.e., $V_1,V_2, \ldots \stackrel{IID}{\sim} B(1,p_A)$, and $W_j$s are identically distributed $B(1,p_B)$ i.e., $W_1,W_2, \ldots \stackrel{IID}{\sim} B(1,p_B)$, though clearly these assumptions are mathematical idealizations and not realistic.
We then let the odds and the odds ratio (OR) be denoted, respectively, by
$$ODDS(p) = \frac{p}{1-p} = \frac{p}{q} \quad \mbox{and} \quad
OR(p_A,p_B) = \frac{p_A/q_A}{p_B/q_B}.$$

There have been papers dealing with scoring systems in tennis. Early papers are those by Kemeny and Snell \cite{Kemeny_Snell_1960} which modeled a single game of tennis as a Markov chain, and Hsi and Burych \cite{Hsi_Burych_1971} where probability models were introduced to analyze a two-player competition, in particular presenting the probability that one player wins a single set of classical tennis. Newton and Keller \cite{Newton_Keller_2005}, perhaps the closest to our paper, examined some of the questions we are also addressing, such as the probability of a player winning a game, a set, or the match, given the players' probabilities of winning a point on their serves, or whether there is an advantage to the first server in a set, hence in the match. The set tie-breaker they considered is a 13-point tie-breaker, whereas we consider the current system employed in the four Grand Slam tournaments where a 7-point set tie-breaker is used for the first 4 sets (in men's) and the first 2 sets (in women's), with the final set having a 10-point set tie-breaker. The formulas for the probabilities in \cite{Newton_Keller_2005} are expressed in recursive forms. Their paper, however, did not deal with the duration of the game, set, or match, which we address in our paper, nor did they make comparisons with other scoring systems or studied the efficiency of scoring systems. An interesting aspect that their paper addressed is the probability that a specific player will win a Grand Slam tournament with 128 players. An earlier paper by Carter and Crews \cite{Carter_Crews_1974} explored the effect of ``tie-breakers'' on the duration of a game, set, or match, and also examined the `first-server' effect. Pollard \cite{Pollard1983} also obtained expressions for the probability of a player winning a game, a set, or the match in tennis, and also examined distributional characteristics of the number of points played in a game, a set, or a match. 

This paper is organized as follows. Section \ref{sec: Game} will examine the game tie-breaker and then the game. The analysis will be focused on the probability that player $A$ will win, and also the characteristics of the number of points that need to be played. A particular result is that a game will certainly end, that is, will end with probability one. Section \ref{sec: Set} will then examine the set tie-breaker's tie-breaker with a result that this will end with probability one; then the set tie-breaker; and finally the set. Section \ref{sec: Match} will then examine the match itself. Section \ref{sec: efficiency and comparison} will be concerned with the efficiency of scoring systems, and comparisons will be performed with best-of-$K$ systems, both in terms of efficiencies, probabilities of player $A$ winning, and characteristics of the number of points that need to be played. Finally, section \ref{sec: Concluding Remarks} will provide concluding remarks, and it will end with a discussion of some statistical questions for future research. \BLUE{All computer programs used in the calculations were coded in the {\tt R} programming language \cite{R} and were designed and developed by the authors, instead of by Artificial Intelligence agents.} All plots were generated using either the {\tt plot}, {\tt matplot}, {\tt contour}, or {\tt filled.contour} objects in the {\tt R} environment.

\section{Game Tie-Breaker and Game}
\label{sec: Game}

The question of whether a tennis match will end with probability one is equivalent to asking whether both the game tie-breaker (GT) and the set tie-breaker (ST) will end with probability one. We first resolve the case of GT.   We shall denote by $\theta_{GT}^{END}(p)$ the probability that the GT will end when the player serving on a point has probability $p$ of winning the point. 

\subsection{Game Tie-Breaker}

\begin{proposition}
\label{prop: game tie-breaker ends}
    If the server has probability $p$ of winning the point, then the game tie-breaker will end with probability 1, that is, $\theta_{GT}^{END}(p) = 1$ for every $p \in [0,1].$
\end{proposition}

\begin{proof}
    It is clear that when $p \in \{0,1\}$, the game tie-breaker immediately ends after two points, though of course in this case, the game will never reach this tie-breaker.

 Assume therefore that $p \in (0,1)$ and that the player $A$ is serving, so the sequence of outcomes of the points played in the tie-breaker is $V_1, V_2, \ldots$ which are IID $B(1,p_A = p)$. Let $E_n$ be the event that the game tie-breaker ends after $2n$ points, with $n = 1,2,\ldots$. Thus, for $n = 1,2,\ldots,$
    \begin{eqnarray*}
    E_n & \equiv & E_n(V_1, V_2, V_3, V_4, \ldots) 
    =  \bigcap_{i=1}^{n-1} \left\{[ V_{2i-1}  = 1,V_{2i} = 0 ]  \bigcup [ V_{2i-1} = 0, V_{2i} = 1 ] \right\}  \\ && \bigcap \left\{ [ V_{2n-1} = 1, V_{2n} = 1 ] \bigcup [ V_{2n-1} = 0, V_{2n} = 0 ]\right\}.
    \end{eqnarray*}
    The $E_n$s are disjoint events, so by the Theorem of Total Probability (cf., Ross \cite{Ross1998}),
    $$\theta_{GT}^{END}(p) = \Pr\{\cup_{n=1}^\infty E_n\} = \sum_{n=1}^\infty \Pr\{E_n\}.$$
    Next, with $\stackrel{d}{=}$ denoting `equal-in-distribution', we invoke the distributional property that
    $$(V_1,V_2,V_3, V_4, V_5,V_6, \ldots) \stackrel{d}{=} (V_3, V_4, V_5,V_6, V_7,V_8, \ldots)$$
    to obtain the third equality below, yielding a recursion equation.
    \begin{eqnarray*}
        \theta_{GT}^{END}(p) & = & \Pr\{E_1\} + \Pr\{\cup_{n=2}^\infty E_n\} \\
         & =  & (p^2 + q^2) + \Pr\{[V_1=1,V_2=0] \cup [V_1=0,V_2=1]\} 
         \sum_{n=2}^\infty \Pr\{E_n(V_3,V_4,V_5,V_6,\ldots)\} \\
        & = & (p^2 + q^2) + (pq + qp) \sum_{n=1}^\infty \Pr\{E_n(V_1,V_2,V_3,V_4,\ldots)\} =
          (p^2 + q^2) + 2pq \theta_{GT}^{END}(p).
    \end{eqnarray*}
    Thus, we have
    $\theta_{GT}^{END}(p) = (p^2 + q^2) + 2pq \theta_{GT}^{END}(p),$
    which leads to $$\theta_{GT}^{END}(p) = \frac{p^2 + q^2}{1-2pq} = \frac{p^2+q^2}{(p+q)^2 - 2pq} = 1.$$
    For pedagogical purposes, the proof is easily visualized using a tree diagram with depth 2, with some edges or branches of depth 2 going back to the root node.
\end{proof}

Next, we compute the probability, denoted by $\theta_{GT}(p)$, that the serving player will win the game tie-breaker.

\begin{proposition}
\label{prop: game probability}
    If the server has probability $p$ of winning the point, then the probability that the server wins the game tie-breaker is
    $$\theta_{GT}(p) = \frac{p^2}{1-2pq} = \frac{p^2}{p^2 + q^2} = \frac{(ODDS(p))^2}{1 + (ODDS(p))^2} \quad
    \mbox{and} \quad
    \theta_{GT}(1-p) = 1 - \theta_{GT}(p).$$
\end{proposition}

\begin{proof}
    Similar to the proof of Proposition \ref{prop: game tie-breaker ends} with the $E_n$ defined as
    \begin{eqnarray*}
    E_n & \equiv & E_n(V_1, V_2, V_3, V_4, \ldots) 
    =  \bigcap_{i=1}^{n-1} \left\{[ V_{2i-1}  = 1,V_{2i} = 0 ]  \bigcup [ V_{2i-1} = 0, V_{2i} = 1 ] \right\} 
    \bigcap \left\{ V_{2n-1} = 1, V_{2n} = 1 \right\},
    \end{eqnarray*}
   so the $q^2$ term in the recursion is not present since $\Pr\{E_1\} = p^2.$
\end{proof}

\subsection{Game Probabilities}

Since the GT is certain to end, we could then calculate the probability that the serving player will win the game, which is the first player to reach at least 4 points, but with an advantage of at least 2 points.

\begin{theorem} \label{thm: prob wins game}
If the server has probability $p \in [0,1]$ of winning the point on his/her serve, then his/her probability of winning the game is
$$\theta_G(p) = p^4 + 4p^4 q + 10p^4 q^2 + 20p^3 q^3 \theta_{GT}(p) \quad
\mbox{and} \quad
\theta_G(1-p) = 1 - \theta_G(p),$$
with this last property being an $S$-shapedness property of the mapping $p \mapsto \theta_G(p)$.
\end{theorem}

\begin{proof}
For $n \in \{4,5,6\},$ define the event
$$E_n = \left\{\sum_{i=1}^{n-1} V_i = 3, V_n = 1\right\} \quad \mbox{and} \quad
T_6 = \left\{\sum_{i=1}^6 V_i = 3\right\}.$$
Then
$\left\{\mbox{Server Wins Game}\right\} = \left\{\cup_{n=4}^6 E_n \right\} \bigcup \left\{T_6 \cap \{\mbox{Server Wins GT}\}\right\}.$
Therefore,
\begin{eqnarray*}
\theta_G(p) & = & \sum_{n=4}^6 \Pr\{E_n|p\} + \Pr\{T_6|p\} \theta_{GT}(p)  
 =  \left[\sum_{n=4}^6 {{n-1} \choose 3} p^3 q^{(n-1)-3}\right] p + \left[{6 \choose 3} p^3 q^3\right] \theta_{GT}(p).
\end{eqnarray*}
Expanding this expression yields the expression in the statement of the theorem.   To prove the $S$-shapedness property, since $\theta_{GT}(1-p) = 1 - \theta_{GT}(p)$, we have that
    \begin{eqnarray*}
       \lefteqn{ \theta_G(1-p)  =  q^4 +4q^4p + 10q^4p^2 + 20q^3p^3\theta_{GT}(1-p) } \\
        & = & q^4 +4q^4p + 10q^4p^2 + 20q^3p^3 - 20q^3p^3 \theta_{GT}(p) \\
        & = & \left\{\left[q^4 +4q^4p + 10q^4p^2\right] + 
         \left[p^4 +4p^4q + 10p^4q^2\right] + 20q^3p^3\right\} - \theta_{G}(p) 
          =  1 - \theta_G(p)
    \end{eqnarray*}
    since $\left\{\left[q^4 +4q^4p + 10q^4p^2\right] + 
         \left[p^4 +4p^4q + 10p^4q^2\right] + 20q^3p^3\right\}$ is the total probability of a (truncated) system where either player reaches 4 points or they reach a tied-score of 3 points each, hence will be equal to 1, since for certain the game will end, either with a win for A, a win for B, or a tied game.
\end{proof}

Note that the expression for $\theta_G(p)$ was also given in equation (5) of Newton and Keller \cite{Newton_Keller_2005}. The top left plot of Figure \ref{fig: game probability and game number of points} is a plot of the probability of the server winning the game with respect to his/her probability of winning the point. Observe that when $p=0.5$, the game probability is also $0.5$, but as $p$ changes from .5 to 1, the graph is above the 45-degree line, whereas as $p$ changes from .5 to 0, the graph is below the 45-degree line. This shows that the game decision system provides an improvement over just playing one point. 

Observe in Figure \ref{fig: game probability and game number of points} that the graph of $\{(p,\theta_G(p)): p \in (0,1)\}$ is $S$-shaped, that is, there is a point $p_0$, which happens to be $p_0 = .5$, such that $\theta_{G}(p_0) = p_0$; on $p \in (p_0,1)$ we have $\theta_G(p) > p$; while on $p \in (0, p_0)$ we have $\theta_G(p) < p$. This {\em $S$-shapedness property} is not unique to tennis: it is a well-known result in information theory and in reliability theory. In the latter it is known that for any coherent system (see \cite{BarlowProschan1975}) with $n \ge 2$ components, with components acting independently and each component having reliability of $p$, the system reliability function has the $S$-shapedness property (see \cite{BarlowProschan1975}, Theorem 5.4). Thus, the game decision system in tennis could be viewed as a coherent system, which provides an additional conceptual interpretation of the $S$-shapedness observed in 
Figure \ref{fig: game probability and game number of points}.

\begin{figure}[h]
    \centering    \includegraphics[width=\mymultB\textwidth,height=\mymultB\textwidth]{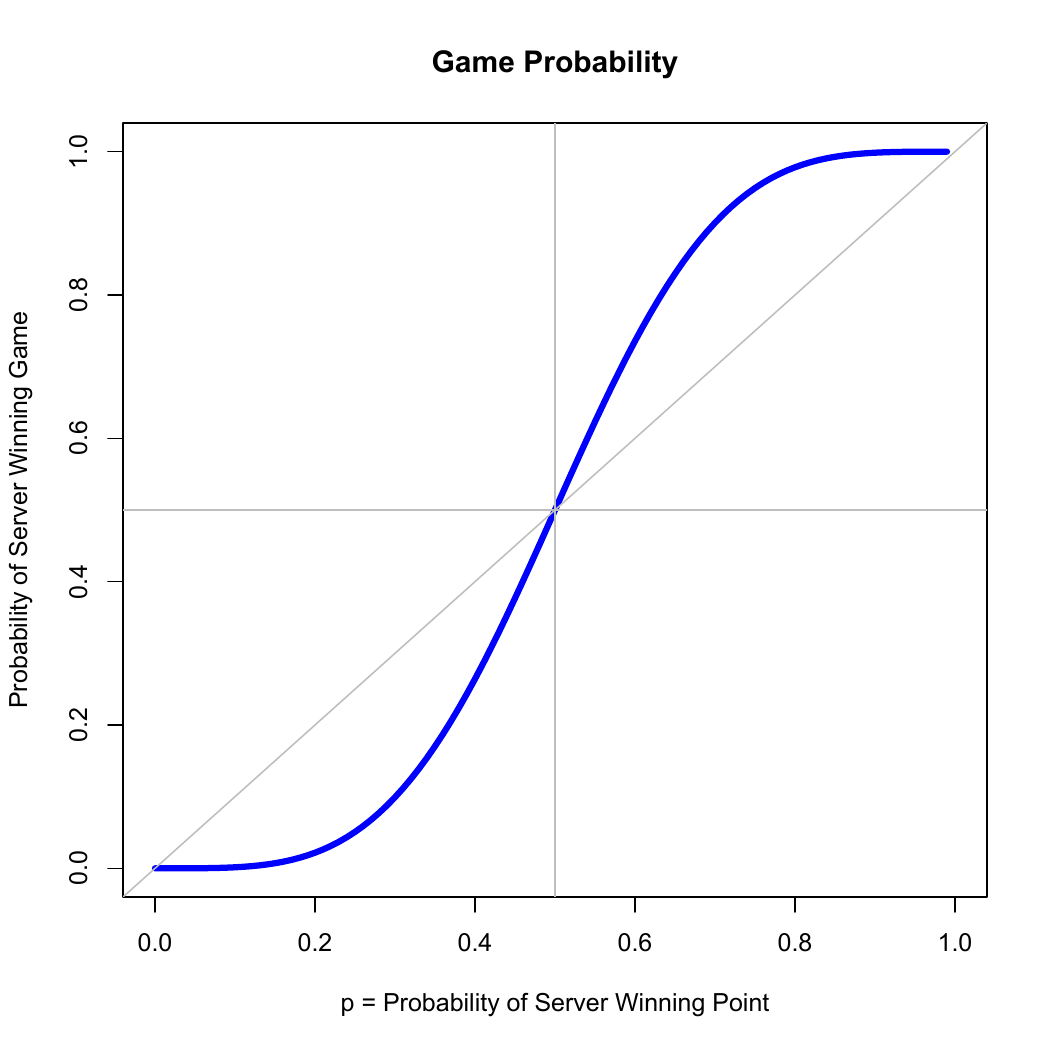} 
\includegraphics[width=\mymultB\textwidth,height=\mymultB\textwidth]{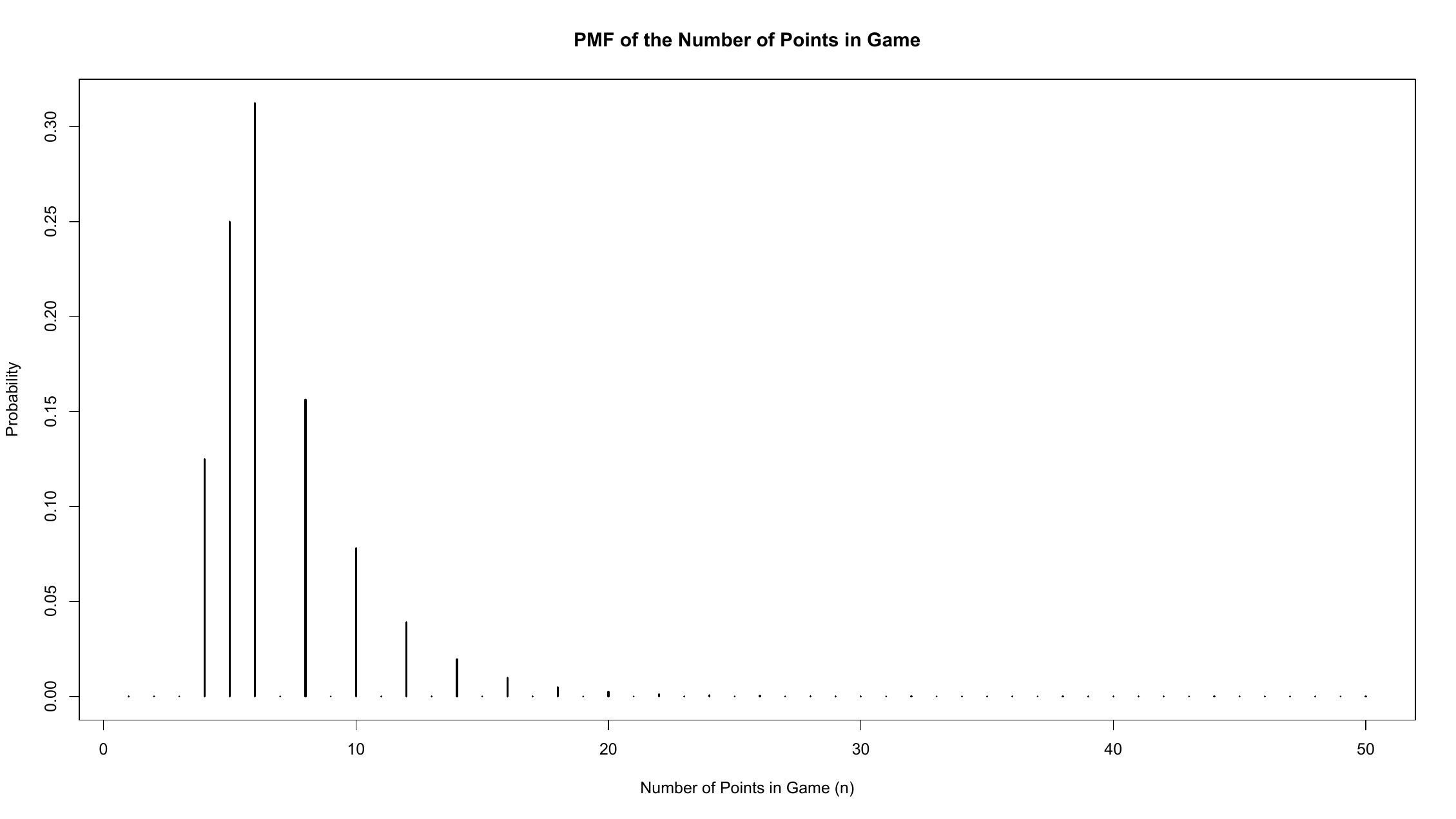} \\
\includegraphics[width=\mymultB\textwidth,height=\mymultB\textwidth]{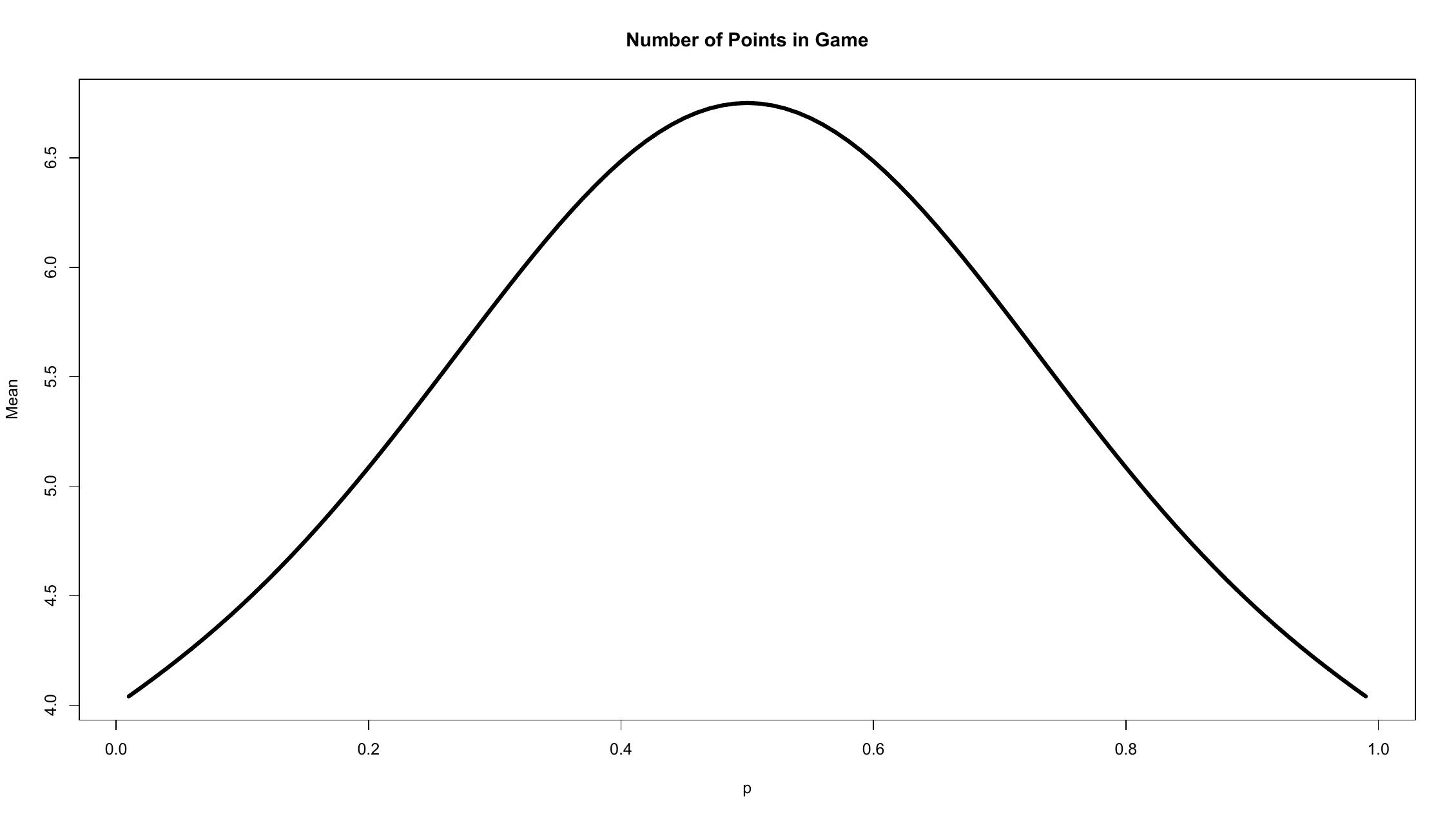} 
\includegraphics[width=\mymultB\textwidth,height=\mymultB\textwidth]{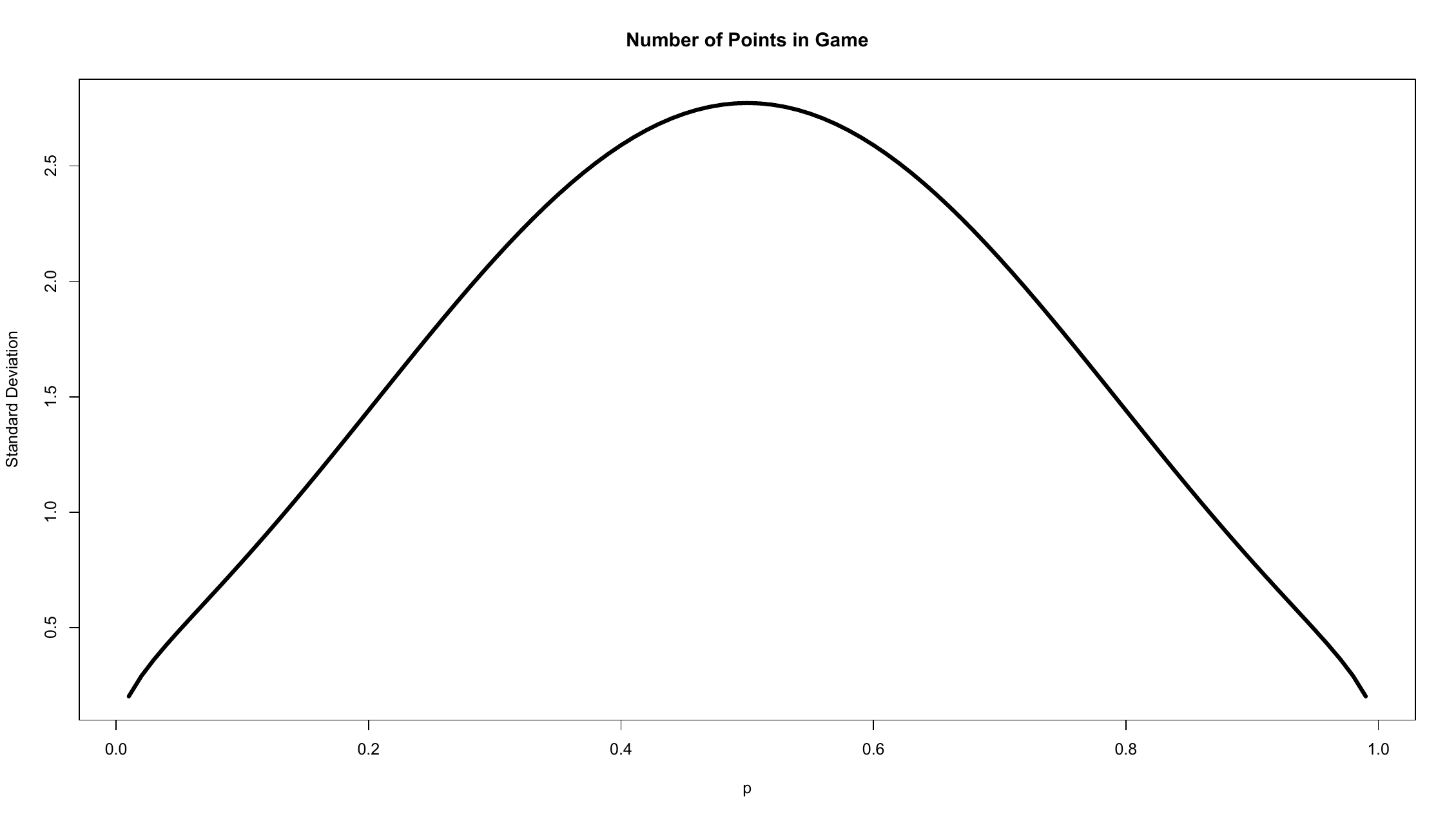}
    \caption{Characteristics of a Game. Top left plot is the probability of server winning the game with respect to server's probability of winning point; top right plot is the PMF of number of points played when $p=.5$. Bottom plots are the means and standard deviations of the number of points for $p \in (0,1)$.}
    \label{fig: game probability and game number of points}
\end{figure}

\subsection{Number of Points in a Game}

In this subsection, we examine the distribution of the number of points, $N_G$, needed to be played to end the game, for the tennis game system.  Denote by $p_{N_G}(n|p) = \Pr\{N_G = n|p\}$, where $p$ is the probability that the server will win the point.

\begin{theorem}
   If the server has probability $p \in [0,1]$ of winning the point on his/her serve, then the probability mass function (PMF) of the number of points $N_G$ in the game under the tennis game system is
    \begin{eqnarray*}
        p_{N_G}(n|p) =  \left\{
        \begin{array}{ccc}
        {{n-1} \choose 3} p^4 q^{n-4} + {{n-1} \choose 3} q^4 p^{n-4} & \mbox{for} & n = 4,5,6\\
        20p^3q^3(p^2+q^2)(2pq)^{\frac{n-6}{2}-1} & \mbox{for} & n = 8, 10, 12, \ldots \\
        0 & \mbox{otherwise} 
        \end{array}
        \right..
    \end{eqnarray*}
\end{theorem}

\begin{proof}
    For $n = 4, 5, 6$ points, the game could end with either player $A$ or player $B$ winning. The winning player must get his/her 4th point on the $n$th point played. If player $A$ is to win, this has probability of ${{n-1} \choose 3} p^4 q^{n-4}$, while if player $B$ is to win, this has probability ${{n-1} \choose 3} q^4 p^{n-4}$, hence the expression  in the theorem.

    For $n = 8, 10, 12, \ldots$, this means that the game reached the game tie-breaker, so that the score must have reached a 3-3 score (in tennis parlance, a 40-40 score, a deuce). The probability of this is $20p^3q^3$. Having reached deuce, the first player to go ahead by 2 points wins. To end at 8 points played, either player $A$ wins two points immediately, or player $B$ wins two points immediately, and this has probability $p^2+q^2$. Otherwise, it reaches deuce again with probability $2pq$. Thus, if the game has to end at $n=(6 + 2k)$ points, then the number of times that a deuce is reached after 6 points is $(k-1)$, and the probability of this event is $(2pq)^{k-1} (p^2+q^2)$. Thus the probability that the game ends after $6+2k$ is $20p^3q^3 (2pq)^{k-1} (p^2+q^2)$. 
\end{proof}

From this PMF, parameters of the distribution of $N_G$, such as the mean ($\mu_G$), variance ($\sigma_G^2$), and standard deviation ($\sigma_G$), as functions of $p$, could be obtained. Closed-form expressions for the mean and variance can be obtained, though the expressions may still contain many (but finite) terms. These expressions are presented in the following theorem.

\begin{theorem}
\label{theo: mean and variance for game number of points}
For the tennis game system, the mean and variance of the number of points played are
$$\mu_G(p) = \sum_{n=4}^6 n {{n-1} \choose 3} \left[p^4q^{n-4} + q^4p^{n-4}\right] +
    \left[6 + \frac{2}{p^2+q^2}\right] {6 \choose 3} p^3q^3;$$
%
%
\begin{eqnarray*}
        \sigma_G^2(p) 
        & = & \left[\left\{ \sum_{n=4}^6 n^2 {{n-1} \choose 3} \left[p^4q^{n-4} + q^4p^{n-4}\right] \right.\right. + \\ && \left.\left. \left[6 + \frac{2}{p^2+q^2}\right]^2 {6 \choose 3} p^3q^3\right\} - \mu_G^2(p)\right] + 4 \left[\frac{1-(p^2+q^2)}{(p^2+q^2)^2}\right] {6 \choose 3} p^3q^3.
    \end{eqnarray*}
Also, $\mu_G(p) = \mu_G(q)$ and $\sigma_G^2(p) = \sigma_G^2(q).$
\end{theorem}

\begin{proof}
    Let $SC = (h_A,h_B)$ represent the event that player $A$ (the server) scores $h_A$ points and player $B$ scores $h_B$ points. Then we could decompose $N_G$ via
    $$N_G = \sum_{n=4}^6 n I\{(SC=(4,n-4)) \cup (SC=(n-4,4))\} + (6 + N_{GT}) I\{SC=(3,3)\},$$
    where $N_{GT}$ is the number of points played in the GT. Observe that $N_{GT} \stackrel{d}{=} 2Ge(p^2+q^2)$, where $Ge(\eta)$ is a geometric random variable with success probability $\eta$, which has PMF $\Pr\{Ge(\eta) = n\} = (1-\eta)^{n-1} \eta, n=1,2,\ldots$. Furthermore, note that $SC$ and $N_{GT}$ are independent of each other. We now employ the Iterated Rules for Mean and Variance (cf., Ross \cite{Ross1998}), conditioning on $SC$, to obtain the mean and variance of $N_G$, and recalling that
    $E\{Ge(\eta)\} = {1}/{\eta}$ and $Var\{Ge(\eta)\} = {(1-\eta)}/{\eta^2}.$
    We have
     \begin{eqnarray*} E(N_G|SC) & = & \sum_{n=4}^6 n I\{(SC=(4,n-4)) \cup (SC=(n-4,4))\} +  \left[6 + \frac{2}{p^2+q^2}\right] I\{SC=(3,3)\}
    \end{eqnarray*}
   so that
    $$\mu_G(p) = \sum_{n=4}^6 n {{n-1} \choose 3} \left[p^4q^{n-4} + q^4p^{n-4}\right] +
    \left[6 + \frac{2}{p^2+q^2}\right] {6 \choose 3} p^3q^3.$$
    Also, we get 
    \begin{eqnarray*} 
    Var(N_G|SC) & = & 4\left[\frac{1-(p^2+q^2)}{(p^2+q^2)^2}\right] I\{SC=(3,3)\}.
    \end{eqnarray*}
    Consequently,
    \begin{eqnarray*}
        \sigma_G^2(p) & = & Var[E(N_G|SC)] + E[Var(N_G|SC)] \\
        & = & \left\{ \sum_{n=4}^6 n^2 {{n-1} \choose 3} \left[p^4q^{n-4} + q^4p^{n-4}\right] + \left[6 + \frac{2}{p^2+q^2}\right]^2 {6 \choose 3} p^3q^3 - \mu_G^2(p)\right\} + \\ && 4 \left[\frac{1-(p^2+q^2)}{(p^2+q^2)^2}\right] {6 \choose 3} p^3q^3.
    \end{eqnarray*}
    The last two identities in the statement of the theorem are immediately clear from the formulas of $\mu_G(p)$ and $\sigma_G^2(p)$.
\end{proof}

To illustrate, in Figure \ref{fig: game probability and game number of points} the top right plot is the PMF of $N_G$ when $p = .5$, while the bottom plots are the mean and standard deviation plots of $N_G$ as $p$ ranges over $(0,1)$. When $p = .5$, on average the game will take the longest duration, with the mean number of points being 6.750, variance of 7.6875, and the standard deviation of 2.7726. Later, in section \ref{sec: efficiency and comparison}, we will compare the performance of the tennis game system with Best-of-$K$ systems.

As a way to summarize some of the stochastic properties of a tennis game, we could have Table \ref{tab: game suimmary}, where, for $h \in \{0,1,2\}$,
\begin{eqnarray*}
   & \theta_G(4,h;p)  =  {{4+h-1} \choose 3} p^4 q^h; \
    \theta_G(h,4;p)  =  {{4+h-1} \choose 3} p^h q^4; & \\
   & \theta_G(h;p)  =  \theta_G(4,h;p) + \theta_G(h,4;p); & \\
   &\theta_G(GT_A;p) = {6 \choose 3} p^3 q^3 \theta_{GT}(p);\ \theta_G(GT_B;p) = {6 \choose 3} p^3 q^3 (1- \theta_{GT}(p)); & \\
   & \theta_{GT}(p) =  \frac{p^2}{p^2+q^2} = \frac{ODDS(p)^2}{1 + ODDS(p)^2}; & \\
   & \theta_G(GT;p) = \theta_G(GT_A;p) + \theta_G(GT_B;p); & \\
   & \theta_G(p) = \sum_{h=0}^2 \theta_G((4,h);p) + \theta_G(GT_A;p); & \\
    & 1 - \theta_G(p) = \sum_{h=0}^2 \theta_G((h,4);p) + \theta_G(GT_B;p); & \\
& \mu_G(h;p) = h + 4; \ \sigma_G^2(h;p) = 0, h = 0, 1, 2; & \\
   & \mu_G(GT;p) =  6 + \frac{2}{p^2+q^2}; \ \sigma_G^2(GT;p) = \frac{8pq}{(p^2+q^2)^2}; & \\
   & \mu_G(p) = \sum_{h=0}^2 \theta_G(h;p) \mu_G(h) + {6 \choose 3} p^3 q^3 \mu_G(GT;p); & \\
   & \sigma_G^2(p) = \left[ \sum_{h=0}^2 \theta_G(h;p) \sigma_G^2(h;p) + {6 \choose 3} p^3 q^3 \sigma_G^2(GT;p) \right] & \\ & + \left[ \left\{ \sum_{h=0}^2 \theta_G(h;p) \mu_G(h;p)^2 + {6 \choose 3} p^3 q^3 \mu_G(GT;p)^2 \right\}  - \mu_G(p)^2\right]; & \\
   & \sigma_G(p) = + \sqrt{\sigma_G^2(p)}. &
\end{eqnarray*}
\begin{table}[h]
\begin{center}
\begin{tabular}{||c||c|c||c||c|c||}\hline\hline
Loser's & \multicolumn{3}{c||}{Probability} & \multicolumn{2}{c||}{Number of Points} \\
 \cline{2-4} \cline{5-6}
Score, $h$ & $A$ wins & $B$ wins & $A$ or $B$ wins & CMean & CVariance \\ \hline\hline
0 & $\theta_G(4,0;p)$ & $\theta_G(0,4;p)$ & $\theta_G(0;p)$ & $\mu_G(0;p)$ & $\sigma_G^2(0;p)$ \\
1 & $\theta_G(4,1;p)$ & $\theta_G(1,4;p)$ & $\theta_G(1;p)$ & $\mu_G(1;p)$ & $\sigma_G^2(1;p)$ \\
2 & $\theta_G(4,2;p)$ & $\theta_G(2,4;p)$ & $\theta_G(2;p)$ & $\mu_G(2;p)$ & $\sigma_G^2(2;p)$ \\ \hline
TB & $\theta_{G}(GT_A;p)$ & $\theta_{G}(GT_B;p)$ & $\theta_G(GT;p)$ & $\mu_G(GT;p)$ & $\sigma_G^2(GT;p)$ \\ \hline\hline
Overall & $\theta_G(p)$ & $1-\theta_G(p)$ & 1 & $\mu_G(p)$ & $\sigma_G^2(p)$ \\ \hline\hline
\end{tabular}
\caption{Summary of stochastic characteristics for a Tennis Game, with the server's probability of winning point being $p$.}
\label{tab: game suimmary}
\end{center}
\end{table}

For a concrete example, suppose that the serving player (player $A$) has probability of $p=.6$ of winning the point. The game's stochastic characteristics are in Table \ref{tab: game stats p=.6}.

\begin{table}[h]
\begin{center}
\begin{tabular}{||c||c|c||c||c|c||}\hline\hline
Loser's & \multicolumn{3}{c||}{Probability} & \multicolumn{2}{c||}{Number of Points} \\ \cline{2-4} \cline{5-6}
Score, $h$ & $A$ wins & $B$ wins & $A$ or $B$ wins & CMean & CVariance \\ \hline\hline
       0 & 0.129 & 0.025 & 0.155 & 4.000 & 0.000 \\
       1 & 0.207 & 0.061 & 0.268 & 5.000 & 0.000 \\
       2 & 0.207 & 0.092 & 0.299 & 6.000 &  0.000 \\ \hline
      TB & 0.191 & 0.085 & 0.276 & 9.846  & 7.100 \\ \hline\hline
 Overall & 0.736 & 0.264 & 1.000 & 6.484 &  6.708 \\ \hline\hline
 \end{tabular}
 \caption{Stochastic characteristics for a game with the serving player (player $A$) having $p=.6$ probability of winning point.}
\label{tab: game stats p=.6}
 \end{center}
 \end{table}

\section{Set Tie-Breaker and Set}
\label{sec: Set}

When the score in a set reaches six games each for players $A$ and $B$, the set goes into a set tie-breaker (ST). STs hold a particular allure in tennis, especially after the classic fourth set ST in the 1980 Wimbledon Finals Match between the icy and cool Swedish Bjorn Borg versus the fiery and temperamental American John McEnroe which ended on a score of 18 points to 16 points in favor of McEnroe, with McEnroe saving five match points, and with this ST lasting for 20 minutes (cf., \cite{BorgMcEnroeRivalry}). Borg, however, won the fifth set, and hence the Wimbledon championship for that year. 

Serving sequence in an ST is either $ABBAABBAA...$ or $BAABBAABB...$, with the player starting the serve being the receiver in the last game played. For a specified $K$, which is usually $K \in \{7, 10\}$ (7 if the set is not the deciding set for the match, while 10 if it is the deciding set for the match), the first player to reach $K$ points and have an advantage of at least two points wins the tie-breaker. However, if the two players each win $(K-1)$ points, then the ST goes into a set tie-breaker's tie-breaker (STT). In the STT, the serving sequence continues, and the winner of this STT, hence of the ST and the set, is the first to get an advantage of two points.

\subsection{Set Tie-Breaker's Tie-Breaker}

First, observe that the STT is different from the GT, since in the latter the same player always serves, whereas in an STT, each player has an opportunity to serve according to the service sequence described above. As in the GT, there is the question of whether this additional tie-breaker will actually end with probability one, which, if this is so, will guarantee that the whole match will end with probability one, since we have already shown that GTs end with probability one.
We first resolve this question. We denote by $\theta_{STT}^{END}(p_A,p_B)$ the probability that this STT will end, when player $A$, who has probability of $p_A$ of winning the point on his/her serve, serves first, while player $B$, whose probability of winning the point on his/her serve is $p_B$, receives first.

\begin{proposition}
    If it is not the case that $(p_A=p_B=0)$ nor $(p_A = p_B = 1)$, then $\theta_{STT}^{END}(p_A,p_B) = 1.$ Furthermore, the result still holds true if the serving sequence is permuted for any adjacent odd/even pairs, e.g., $ABBA$ into $BABA$.
\end{proposition}

\begin{proof}
    The proof is pedagogically facilitated by looking at the tree diagram in Figure \ref{fig: STT diagram}, though this diagram could be reduced to depth 2 only, which leads to the desired recursion. We retained the depth 4 diagrams more for pedagogical purposes and since they coincide with the way the STT is played.
\begin{figure}[h]
    \centering
\includegraphics[height=.3\textwidth]{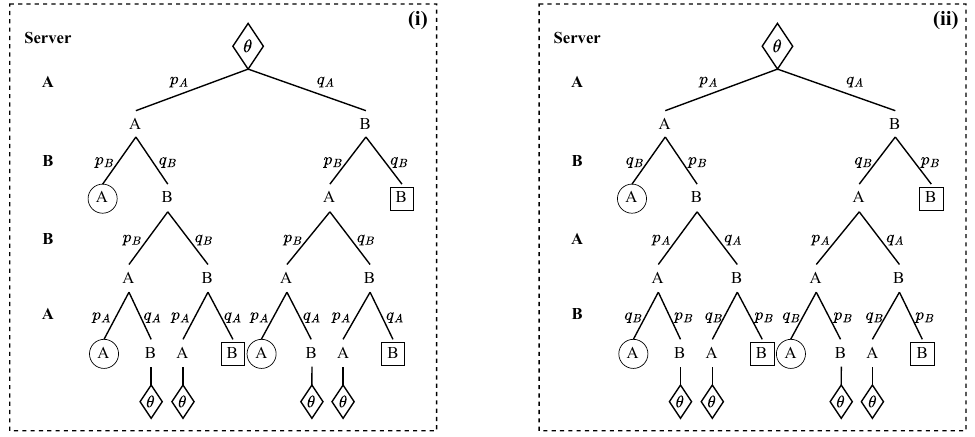}
    \caption{Tree diagram for the set tie-breaker tie-breaker. (i) If the serving sequence follows $ABBAABB \ldots$. (ii) If the serving sequence follows alternate serve ($ABABAB \ldots$).}
    \label{fig: STT diagram}
\end{figure}

To formalize the proof, let
$$Z \equiv (Z_1,Z_2,Z_3,Z_4),\ldots) = ((V_1,W_1),(W_2,V_2),(V_3,W_3),(W_4,V_4),\ldots)).$$
The random variables constituting $Z$ are independent, with $V_i$s IID $B(1,p_A)$ and $W_i$s IID $B(1,p_B)$. Define two mappings on $z = (v,w)$ given by
$$T(z = (v,w)) = T(z=(w,v)) = v+ (1-w) \quad \mbox{and} \quad \pi(z) \in \{(v,w),(w,v)\}.$$
That is, $T(z)$ is the number of wins for player $A$ in two successive odd/even points, while $\pi(z)$ permutates the elements of $z$. Observe the invariance property of these mappings given by $T(z) = T(\pi(z))$. Then define the mappings on the whole sequence $Z$ via
$$\Pi(Z) = (\pi_1(Z_1),\pi_2(Z_2),\ldots) \quad \mbox{and} \quad T(Z) = (T(Z_1),T(Z_2),\ldots).$$
Note that we are allowing different permutations for each of the components of $Z$. Then, again, observe the invariance property
$$T(\Pi(Z)) = T(Z).$$

Now, define the disjoint events
$$E_1(Z) = \{T(Z_1) \in \{0,2\}\};$$ 
$$E_n(Z) = \left[\cap_{l=1}^{n-1} E_l^c(Z)\right] \cap \{T(Z_n) \in \{0,2\}\}, n=2,3,\ldots,$$
so we also have
$$E_n(Z) = E_n(\Pi(Z)), n=1,2, \ldots.$$
Observe that
$$\Pr\{E_1(Z)|(p_A,p_B)\} = \Pr\{E_1(\Pi(Z))|(p_A,p_B)\} = (p_Aq_B + q_Ap_B);$$
$$\Pr\{(E_1(Z))^c|(p_A,p_B)\} = \Pr\{(E_1(\Pi(Z)))^c|(p_A,p_B)\} = (p_Ap_B + q_Aq_B).$$
Then
\begin{eqnarray*}
 \lefteqn{   \theta_{STT}^{END}(p_A,p_B)  =  \Pr\left\{ \cup_{n=1}^\infty E_n(Z) |(p_A,p_B)\right\} 
     =  \sum_{n=1}^\infty \Pr\{E_n(Z)|(p_A,p_B)\} } \\
    & = & \Pr\{E_1(Z)|(p_A,p_B)\} + \sum_{n=2}^\infty \Pr\{E_n(Z)|(p_A,p_B)\} \\
     & = & (p_Aq_B + q_Ap_B) + \Pr\{E_1^c(Z)|(p_A,p_B)\} 
     \sum_{n=2}^\infty \Pr\{E_n(Z)|E_1^c(Z),(p_A,p_B)\} \\
    & = & (p_Aq_B + q_Ap_B) + (p_Ap_B + q_Aq_B) 
    \sum_{n=1}^\infty \Pr\{E_n(S_1(Z))|(p_A,p_B)\} 
\end{eqnarray*}
where we have the shift operator
$S_1(Z) = S_1(Z_1,Z_2,Z_3,\ldots) = (Z_2,Z_3,Z_4,\ldots).$
Since
$T(Z) \stackrel{d}{=} T(S_1(Z)) \stackrel{d}{=} T(\Pi(Z)) \stackrel{d}{=} T(\Pi(S_1(Z))),$
then 
$\sum_{n=1}^\infty \Pr\{E_n(S_1(Z))|(p_A,p_B)\} = \theta_{STT}^{END}(p_A,p_B).$
We therefore obtain the equation (also seen from the diagram)
$$\theta_{STT}^{END}(p_A,p_B) = (p_Aq_B + q_Ap_B) + (p_Ap_B + q_Aq_B) \theta_{STT}^{END}(p_A,p_B)$$
leading to
$$\theta_{STT}^{END}(p_A,p_B) = \frac{p_Aq_B + q_Ap_B}{1 - (p_Ap_B + q_Aq_B)} = 1,$$
provided that neither $(p_A=1,p_B=1)$ nor $(p_A=0,p_B=0)$ is true.
\OLDPROOF{From this diagram, we obtain the equation
\begin{eqnarray*}
\lefteqn{ \theta_{STT}^{END}(p_A,p_B)  = } \\ 
&& p_Aq_B + q_Ap_B + \\ && p_Ap_Bq_Bp_A + p_Ap_Bp_Bq_A + q_Aq_Bq_Bp_A + q_Aq_Bp_Bq_A + \\
&& (p_Ap_Bq_Bq_A + p_Ap_Bp_Bp_A + q_Aq_Bq_Bq_A + q_Aq_Bp_Bp_A) \theta_{STT}^{END}(p_A,p_B) \\
& = & (p_Aq_B + p_Bq_A)(1 + p_Ap_B + q_Aq_B) + \\
&& (p_Ap_B +q_Aq_B)(q_Bq_A + p_Bp_A) \theta_{STT}^{END}(p_A,p_B).
\end{eqnarray*}
If neither $(p_A=p_B=0)$ nor $(p_A=p_B=1)$ is true, then $(p_Ap_B +q_Aq_B)(q_Bq_A + p_Bp_A) \ne 1$, so we could solve for $\theta_{STT}^{END}(p_A,p_B)$ which yields
\begin{eqnarray*}
\theta_{STT}^{END}(p_A,p_B) & = & \frac{(p_Aq_B + p_Bq_A)(1 + p_Ap_B + q_Aq_B)}{1 - (p_Ap_B +q_Aq_B)(q_Bq_A + p_Bp_A)} \\
& = & \frac{(p_Aq_B + p_Bq_A)(1 + p_Ap_B + q_Aq_B)}{[1 - (p_Ap_B +q_Aq_B)][1 + (p_Ap_B +q_Aq_B)]} \\
& = & \frac{(p_Aq_B + p_Bq_A)}{[1 - (p_Ap_B +q_Aq_B)]} \\
& = & \frac{(p_Aq_B + p_Bq_A)}{[(p_A+q_A)(p_B+q_B) - (p_Ap_B +q_Aq_B)]} \\
& = & \frac{(p_Aq_B + p_Bq_A)}{(p_Aq_B + p_Bq_A)} \\
& = & 1.
\end{eqnarray*}
}
Since the distributional results remain invariant under permutations of successive odd/even pairs under the mapping $\Pi$, the second claim in the statement of the proposition holds true.
Note that if $(p_A=p_B=0)$, then each player {\em always} loses the point when serving; while if $(p_A=p_B=1)$, then each player {\em always} wins the point when serving; so in either case, neither player will be able to achieve an advantage of two points, hence the STT never ends.
\end{proof}

Next, we calculate the probability that player $A$, when serving first, wins the STT, denoted by $\theta_{STT}(p_A,p_B)$.  

\begin{proposition} \label{A wins STT}
If A serves first in the STT, and assuming that neither $(p_A=p_B=0)$ nor $(p_A=p_B=1)$ is true, then the probability of A winning the STT is
$$\theta_{STT}(p_A,p_B) = \frac{OR(p_A,p_B)}{1 + OR(p_A,p_B)}
\quad \mbox{and} \quad
\theta_{STT}(p_A,p_B) =  \theta_{STT}(q_B,q_A).$$
Furthermore, the result is invariant under permutation of serving order for every successive odd/even pair.
\end{proposition}

\begin{proof}
    The proof of this result is very similar to the proof of the preceding proposition. We just need to re-define the events to be
    $$E_1^\prime(Z) =  \{T(Z_1) = 2\};$$
    $$E_n^\prime(Z) = \left[\cap_{l=1}^{n-1} E_l^c(Z)\right] \cap \{T(Z_n) = 2\}, n=2,3,\ldots,$$
    where the events $E_l(Z), l=1,2,\ldots$ are as in the proof of the preceding proposition. Again, note that
    $$E_n^\prime(Z) = E_n^\prime(\Pi(Z)), n=1,2,\ldots.$$
    Following the same steps as in the preceding proof, we obtain the equation (see also Figure \ref{fig: STT diagram})
    $$\theta_{STT}(p_A,p_B) = p_Aq_B + (p_Ap_B + q_Aq_B) \theta_{STT}(p_A,p_B)$$
    leading to
    \begin{eqnarray*}
    \theta_{STT}(p_A,p_B) & = & \frac{p_Aq_B}{1 - (p_Ap_B + q_Aq_B)} 
     =  \frac{p_Aq_B}{p_Aq_B + q_Ap_B} 
     =  \frac{OR(p_A,p_B)}{1+OR(p_A,p_B)}.
    \end{eqnarray*}

\OLDPROOF{It amounts to dropping the terms associated with player $B$ winning the tie-breaker's tie-breaker (see Figure \ref{fig: STT diagram} where the $B$ is enclosed by a square), so dropping the terms $q_Ap_B$, $p_Ap_Bp_Bq_A$, and $q_Aq_Bp_Bq_A$. Doing so, then computing $\theta_{STT}(p_A,p_B)$, we obtain
    \begin{eqnarray*}
    \theta_{STT}(p_A,p_B) & = & \frac{(p_Aq_B)(1 + p_Ap_B + q_Aq_B)}{1 - (p_Ap_B +q_Aq_B)(q_Bq_A + p_Bp_A)} \\
& = & \frac{(p_Aq_B)(1 + p_Ap_B + q_Aq_B)}{[1 - (p_Ap_B +q_Aq_B)][1 + (p_Ap_B +q_Aq_B)]} \\
& = & \frac{p_Aq_B}{p_Aq_B + q_Ap_B} 
 =  \frac{OR(p_A,p_B)}{1 + OR(p_A,p_B)}.
\end{eqnarray*}
}
The identity $\theta_{STT}(p_A,p_B) = \theta_{STT}(q_B,q_A)$ is immediate from the derived equation, or this could also be seen from the fact that $$T(Z_n) = (V_n + (1-W_n))|(p_A,p_B) \stackrel{d}{=} (V_n + (1-W_n))|(q_B,q_A), n=1,2,\ldots.$$ The serving order invariance follows from $E_n^\prime(Z) = E_n^\prime(\Pi(Z)), n=1,2,\ldots.$
\end{proof}

\begin{corollary}
If $p_A = p_B = p \in (0,1)$, so the players are of equal abilities, then $\theta_{STT}(p,p) = 1/2,$ for every $p \in (0,1)$, that is, they are equally likely to win the STT. Thus, when of equal abilities, but not perfect on their serves, there is no advantage to whether the player serves or receives first.
\end{corollary}

\OLDPROOF{
\begin{proposition} [Serving order invariance]
    If the serving sequence in the set tie-breaker’s tie-breaker follows $ABABABAB \ldots$, that is, the players simply alternate serve, and assuming that neither $(p_A=p_B=0)$ nor $(p_A=p_B=0)$ is true, then the probability of A winning the set tie-breaker's tie-breaker remains the same as under the original serving order. In particular,
    $$\theta_{STT}(p_A,p_B) = \frac{OR(p_A,p_B)}{1 + OR(p_A,p_B)}.$$
    Thus, the serving order does not affect the probability that A wins.
\end{proposition}

\begin{proof}
    The proof mirrors that of Proposition~\ref{A wins STT}. Based on the tree diagram in Figure~\ref{fig: STT diagram} (ii), we again obtain 
    $$\theta_{STT}(p_A,p_B) = \frac{(p_Aq_B)(1 + p_Ap_B + q_Aq_B)}{1 - (p_Ap_B +q_Aq_B)(q_Bq_A + p_Bp_A)}.$$

    The algebraic simplification proceeds identically and leads to the same expression
    $$\theta_{STT}(p_A,p_B) = \frac{OR(p_A,p_B)}{1 + OR(p_A,p_B)}.$$
    
    Therefore, alternate serving does not change the outcome probabilities.
\end{proof}
}

\subsection{Set Tie-Breaker}

We now consider the set tie-breaker (ST) with parameters $(p_A,p_B,K)$, where the player to reach at least $K$ points wins the ST, provided that he/she is ahead by at least 2 points, and the player serving first has probability $p_A$ of winning the point. If the score is tied at $(K-1)$ points each, then the STT ensues.

Define the following functions over the positive integers:
$$s_A(n) = \mbox{number of times that player A serves in the first $n$ points};$$
$$s_B(n) = n - s_A(n) = \mbox{number of times that player B serves in the first $n$ points};$$
and
$$I_A(n) = \mbox{indicator that player A is serving on the $n$th point}.$$
The following results concerning the number of service points and who is serving are immediate based on the $ABBAABBAA...$ serving sequence.

\begin{lemma}
    If player A serves first in the ST, then
    $$s_A(n) = \sum_{j=1}^n I\{(j \bmod 4) \in \{0,1\}\}; \quad s_B(n) = n - s_A(n); \quad 
    I_A(n) = I\{(n \bmod 4) \in \{0,1\}\}.$$
\end{lemma}

We are now in position to determine the probability that the first to serve in a $K$-point ST will win the ST. We denote this probability by $\theta_{ST}(p_A,p_B,K)$.

\begin{proposition}
\label{prop: set tie-breaker probs}
    If player A serves first in the set tie-breaker, the probability that he/she wins the set tie-breaker is
    \begin{eqnarray*}
    \theta_{ST}(p_A,p_B,K)  & = &  \sum_{h=0}^{K-2} \theta_{ST}(K,h;p_A,p_B,K) + 
     \theta_{ST}(K-1,K-1;p_A,p_B,K) \times \\ &&
     [\theta_{STT}(p_A,p_B)]^{I_A(2K-1)} [1-\theta_{STT}(p_B,p_A)]^{1-I_A(2K-1)},
    \end{eqnarray*}
    where
    \begin{eqnarray*}
    \theta_{ST}(K,h;p_A,p_B,K) & \equiv & \Pr\{\mbox{$A$ wins on a $(K,h)$ score} | p_A,p_B\}   \\ & = & \left\{
        \sum_{j=0}^{K-1} {s_A(K+h-1) \choose j} {s_B(K+h-1) \choose {(K-1) - j}} \times \right. \\
        && \left.
        \left[p_A^j q_A^{s_A(K+h-1) - j}\right]
        \left[q_B^{(K-1)-j} p_B^{s_B(K+h-1) - ((K-1)-j)}\right]
        \right\}  \left[p_A^{I_A(K+h)} q_B^{1-I_A(K+h)}\right]; \\
    %
   \theta_{ST}(K-1,K-1;p_A,p_B,K) & \equiv &  \Pr\{(K-1,K-1)\ \mbox{score} | p_A,p_B\} \\
      & = & \left\{
        \sum_{j=0}^{K-1} {{K-1} \choose j} {{K-1} \choose {j}}
        \left[p_A^j q_A^{(K-1) - j}\right]
        \left[q_B^{(K-1)-j} p_B^{j}\right]
        \right\}.
    \end{eqnarray*}
    Alternatively, with $B(K-1,p_A)$ and $B(K-1,q_B)$ independent binomial random variables,
    \begin{eqnarray*}
        \sum_{h=0}^{K-2} \theta_{ST}(K,h;p_A,p_B,K) =
        \Pr\{B(K-1,p_A) + B(K-1,q_B) \ge K\}.
    \end{eqnarray*}
In addition,
    $\theta_{ST}(p_A,p_B,K) = \theta_{ST}(q_B,q_A,K).$
\end{proposition}

\begin{proof}


For $h \in \{0,1,\ldots,K-2\}$, define the event
$$F_h = \{\mbox{$A$ wins with a score of $(K,h)$}\} \quad \mbox{and} \quad 
T_{2(K-1)} = \{\mbox{score tied at $(K-1:K-1)$\}.}$$
Then
\begin{eqnarray*}
\theta_{ST}(p_A,p_B,K) & = & \sum_{h=0}^{K-2} \Pr\{F_h|(p_A,p_B,K)\} + 
\Pr\{T_{2(K-1)}|(p_A,p_B,K)\} \Pr\{\mbox{$A$ wins STT}|(p_A,p_B,K)\}.
\end{eqnarray*}
For $A$ to win with a score of $(K,h)$, he/she must win $(K-1)$ points in the first $(K+h-1)$ points played, for which he/she will be serving a total of $s_A(K+h-1)$ of these points, and then he/she must also win the $(K+h)$th point played. That is,
\begin{eqnarray*}
F_h & = & \left\{\sum_{i=1}^{s_A(K+h-1)} V_i + \sum_{j=1}^{s_B(K+h-1)} (1 - W_j) = K-1\right\} \cap \\ &&
\left\{(I_A(K+h) = 1, V_{s_A(K+h-1)+1} = 1) \cup 
(I_A(K+h) = 0, 1-W_{s_B(K+h-1)+1} = 1)\right\}.
\end{eqnarray*}
The probability of this event is
\begin{eqnarray*} 
\lefteqn{\Pr\{F_h|(p_A,p_B,K)\}  \equiv \theta_{ST}(K,h;p_A,p_B,K) } \\
& = & \Pr\{B(s_A(K+h-1),p_A) + B(s_B(K+h-1),q_B) = K-1\} p_A^{I_A(K+h)} q_B^{1-I_A(K+h)},
\end{eqnarray*}
where the binomial random variables $B(s_A(K+h-1),p_A)$ and $B(s_B(K+h-1),q_B)$ are independent \BLUE{(this independence will be an implicit assumption in the sequel)}. Thus,
\begin{eqnarray*}
\lefteqn{ \theta_{ST}(K,h;p_A,p_B,K)  =   
\left[\sum_{l=0}^{K-1}
{{s_A(K+h-1)} \choose l} p_A^l q_A^{{s_A(K+h-1)}-l} \times \right. } \\ && \left. {{s_B(K+h-1)} \choose {K-1-l}} q_B^{K-1-l} p_B^{s_B(K+h-1)-(K-1-l)}\right]  p_A^{I_A(K+h)} q_B^{1-I_A(K+h)}.
\end{eqnarray*}
Similarly, since $s_A(2(K-1)) = s_B(2(K-1)) = K-1$, we obtain
\begin{eqnarray*}
   \Pr\{T_{2(K-1)}|(p_A,p_B,K)\} & \equiv & \theta_{ST}(K-1,K-1;p_A,p_B,K) \\  & = & \Pr\{B(K-1,p_A) + B(K-1,q_B) = K-1\}  \\
    & = & \left[\sum_{l=0}^{K-1}
{{K-1} \choose l} p_A^l q_A^{{K-1}-l} {{K-1} \choose {l}} q_B^{K-1-l} p_B^{l}\right].
\end{eqnarray*}

Finally, we note that if the ST reaches an STT, then we need to determine who serves first in this STT, which is indicated by $I_A(2K-1)$. Thus, we have 
\begin{eqnarray*}
\Pr\{\mbox{$A$ wins STT}|(p_A,p_B,K)\} & = & [\theta_{STT}(p_A,p_B)]^{I_A(2K-1)} [1-\theta_{STT}(p_B,p_A)]^{1 - I_A(2K-1)} \\ 
& = & [\theta_{STT}(p_A,p_B)]^{I_A(2K-1)} [1-\theta_{STT}(\RED{q_A,q_B)}]^{1 - I_A(2K-1)}.   
\end{eqnarray*}

To establish $\theta_{ST}(p_A,p_B,K) = \theta_{ST}(q_B,q_A,K)$, it suffices to show that
$$\sum_{h=0}^{K-2} \Pr\{F_h|(p_A,p_B,K)\} = 
\sum_{h=0}^{K-2} \Pr\{F_h|(q_B,q_A,K)\}.$$
We establish this using an alternative way of looking at how a player could win based on looking at the complete paths of length $2(K-1)$ points.

Consider all paths or would-be paths of length $2(K-1)$. Note that some of these paths will never reach $2(K-1)$ points since it is possible that one player had already won $K$ points. However, think of extending the path until $2(K-1)$ points. Over these $2(K-1)$ points, each player will serve a total of $(K-1)$ times. 
Player $A$ then wins if
$\sum_{i=1}^{K-1} V_i + \sum_{j=1}^{K-1} (1-W_j) \ge K$
while player $B$ wins if
$\sum_{i=1}^{K-1} (1 - V_i) + \sum_{j=1}^{K-1} W_j \ge K.$
The other paths will lead to a tied score with each player having $(K-1)$ points. Under $(p_A,p_B)$ the probability of $A$ winning, without going to an STT, is
$\Pr\{B(K-1,p_A) + B(K-1,q_B) \ge K\},$
which is the alternative expression provided in the statement of the proposition for $\sum_{h=0}^{K-2} \Pr\{F_h|(p_A,p_B,K)\}$.
But note that this is also the probability of $A$ winning, without going to an STT, 
{\em under} $(q_B,q_A)$. As such, we have established, via an alternative viewpoint, that $
\sum_{h=0}^{K-2} \Pr\{F_h|(p_A,p_B,K)\} = 
\sum_{h=0}^{K-2} \Pr\{F_h|(q_B,q_A,K)\}.$
We note that showing {\em directly} the equality above using the formula for $\sum_{h=0}^{K-2} \Pr\{F_h|(p_A,p_B,K)\}$ seems not a straightforward matter.

\end{proof}

Figure \ref{fig: contour set tie breaker} presents a contour plot of $\theta_{ST}$ as $p_A$ and $p_B$ range over $0.01$ to $0.99$ and when $K = 7$. The probabilities were computed using an {\tt R \cite{R}} program we developed implementing the formula in the preceding proposition. 
\begin{figure}[h]
    \centering  \includegraphics[width=\mymultB\textwidth,height=\mymultB\textwidth]{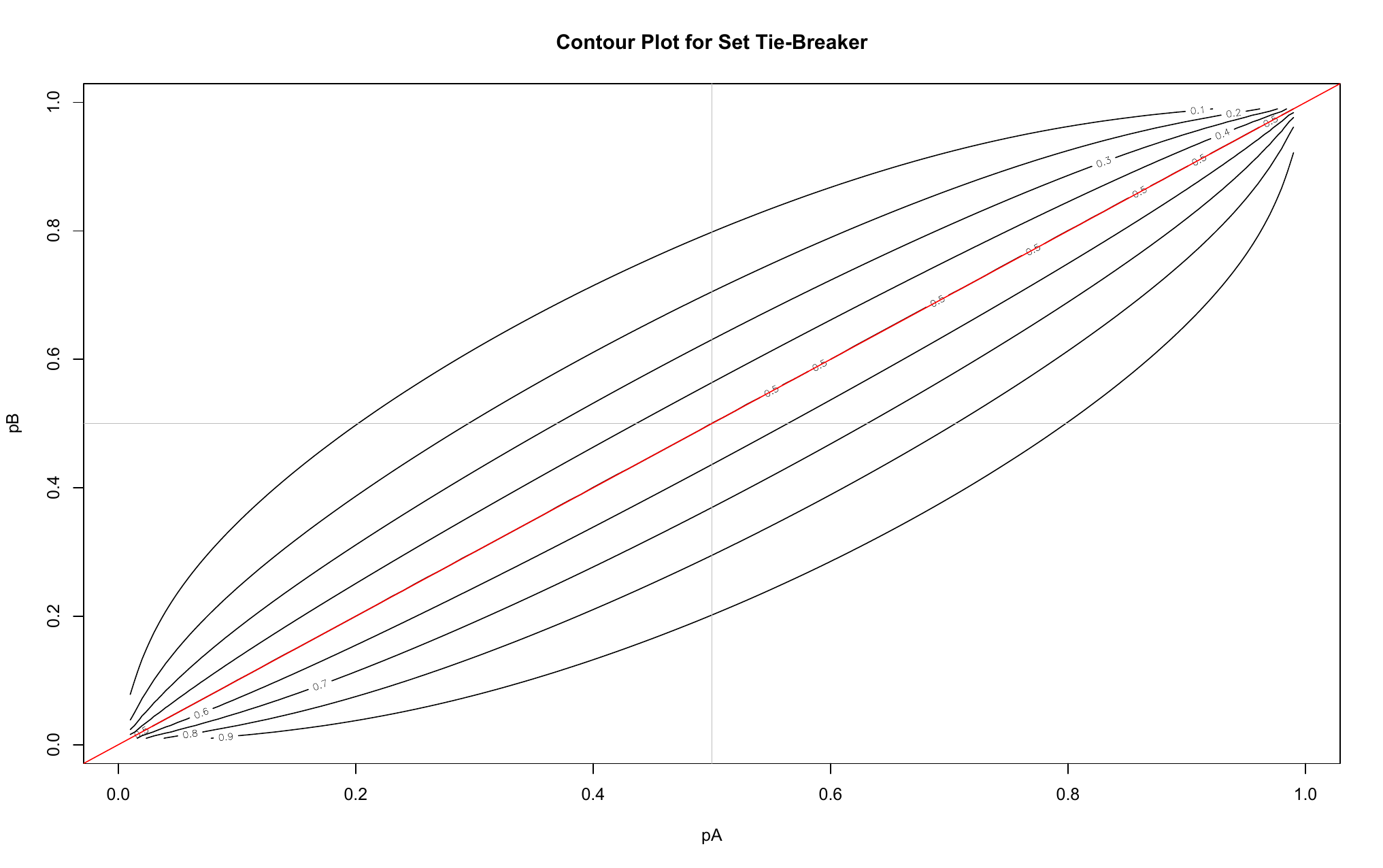}
    \caption{Contour plot of the probability of the first server winning the set tie-breaker when $K = 7$. $pA$ is the probability of first server winning point on his/her serve; while $pB$ is the probability of the first receiver winning point on his/her serve.}
    \label{fig: contour set tie breaker}
\end{figure}
Observe that when $p_A = p_B$, the probability of the server winning the set tie-breaker is exactly $1/2$, indicating that the set tie-breaker does not endow an advantage to the first server, in essence, the set tie-breaker is a fair system. This result is formally stated and proved in the following corollary.

\begin{corollary}
\label{coro: under pA=pB for theta_ST}
If $p_A = p_B = p$, then $\theta_{ST}(p,p,K) = 1/2$ for every $p \in (0,1)$ and $K \ge 2$.
\end{corollary}

\begin{proof}
    Using the notation in the proof of Proposition \ref{prop: set tie-breaker probs} but with $V_i$s IID $Ber(p)$ and $W_j$s IID $Ber(p)$, we have that
    \begin{eqnarray*}
     \theta_{ST}(p,p,K) & = & \Pr\left\{\sum_{i=1}^{K-1} V_i + \sum_{j=1}^{K-1} (1-W_j) \ge K\right\} +  \Pr\left\{\sum_{i=1}^{K-1} V_i + \sum_{j=1}^{K-1} (1-W_j) = K-1\right\} \theta_{STT}(p,p)
     \end{eqnarray*}
    where $\theta_{STT}(p,p) = 1/2$. The probability that $B$ wins the tie-breaker, which will be $1 - \theta_{ST}(p,p,K)$ is given by
    \begin{eqnarray*}
   1 - \theta_{ST}(p,p,K) & = &  \Pr\left\{\sum_{i=1}^{K-1} (1-V_i) +   \sum_{j=1}^{K-1} W_j \ge K\right\} +   \Pr\left\{\sum_{i=1}^{K-1} (1-V_i) + \sum_{j=1}^{K-1} W_j = K-1\right\} \theta_{STT}(p,p)  \\
     & = & \Pr\left\{ \sum_{j=1}^{K-1} W_j + \sum_{i=1}^{K-1} (1-V_i) \ge K\right\} + \Pr\left\{\sum_{j=1}^{K-1} W_j + \sum_{i=1}^{K-1} (1-V_i) = K-1\right\} \theta_{STT}(p,p) \\ & = & \theta_{ST}(p,p,K).
     \end{eqnarray*}
Thus, it follows that $\theta_{ST}(p,p,K) = 1/2$.
\end{proof}

\subsection{Number of Points in a Set Tie-Breaker}

Let $N_{ST}$ be the number of points played in a set tie-breaker with parameter $(p_A,p_B,K)$. We seek the PMF of $N_{ST}$, that is, 
$$p_{N_{ST}}(n|p_A,p_B,K) = \Pr\{N_{ST} = n|p_A,p_B,K\},$$
as well as some of its basic characteristics such as its mean ($\mu_{ST}$), variance ($\sigma_{ST}^2$), and standard deviation ($\sigma_{ST}$). This will further shed some probabilistic light on the classic Borg-McEnroe tie-breaker, considered one of the best set tie-breaker ever played in a Grand Slam tennis tournament, as well as other set tie-breakers that were lengthy. From {\em Wikipedia} \cite{wiki:longest}, the longest set tie-breakers that were played in Grand Slam Tournaments (US Open, Wimbledon in the United Kingdom, French Open, and Australian Open) are enumerated below, with the tie-breaker score in {\bf bold font}. It should be noted that in all four Grand Slam tournaments, starting in March 2022 (effectively in May 2022 during the French Open) \cite{ESPN2022}, the final set (fifth set for men; third set for women) tie-breakers became 10-point tie-breakers. Prior to March 2022, these Grand Slam tournaments have different final set tie-break systems. From the article \cite{ESPN2022}, ``the Australian Open used the 10-point tiebreak, the US Open employed a traditional, seven-point tiebreak even in the final set. Wimbledon featured a seven-point tiebreak when the score reached 12-12 in the final set of matches at the All England Lawn Tennis Club. The French Open has been the only major without a final-set tiebreak, with matches continuing until a player secures a two-game lead in the decider.''

\medskip

\noindent
{\bf Wimbledon:} Bj{o}rn Borg defeated Premjit Lall 6–3, 6–4, 9–8 {\bf (20-18)} in the first round of the men's singles in 1973. John McEnroe and Borg took over 20 minutes to complete a fourth set tiebreaker in the 1980 men's singles final, with McEnroe winning {18–16}, but Borg eventually won the match 1–6, 7–5, 6–3, 6–7 {\bf (16-18)}, 8–6. John Isner won {19–17} in a first set tiebreaker against Jarkko Nieminen in the second round in 2014, eventually winning the match 7–6 {\bf (19-17)}, 7–6 (7–3), 7–5. Alexander Zverev won 17—15 in a third set tiebreaker against Cameron Norrie in the third round in 2024, winning the match 6—4, 6—4, 7—6 {\bf (17-15)}.

\smallskip

\noindent{\bf US Open:} Goran Ivanišević and Daniel Nestor played a 20–18 tiebreaker in their 1993 US Open first round match, won by Ivanišević 6–4, 7–6(7–5), 7–6 {\bf (20-18)}. Ken Flach won the deciding tie-break {17-15} over Darren Cahill during their second round match at the 1987 US Open. Flach survived 1–6, 6–4, 3–6, 6–1, 7–6 {\bf (17-15)}  saving five match points. Novak Djokovic won the first-set tiebreaker 16–14 over Alexandr Dolgopolov during their fourth round matchup at the 2011 US Open. Djokovic won the match in straight sets 7–6 {\bf (16-14)}, 6–4, 6–2, and eventually the tournament. In the 2012 US Open men's singles final, Andy Murray won the first set tiebreaker against Djokovic 12–10, which was the longest tiebreak in a US Open final match, and went on to win the match 7–6 {\bf (12-10)}, 7–5, 2–6, 3–6, 6–2. 

\smallskip

\noindent
{\bf Australian Open:} Jo-Wilfried Tsonga won a tiebreaker 20–18 in his first round match against Andy Roddick at the 2007 Australian Open, but Roddick eventually won the match, 6–7 {\bf (18-20)}, 7–6(7–2), 6–3, 6–3. Pedro Martínez won a tie break 17–15 in his first round match against Federico Delbonis at the 2022 Australian Open. Martinez won 7–6 {\bf (17–15)}, 3–6, 6–4, 6–2. Rafael Nadal won 16–14 in a first-set tiebreaker against Adrian Mannarino in the fourth round in 2022, and Nadal eventually winning the match 7–6 {\bf (16-14)}, 6–2, 6–2. 

\smallskip

\noindent
{\bf French Open:} Lorenzo Sonego won the third set tiebreaker {19-17} against Taylor Fritz in their third round match at the 2020 French Open. Sonego won 7–6(7–5), 6–3, 7–6 {\bf (19-17)}. The previous longest tie-break at the French Open had a score of {\bf 16–14}, which happened twice. 

\medskip

We now present the probability mass function of the number of points required in a $K$-point set tie-breaker under the player probabilities of $(p_A,p_B)$, with player $A$ being the first server.

\begin{theorem}
    For a $(p_A,p_B,K)$ set tie-breaker, with player $A$ serving first,
\begin{eqnarray*}
    \lefteqn{ p_{N_{ST}}(n|p_A,p_B,K) = \Pr\{\mbox{Set Tie-Breaker Ends After $n$ points}|p_A,p_B,K\} } \\ & = &
    \left[\theta_{ST}(K,n-K;p_A,p_B,K) + \theta_{ST}(n-K,K;p_A,p_B,K)\right] I\{K \le n \le 2(K-1)\} + \\ &&
   \left[\theta_{ST}(K-1,K-1;p_A,p_B,K) \Pr\{N_{STT} = 2\left(n - (2K-2)\right)/2\} \right] 
   I\{n \ge 2K;\ n\ \mbox{even}\},
\end{eqnarray*}
with
$$\Pr\{N_{STT} = 2L\} =  (p_Ap_B + q_Aq_B)^{L-1} (p_Aq_B + q_Ap_B) \ \mbox{for}\ L=1,2,\ldots.$$
\end{theorem}

\begin{proof}
    Consider first the length of a STT, denoted by $N_{STT}$. When it reaches this stage, which happens if the score reaches $(K-1,K-1)$, each player will have a serve for every two successive points played (since $AB$ then $BA$, etc.). Consequently,
    $$\Pr\{N_{STT} = 2L\} =  (p_Ap_B + q_Aq_B)^{L-1} (p_Aq_B + q_Ap_B) \ \mbox{for}\ L=1,2,\ldots,$$
    that is, $N_{STT}/2$ is geometric with success probability $p_Aq_B + q_Ap_B$.

    For $n \le 2(K-1)$, the set tie-breaker could end with a score of $(K,n-K)$ with player $A$ winning, or a score of $(n-K,K)$ with player $B$ winning. The probabilities of these events are given in Proposition \ref{prop: set tie-breaker probs}.

    For $n \ge 2K$ and $n$ even, letting $L = (n-(2K-2))/2$, the probability that the set tie-breaker ends with $n$ points played is
    $$\Pr\{(K-1,K-1)\} \Pr\{N_{STT} = 2L\} = \theta_{ST}(K-1,K-1;p_A,p_B,K) \Pr\{N_{STT} = 2L\}.$$
\end{proof}

Figure \ref{fig: PMFs Num Points Set Tie-Breaker} presents the probability mass functions of $N_{ST}$ when $K = 7$ and for two sets of $(p_A,p_B)$: (i) $p_A = p_B = .5$, and (ii) $p_A = p_B = .9$. From this PMF of $N_{ST}$, the distributional summary measures $\mu_{ST}$, $\sigma_{ST}^2$, and $\sigma_{ST}$, could be obtained via the usual formulas. It is actually possible to obtain somewhat closed-form expressions for these quantities as functions of $(p_A,p_B,K)$ owing to the fact that $N_{STT}/2$ is a geometric random variable. For instance, we obtain
\begin{eqnarray*}
 \mu_{ST}(p_A,p_B,K)  & = & \sum_{n=K}^{2(K-1)} n \left[\theta_{ST}(K,n-K;p_A,p_B,K) + \theta_{ST}(n-K,K;p_A,p_B,K) \right] +  \\
    && \theta_{ST}(K-1,K-1;p_A,p_B,K) \left[ 2(K-1) + \frac{2}{p_Aq_B + q_Ap_B}\right].
\end{eqnarray*}
\BLUE{Interestingly, even though $\theta_{ST}(p_A,p_B,K) =  \theta_{ST}(q_B,q_A,K) = 1 - \theta_{ST}(p_B,p_A,K) = 1 - \theta_{ST}(q_A,q_B,K)$, and $\mu_{ST}(p_A,p_B,K) = \mu_{ST}(q_A,q_B,K)$ as well as $\sigma_{ST}^2(p_A,p_B,K) = \sigma_{ST}^2(q_A,q_B,K)$, it turns out that $\mu_{ST}(p_A,p_B,K)$ need not equal $\mu_{ST}(p_B,p_A,K)$, and similarly $\sigma^2_{ST}(p_A,p_B,K)$ need not equal $\sigma^2_{ST}(p_B,p_A,K)$, so that with respect to the mean and variance, there is an effect of who serves first in the set-tiebreaker. As a concrete example with $K=3$, so easily verifiable manually, we have that $\mu_{ST}(.6,.55,7) = 5.3040$, while $\mu_{ST}(.55,.6,3) = 5.2965$, and $\sigma_{ST}^2(.6,.55,3) = 8.0736$, while $\sigma_{ST}^2(.55,.6,3) = 8.1006$.}

The expressions for $\mu_{ST}(p_A,p_B,K)$, as well as $\sigma_{ST}^2(p_A,p_B,K)$, will still contain several terms and will not be as informative compared to the plots of $\mu_{ST}$ and $\sigma_{ST}$ as functions of $(p_A,p_B,K)$. Figure \ref{fig: Mean and Std Num Points Set Tie-Breaker} shows the plot of the mean and standard deviation of $N_{ST}$ when $K=7$ as $p_A = p_B$ (left panels)  and $p_A = 1 - p_B$ (right panels) varies over $(0,1)$. More details of the behavior can be found in Figure \ref{fig: contour plot of mean of NST for K=7} which presents a contour plot of the mean of $N_{ST}$ over the unit-square. Interestingly, this contour plot is saddle-shaped, with a saddle-point of $(p_A, p_B) = (0.5,0.5)$, though this saddle-shapedness is already indicated by the plots in Figures \ref{fig: Mean and Std Num Points Set Tie-Breaker}. The theoretical results here indicate that when two players are outstanding servers, so that their $p_A$ and $p_B$ are very high, then the ST will require many points to complete. In practice, this lengthy STs manifested, for instance, with the now-retired American player, John Isner, whose very powerful and fast serves gave him a huge advantage when serving. He had a fifth set ST at Wimbledon in 2010 against Frenchman Nicolas Mahut which ended on a score of 70 games to 68 games (and spanned three days to play the whole match) \cite{wiki:IsnerMahut}, when the Wimbledon Grand Slam tournament was still using a tie-breaker system in the final set, wherein the player who obtains an advantage of two {\em games}, not {\em points}, wins the tie-breaker! Note that if a player has a high $p$ of winning a point when serving, then this gets further magnified in terms of the probability of winning a service game. From the formula of $\theta_G(p)$, if for example, Isner has probability of 0.70 winning a service point, then the probability of his winning a service game is 0.90, so if his opponent has a probability of also 0.70 of winning his own service point, then both of them each has probability of 0.90 of winning their service games. As such, this could lead to very lengthy final set pre-2019 Wimbledon-style tie-breakers, such as the afore-mentioned Isner-Mahut fifth set tie-breaker! Later in section \ref{subsec: two-parameter systems}, we compare the tennis match system with a Best-of-$(2L+1)$-Games system with a tie-breaker system that requires the winner to win by 2 games, which will shed further light on the Isner-Mahut match.


\begin{figure}[h]
\begin{center}
\includegraphics[width=\mymultB\textwidth,height=\mymultB\textwidth]{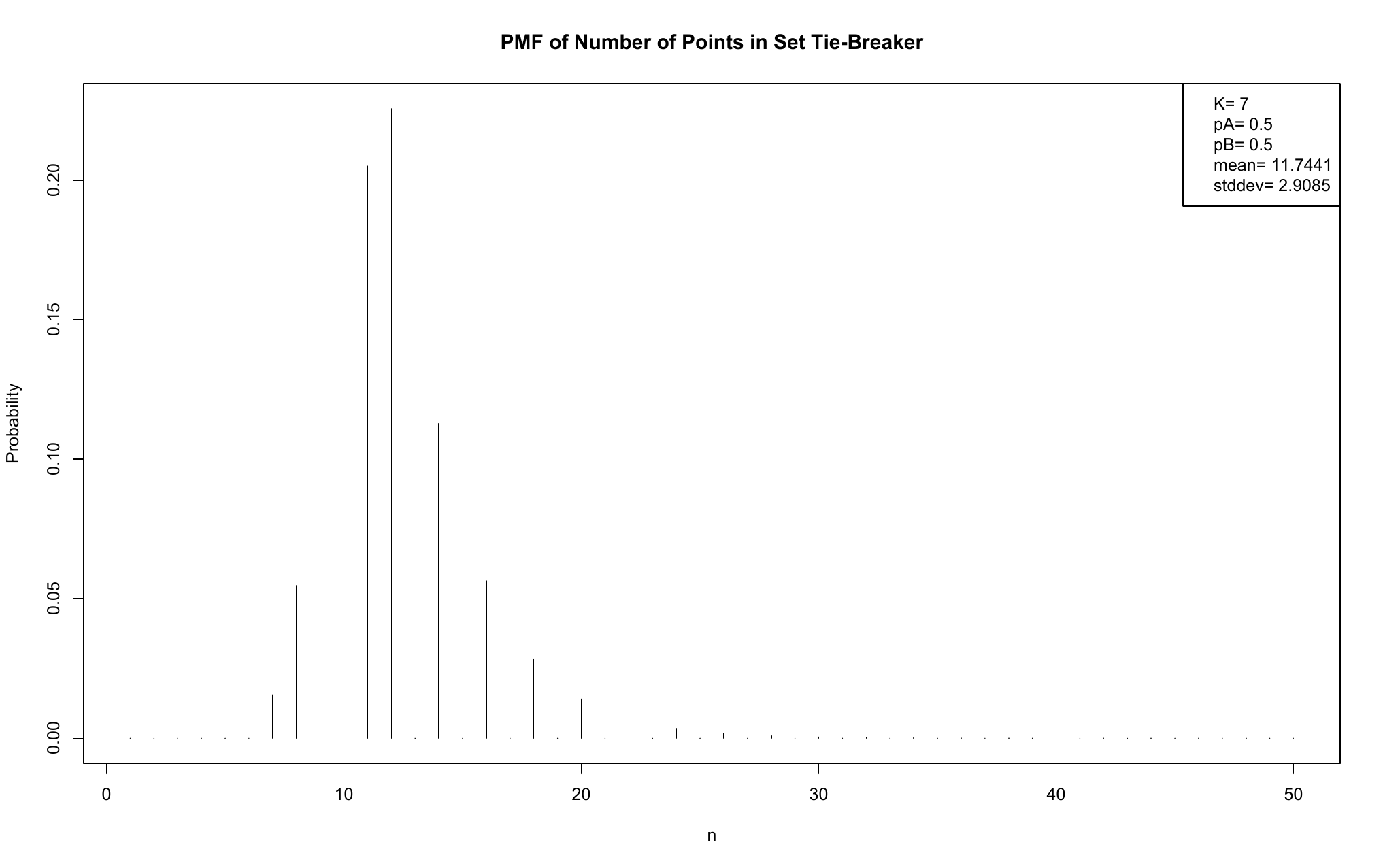} 
\includegraphics[width=\mymultB\textwidth,height=\mymultB\textwidth]{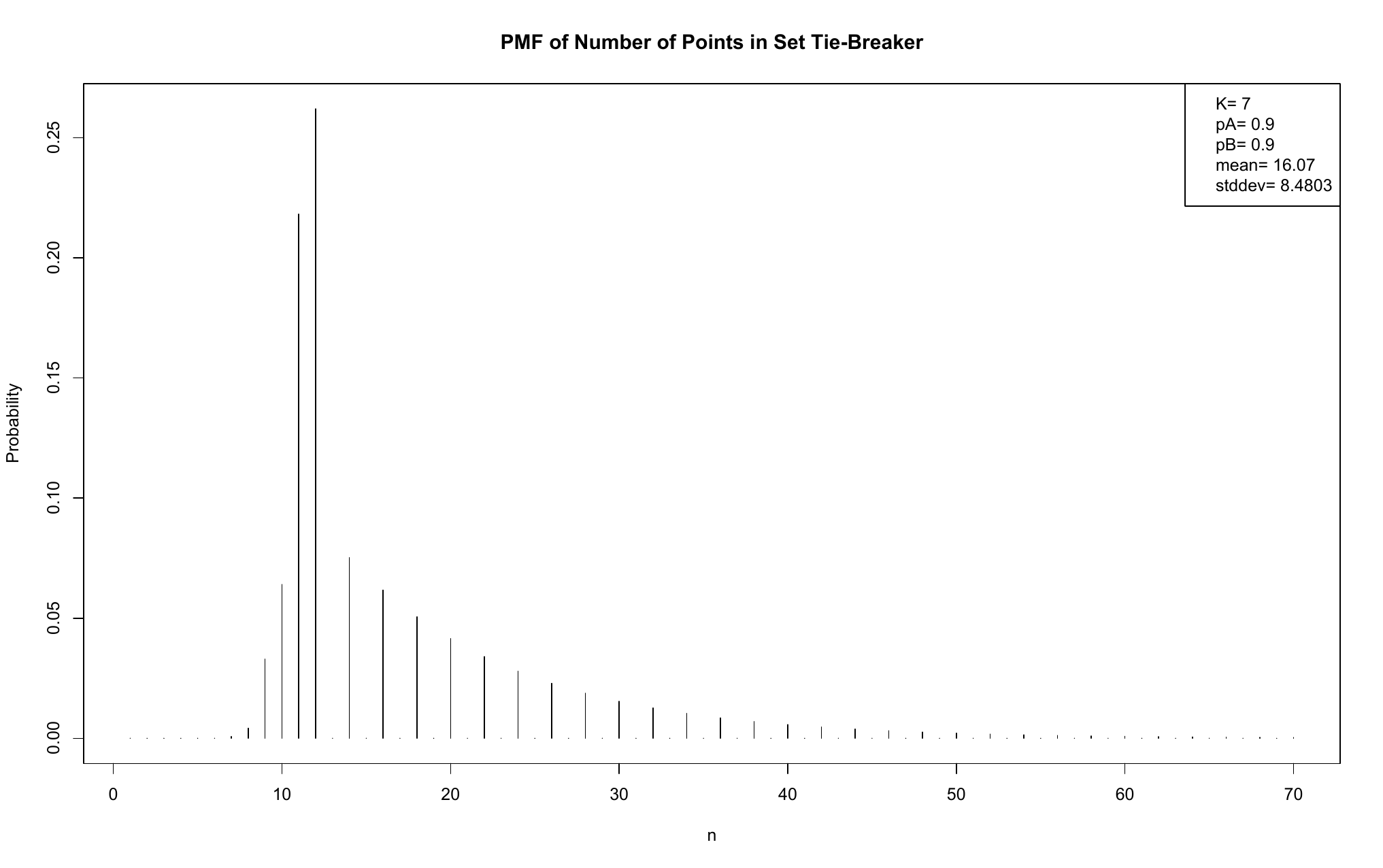}
\caption{Probability mass function (pmf) of $N_{ST}$ for $K=7$ and when $p_A=p_B=.5$ (top) and $p_A=p_B=.9$ (bottom).}
\label{fig: PMFs Num Points Set Tie-Breaker}
\end{center}
\end{figure}

\begin{figure}[h]
\begin{center}
\includegraphics[width=\mymultB\textwidth,height=\mymultB\textwidth]{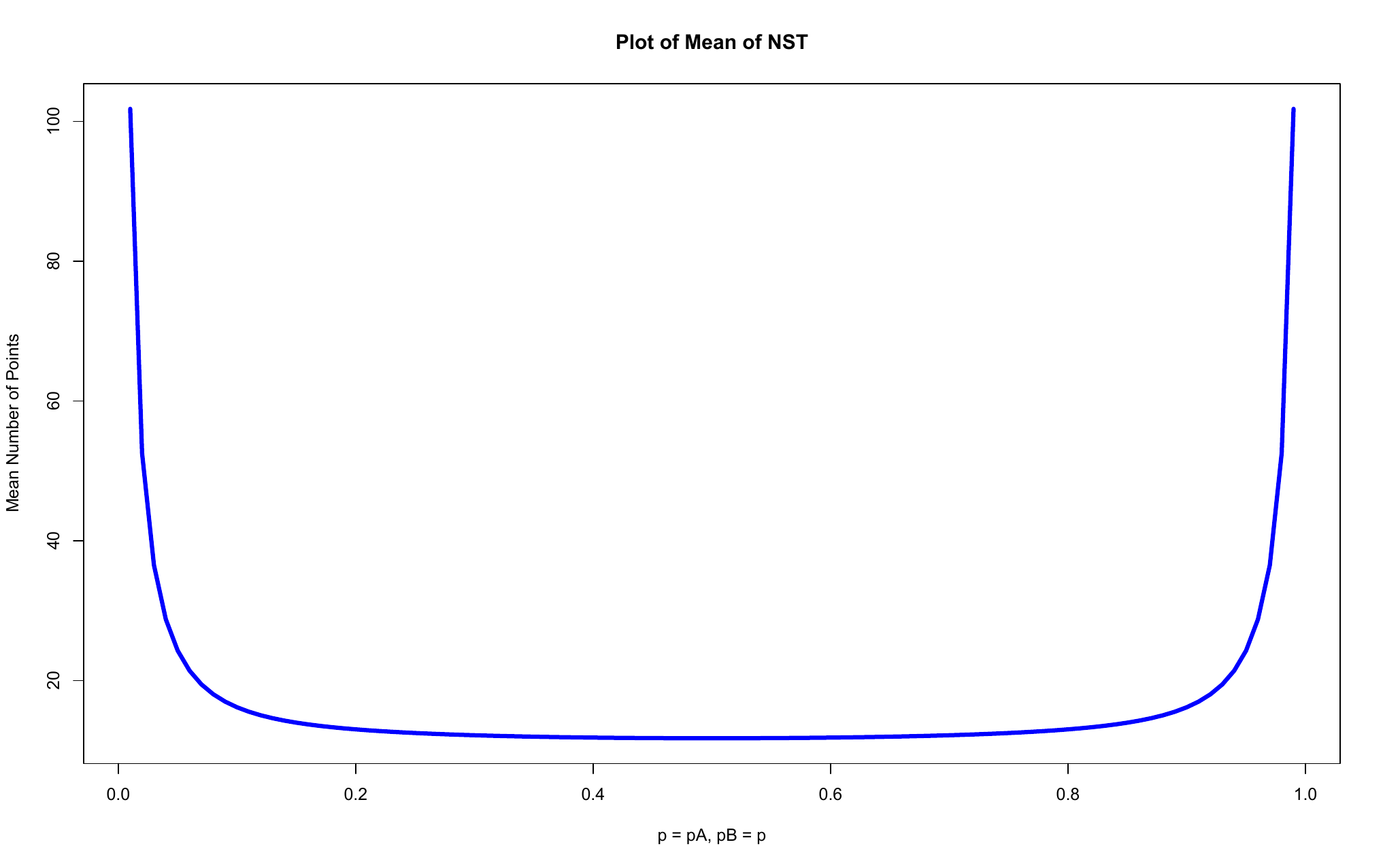} 
\includegraphics[width=\mymultB\textwidth,height=\mymultB\textwidth]{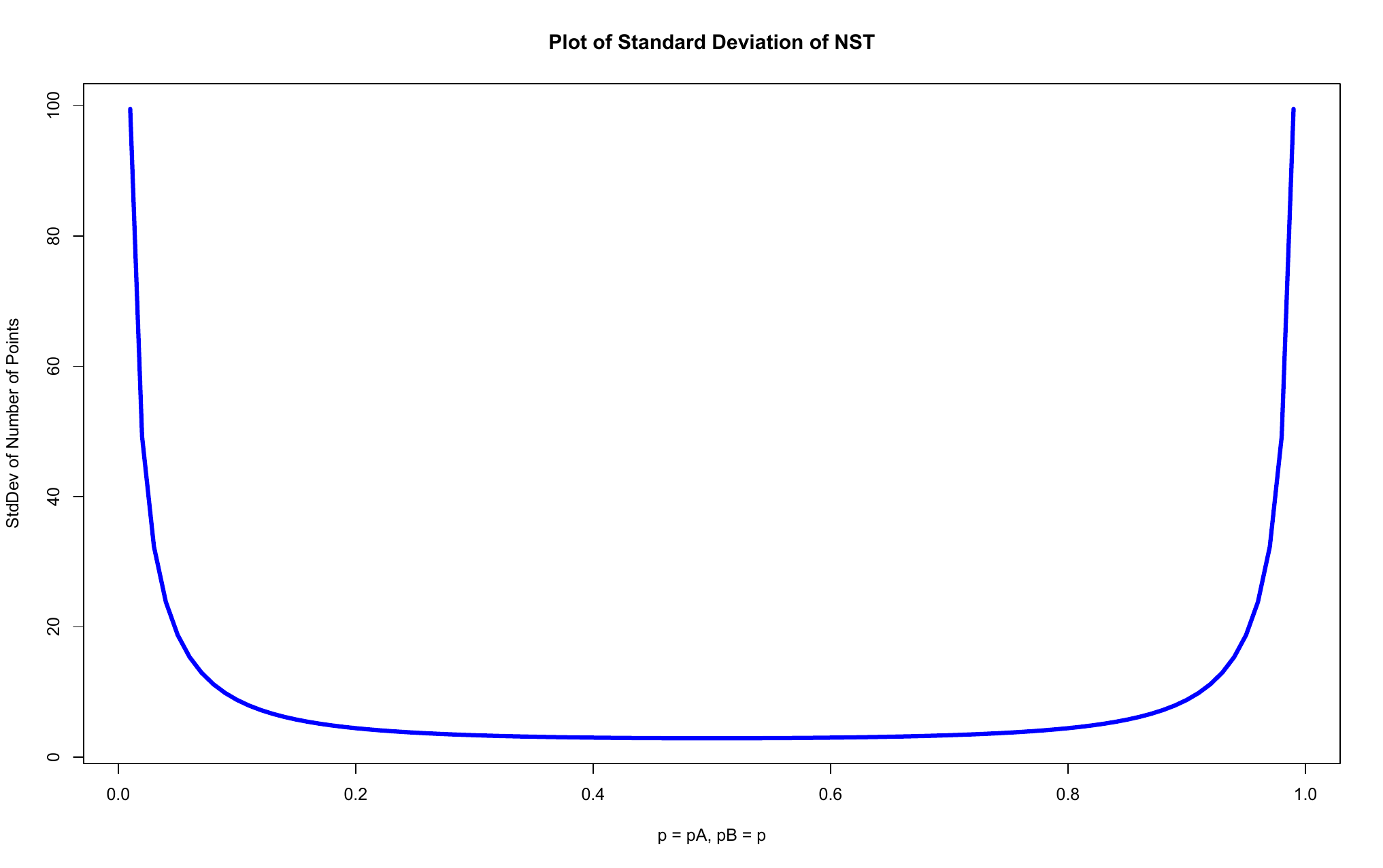} \\
\includegraphics[width=\mymultB\textwidth,height=\mymultB\textwidth]{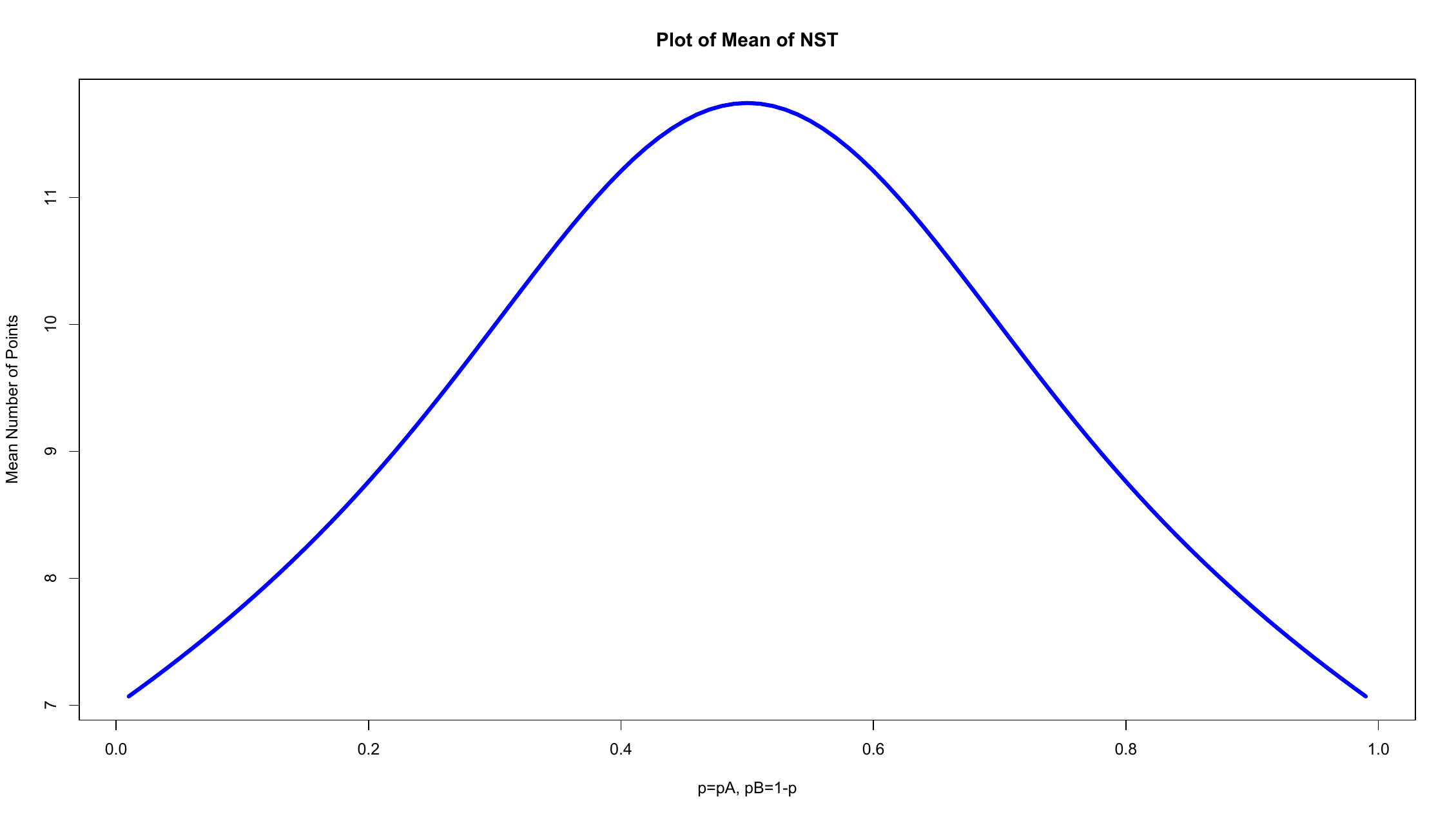} 
\includegraphics[width=\mymultB\textwidth,height=\mymultB\textwidth]{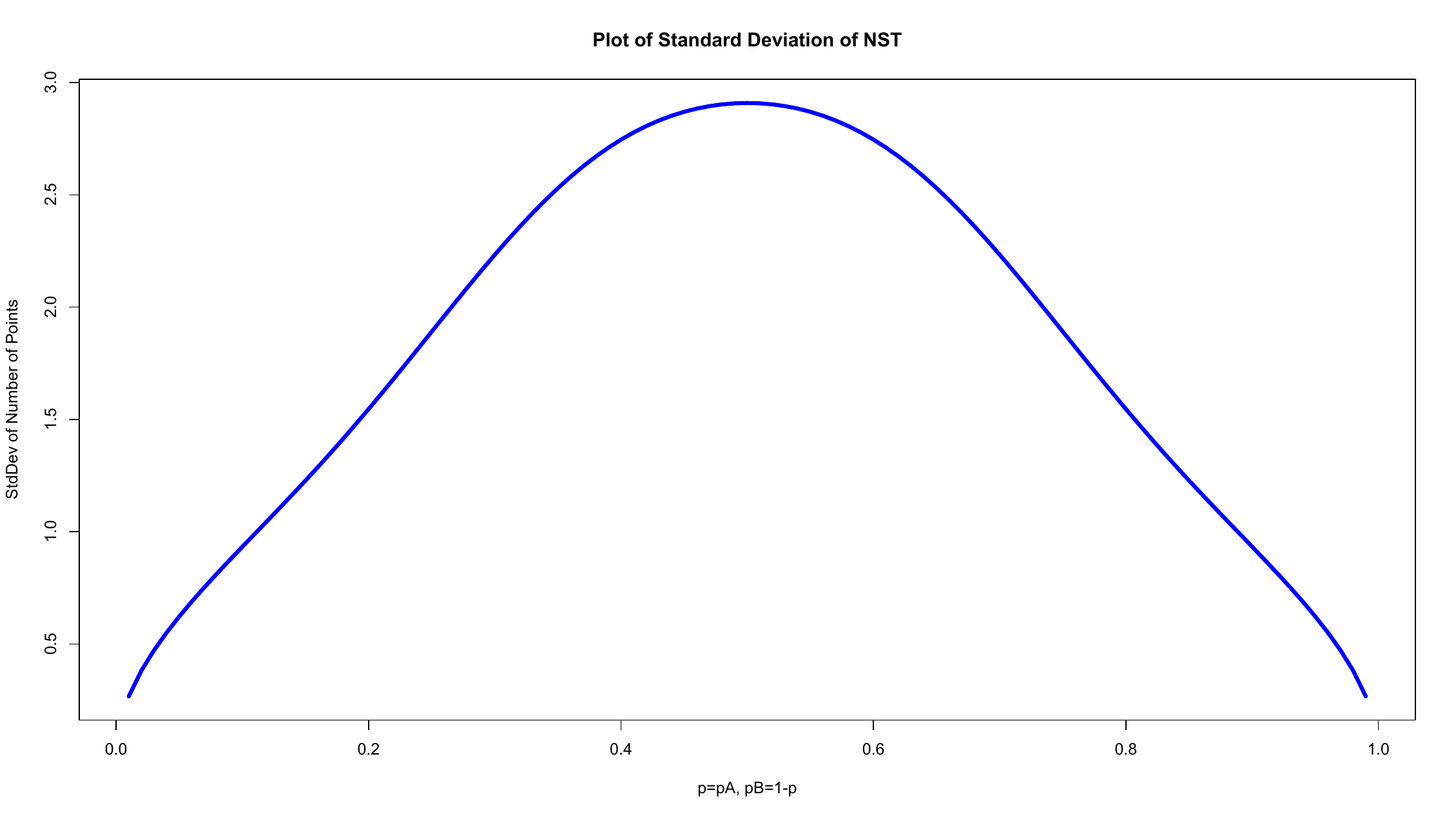} 
\caption{Mean and Standard Deviation of $N_{ST}$, the number of points played to end a set tie-breaker, for $K=7$, as $p_A = p_B$ (top panels) and $p_A = 1 - p_B$ (bottom panels) ranges over $(0,1)$.}
\label{fig: Mean and Std Num Points Set Tie-Breaker}
\end{center}
\end{figure}

\begin{figure}[h]
\begin{center}
\includegraphics[width=\mymultB\textwidth,height=\mymultB\textwidth]{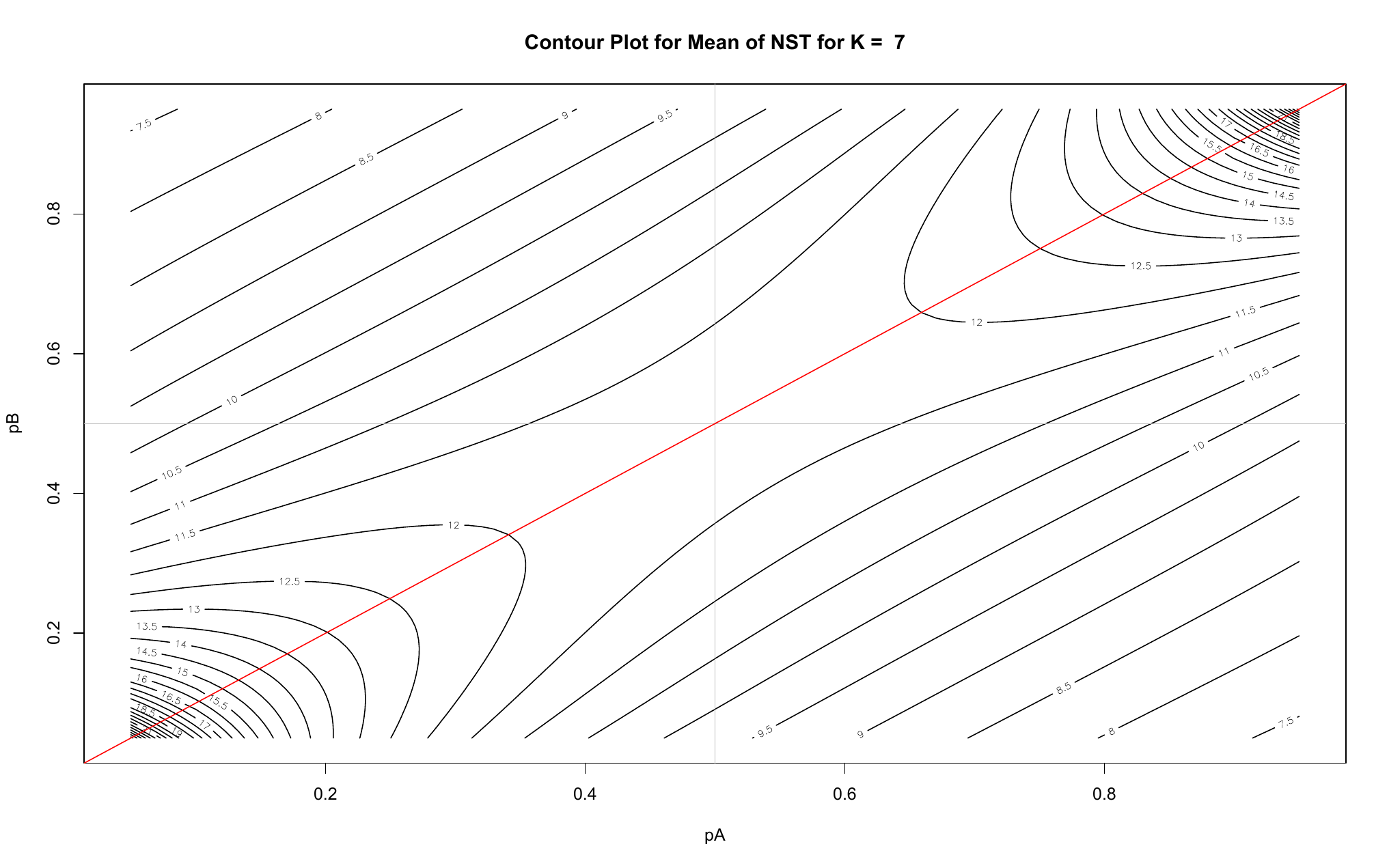} 
\caption{Contour Plot of the Mean of $N_{ST}$, the number of points played to end a set tie-breaker, for $K=7$.}
\label{fig: contour plot of mean of NST for K=7}
\end{center}
\end{figure}

\subsection{Set Probabilities}

We are now in a position to compute the probability that player $A$ will win the set with a $K$-points set tie-breaker, if he/she is the first to serve. We denote this probability by $\theta_S(p_A,p_B,K)$. Recall that in a set, the two players alternate in serving games, and the first player to win 6 games, with an advantage of at least 2 games, wins the set. If each wins 5 games, then the first to win 7 games, provided that the other is stuck at 5 won games, wins the set, otherwise the $K$-points ST ensues.

We introduce the following notation: $t_A(g)$ is the number of games that player $A$ serves, and $t_B(g) = g - t_A(g)$ is the number of games player $B$ serves, in the first $g$ games of the set. $J_A(g)$ will indicate whether player $A$ is serving on the $g$th game ($J_A(g) = 1$). Then, it is immediately clear that
$$t_A(g) = \lfloor \frac{g+1}{2} \rfloor \quad \mbox{and} \quad J_A(g) = g \bmod 2,$$
where $\lfloor r \rfloor$ is the largest integer less than or equal to $r$, the {\em floor} of $r$.

\begin{theorem}
For a set with a $K$-points set tie-breaker,
\begin{eqnarray*}
    \theta_S(p_A,p_B,K) & = & \sum_{h=0}^4  \theta_S(6,h;p_A,p_B) + \theta_S(7,5;p_A,p_B) + \theta_S(6,6;p_A,p_B) \theta_{ST}(p_A,p_B,K),
\end{eqnarray*}
where
\begin{eqnarray*}
\lefteqn{ \theta_S(6,h;p_A,p_B)   \equiv    \Pr\{\mbox{$(6,h)$ score}\}  = [\theta_G(p_A)]^{J_A(6+h)} [1 - \theta_G(p_B)]^{1-J_A(6+h)} \times} \\ 
    && \sum_{j=0}^5 \left\{ {{t_A(6+h-1)} \choose j} {{t_B(6+h-1)} \choose {5-j}}  [\theta_G(p_A)]^j [1 - \theta_G(p_A)]^{t_A(6+h-1) - j} \times \right. \\ && \left. [1 - \theta_G(p_B)]^{5-j} [\theta_G(p_B)]^{t_B(6+h-1) - (5-j)} \right\};
\end{eqnarray*}
\begin{eqnarray*}
\theta_S(5,5;p_A,p_B)  & \equiv & \Pr\{\mbox{$(5,5)$ score} \} = \sum_{j=0}^5
    {5 \choose j}^2 [\theta_G(p_A) \theta_G(p_B)]^j 
    [(1-\theta_G(p_A))(1-\theta_G(p_B))]^{5-j}; \\
%
\theta_S(7,5;p_A,p_B) & \equiv &  \Pr\{\mbox{$(7,5)$ score} \} = 
 \theta_S(5,5;p_A,p_B) \theta_G(p_A) [1-\theta_G(p_B)]; \\
%
 \theta_S(6,6;p_A,p_B) & \equiv & \Pr\{\mbox{$(6,6)$ score} \} =  \theta_S(5,5;p_A,p_B) \left[\theta_G(p_A) \theta_G(p_B) + (1-\theta_G(p_A))(1-\theta_G(p_B))\right].
\end{eqnarray*}
Alternatively,
\begin{eqnarray*}
\lefteqn{ \theta_S(p_A,p_B,K)   =  \Pr\{B(5,\theta_G(p_A)) + B(5,1-\theta_G(p_B)) \ge 6\} + } \\ &&  \Pr\{B(5,\theta_G(p_A)) + B(5,1-\theta_G(p_B)) = 5\} \left[
    \theta_G(p_A)(1-\theta_G(p_B)) + \right. \\ && \left. \left\{\theta_G(p_A) \theta_G(p_B) + (1-
    \theta_G(p_A)) (1 - \theta_G(p_B))\right\} \theta_{ST}(p_A,p_B,K) \right]. 
\end{eqnarray*}
In addition, $\theta_S(p_A,p_B,K) = \theta_S(q_B,q_A,K) = 1 - \theta_S(q_A,q_B,K) = 1 - \theta_S(p_B,p_A,K).$
\end{theorem}

\begin{proof}
   The first equation is just by virtue of the Theorem of Total Probability and how player $A$ could win the set. The expression for $\Pr\{\mbox{$(6,h)$ score}\}$ is obtained by considering the ways in which player $A$ could win 6 games, with the last game played won by $A$, his/her 6th won game. Prior to the last game, player $A$ must win 5 games, but this could be achieved by winning some of the games he/she serves, and winning some of the games served by player $B$. The combinatorial terms just count the number of ways of achieving this. To get in a score of $(7,5)$, the score must reach $(5,5)$, then $A$ must win two games, one with his/her serve, the other with $B$'s serve. On the other hand, to get to the ST, the score must reach $(6,6)$ from $(5,5)$, which would happen with a game won by each of them after $(5,5)$, but this could happen in two ways: first $A$ wins, then $B$; and second, $A$ loses, then wins, again taking into account that each of them will have a game served after $(5,5)$. When the score reaches $(6,6)$, then the ST ensues, and player $A$ will win this ST with probability $\theta_{ST}(p_A,p_B,K)$.

   The alternative form of $\theta_S(p_A,p_B,K)$ is obtained similarly from the arguments made in the proofs of Proposition \ref{prop: set tie-breaker probs} or Corollary \ref{coro: under pA=pB for theta_ST}.
   The last identity in the statement of the Theorem follows 
   %
   %
   by using the alternative form for $\theta_S(p_A,p_B,K)$ and noting that
   $1-\theta_G(q_A) =  1 - [1-\theta_G(p_A)] = \theta_G(p_A)$.
\end{proof}

\begin{figure}[h]
    \centering
\includegraphics[width=\mymultB\textwidth,height=\mymultB\textwidth]{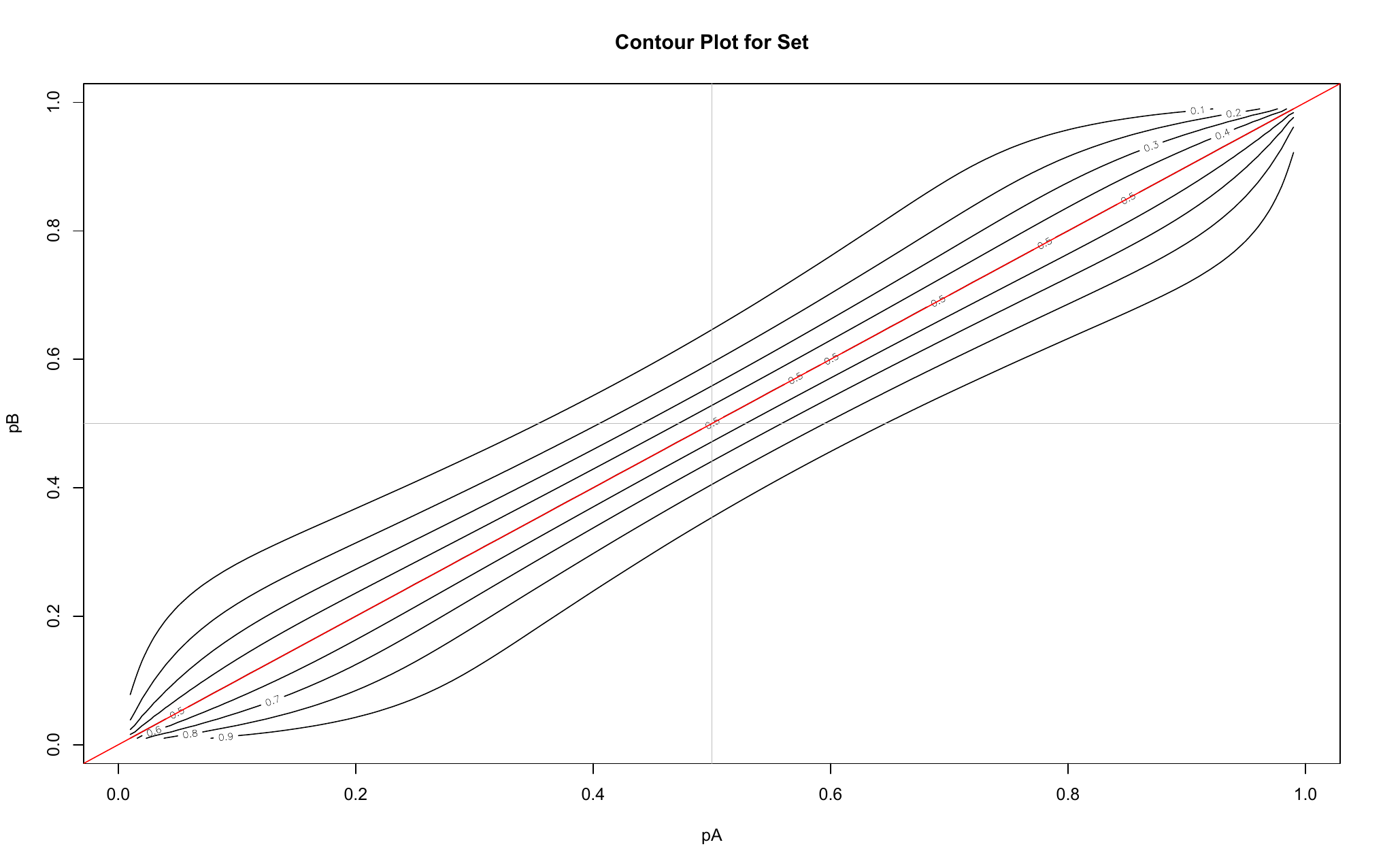}
    \caption{Contour plot of the probability of first server winning the set when $K = 7$. $pA$ is the probability of first server winning point on his/her serve; while $pB$ is the probability of the first receiver winning point on his/her serve.}
    \label{fig: contour set}
\end{figure}

Figure \ref{fig: contour set} presents a contour plot of $\theta_{S}$ as $p_A$ and $p_B$ range over $0.01$ to $0.99$ and when $K = 7$. The probabilities were computed using an {\tt R} program we developed implementing the formula in the preceding theorem.
We again observe from this contour plot that when $p_A = p_B = p \in (0,1)$, we have that $\theta_S(p,p,K) = 1/2$, demonstrating that there is {\em no advantage to who serves first in the set}, provided that the players are of equal abilities, but when not perfect on their serves, that is, $0 < p < 1$. This fairness question was also addressed in Newton and Keller's paper \cite{Newton_Keller_2005} in their subsection 3.3. Compared to the contour plot for the ST in Figure \ref{fig: contour set tie breaker}, observe that for the set, the contour plot is more compact or narrower. This demonstrates that the tennis set decision system is able to more efficiently discriminate between the two players compared to the ST, as one should expect.

\subsection{Number of Points in a Set}

Let $N_S$ denote the number of points played in a set. We seek the mean $\mu_S$ and the standard deviation $\sigma_S$ of $N_S$ as a function of $(p_A,p_B,K)$, with player $A$ serving in the first game of the set. To do so, we need to decompose $N_S$ into the different cases in which the set could end. For $a,b \in \{0,1,\ldots,7\}$, let
\RED{$$N_S(a,b) = \mbox{number of points when $A\ (B)$ wins $a\ (b)$ games, respectively.}$$}
Then, with $G_A\ (G_B)$ denoting the number of games won by player $A\ (B)$, respectively, we have the decomposition of $N_S$ given by
\begin{eqnarray*}
    N_S & = & \sum_{h=0}^4 \left[I\{G_A = 6,G_B = h\} N_S(6,h) + I\{G_A = h,G_B = 6\} N_S(h,6) \right] + \\
   && I\{G_A=7,G_B=5\} N_S(7,5) + I\{G_A=5,G_B=7\} N_S(5,7) + 
   I\{G_A=6,G_B=6\} \left[N_S(6,6) + N_{ST}\right].
\end{eqnarray*}
Next, let
$$\mu_S(a,b; p_A,p_B) = E[N_S(a,b)|p_A,p_B] \quad \mbox{and} \quad
\sigma_S^2(a,b; p_A,p_B) = Var[N_S(a,b)|p_A,p_B].$$
We could also decompose $N_S(a,b)$ via the games in which $A$ serves, and those where $B$ serves. For a total of $a+b$ games, the number of games served by $A$ is $t_A(a+b)$, while those by $B$ is $t_B(a+b) = (a+b) - t_A(a+b)$. Then,
$$N_S(a,b) = \sum_{i=1}^{t_A(a+b)} N_{Gi}^A + \sum_{j=1}^{t_B(a+b)} N_{Gj}^B,$$
\RED{where $N_{Gi}^A$ is the number of points played in the $i$th game served by $A$, and $N_{Gj}^B$ is the number of points played in the $j$th game served by $B$.} Note that the $N_{Gi}^A$s are independent and identically distributed (IID), and similarly, $N_{Gj}^B$s are also IID. Consequently,
$$\mu_S(a,b;p_A,p_B) = t_A(a+b) \mu_G(p_A) + t_B(a+b) \mu_G(p_B);$$
$$\sigma_S^2(a,b;p_A,p_B) =  t_A(a+b) \sigma_G^2(p_A) + t_B(a+b) \sigma_G^2(p_B).$$

\begin{theorem}
    For a set where player $A$ serves first, the mean, variance, and standard deviation of the number of points played are, respectively,
    \begin{eqnarray*}
       \lefteqn{ \mu_S(p_A,p_B,K)  =  \sum_{h=0}^4
        \left[\theta_S(6,h;p_A,p_B) \mu_S(6,h;p_A,p_B) + \theta_S(h,6;p_A,p_B) \mu_S(h,6;p_A,p_B)\right] + } \\
        &&  \theta_S(7,5;p_A,p_B) \mu_S(7,5;p_A,p_B) + \theta_S(5,7;p_A,p_B) \mu_S(5,7;p_A,p_B) + \\
        && \theta_S(6,6;p_A,p_B)\left[\mu_S(6,6;p_A,p_B) + \mu_{ST}(p_A,p_B,K)\right];
    \end{eqnarray*}
    \begin{eqnarray*}
        \lefteqn{ \sigma_S^2(p_A,p_B,K)  = 
        \left[ \left\{ \sum_{h=0}^4
        \left[\theta_S(6,h;p_A,p_B) \mu_S^2(6,h;p_A,p_B) + \theta_S(h,6;p_A,p_B) \mu_S^2(h,6;p_A,p_B)\right] + \right.\right. } \\
        &&  \theta_S(7,5;p_A,p_B) \mu_S^2(7,5;p_A,p_B) + \theta_S(5,7;p_A,p_B) \mu_S^2(5,7;p_A,p_B) +  \\
        && \left.\left. \theta_S(6,6;p_A,p_B)\left[\mu_S(6,6;p_A,p_B) + \mu_{ST}(p_A,p_B,K)\right]^2 \right\} - \mu_S^2(p_A,p_B,K)\right] + \\ 
        && \left\{ \sum_{h=0}^4
        \left[ \theta_S(6,h;p_A,p_B) \sigma_S^2(6,h;p_A,p_B) + \theta_S(h,6;p_A,p_B) \sigma_S^2(h,6;p_A,p_B)\right] + \right. \\
        &&  \theta_S(7,5;p_A,p_B) \sigma_S^2(7,5;p_A,p_B) + \theta_S(5,7;p_A,p_B) \sigma_S^2(5,7;p_A,p_B) + \\
        && \left. \theta_S(6,6;p_A,p_B)\left[\sigma_S^2(6,6;p_A,p_B) + \sigma_{ST}^2(p_A,p_B,K)\right] \right\}; 
    \end{eqnarray*}
    and $\sigma_S(p_A,p_B,K) = +\sqrt{\sigma_S^2}(p_A,p_B,K)$.
    \BLUE{Furthermore, we have
    %
    %
    %
    %
    %
    %
    $$\mu_S(p_A,p_B,K) = \mu_S(q_A,q_B,K) \quad \mbox{and} \quad \sigma_S^2(p_A,p_B,K) = \sigma_S^2(q_A,q_B,K).$$}
\end{theorem}

\begin{proof}
The results follow immediately using the iterated rules for mean and variance, where the conditioning is on $(G_A,G_B)$. 
\BLUE{
The last identities follow if we could show that that $\theta_S(a,b;p_A,p_B) = \theta_S(b,a;q_A,q_B)$ for all $(a,b)$ with $a,b \in \{0,1,\ldots,7\}$, since if this is true, then we obtain that
$$\Pr\{N_S=n|p_A,p_B\} = \Pr\{N_S=n|q_A,q_B\}$$
for all realizable $n$, so that the claimed identities result. 

Next, recall that if $V \sim B(n,p)$, then $(n-V) \sim B(n,1-p)$. In addition, an earlier result is that $\theta_G(p) = 1 - \theta_G(q)$. Then, since  $t_A(a+b) + t_B(a+b) = a+b$, we have
\begin{eqnarray*}
\lefteqn{\theta_S(a,b;p_A,p_B)  =  \Pr\{B(t_A(a+b),\theta_G(p_A)) + B(t_B(a+b),1-\theta_G(p_B)) = a\} } \\
& = & \Pr\{[t_A(a+b) - B(t_A(a+b),1-\theta_G(p_A))] + [t_B(a+b) - B(t_B(a+b),1 -(1-\theta_G(p_B))] = a\} \\
& = & \Pr\{B(t_A(a+b),1-\theta_G(p_A)) + B(t_B(a+b),1 -(1-\theta_G(p_B)) = (a+b) - a\} \\
& = & \Pr\{B(t_A(a+b),\theta_G(q_A)) + B(t_B(a+b),1 -\theta_G(q_B)) = b\} \\
& = & \theta_S(b,a;q_A,q_B).
\end{eqnarray*}
This completes the proof of the theorem.}
\end{proof}

Given $(p_A,p_B,K)$, these results for a set provide us an informative summary about the probabilities of player $A$ ($B$) winning with a certain score, as well as the conditional mean and the conditional standard deviation of the number of points that will be played under different score configurations as in Table \ref{tab: set statistics}. Note that a set score of $(7,6)$ or $(6,7)$ means that the set went into a set tie-breaker and $A$ or $B$, respectively, won the set tie-breaker.
\begin{table}[h]
\begin{center}
\begin{tabular}{||c||c|c|c||c|c||} \hline
Loser's & \multicolumn{3}{c||}{Probability} & \multicolumn{2}{c||}{Points Played} \\ \cline{2-4} \cline{5-6}
Score in Set & $A$ wins & $B$ wins & $A$ or $B$ wins & CMean & CVariance \\ \hline\hline
$0$ & $\theta_S(6,0)$ & $\theta_S(0,6)$ &  $\theta_S(0)$ & $\mu_S(0)$ & $\sigma_S^2(0)$ \\ 
$1$ & $\theta_S(6,1)$ & $\theta_S(1,6)$ & $\theta_S(1)$ & $\mu_S(1)$ & $\sigma_S^2(1)$ \\ 
$2$ & $\theta_S(6,2)$ & $\theta_S(2,6)$ & $\theta_S(2)$ & $\mu_S(2)$ & $\sigma_S^2(2)$ \\ 
$3$ & $\theta_S(6,3)$ & $\theta_S(3,6)$ & $\theta_S(3)$ & $\mu_S(3)$ & $\sigma_S^2(3)$ \\ 
$4$ & $\theta_S(6,4)$ & $\theta_S(4,6)$ & $\theta_S(4)$ & $\mu_S(4)$ & $\sigma_S^2(4)$ \\ 
$5$ & $\theta_S(7,5)$ & $\theta_S(5,7)$ &  $\theta_S(5)$ & $\mu_S(5)$ & $\sigma_S^2(5)$ \\ 
$6$ & $\theta_S(7,6)$ & $\theta_S(6,7)$ &  $\theta_S(6)$ & $\mu_S(6)$ & $\sigma_S^2(6)$ \\\hline
Set Overall & $\theta_S$ & $1 - {\theta}_S$ & 1 & $\mu_S$ & $\sigma_S^2$ \\ \hline\hline
\end{tabular}
\caption{Summary of Statistics for a Set given $(p_A,p_B,K)$. In the table, $\theta_S(g_A,g_B,p_A,p_B,K)$ will be written as $\theta_S(g_A,g_B)$, and similarly for $\mu_S(g_L,p_A,p_B,K)$ and $\sigma_S(g_L,p_A,p_B,K)$, while $\theta_S(h)$ is the probability of either player winning with the loser having $h$ points.}
\label{tab: set statistics}
\end{center}
\end{table}
As a concrete example, if $p_A=.6,p_B=.55,K=7$, these summary characteristics are contained in Table \ref{tab: characteristics for set with (pA,pB,K) = (.6,.55,7)}.
\begin{table}[h]
\begin{center}
\begin{tabular}{||c||c|c|c||c|c||} \hline
Loser's & \multicolumn{3}{c||}{Probability} & \multicolumn{2}{c||}{Points Played} \\ \cline{2-4} \cline{5-6}
Score in Set & $A$ wins & $B$ wins & $A$ or $B$ wins & CMean & CVariance \\ \hline\hline
       0& 0.021& 0.004&  0.026&   39.495&  42.419 \\
       1& 0.095& 0.012&  0.107&   45.979&  49.127 \\
       2& 0.095& 0.050&  0.145&   52.660&  56.559\\
       3& 0.204& 0.044&  0.248&   59.145&  63.267\\
       4& 0.105& 0.122&  0.227&   65.825&  70.698\\
       5& 0.069& 0.041&  0.109&   78.991&  84.838\\
       6& 0.080& 0.058&  0.138&   90.751&  93.551 \\ \hline
 Overall & 0.669& 0.331&  1.000&   64.352& 267.271 \\ \hline\hline
\end{tabular}
\caption{Summary of Statistics for a Set given $(p_A,p_B,K) = (.6,.55,7)$.}
\label{tab: characteristics for set with (pA,pB,K) = (.6,.55,7)}
\end{center}
\end{table}
Figure \ref{fig: mean plot points set pA=pB=p and contour} presents a plot of the mean number of points played in a set as a function of $p = p_A = p_B$, together with the 2 standard deviation bounds, and also a contour plot of the mean number of points played in a set, again for $K=7$.

\begin{figure}[h]
\begin{center}
\includegraphics[width=\mymultB\textwidth,height=\mymultB\textwidth]{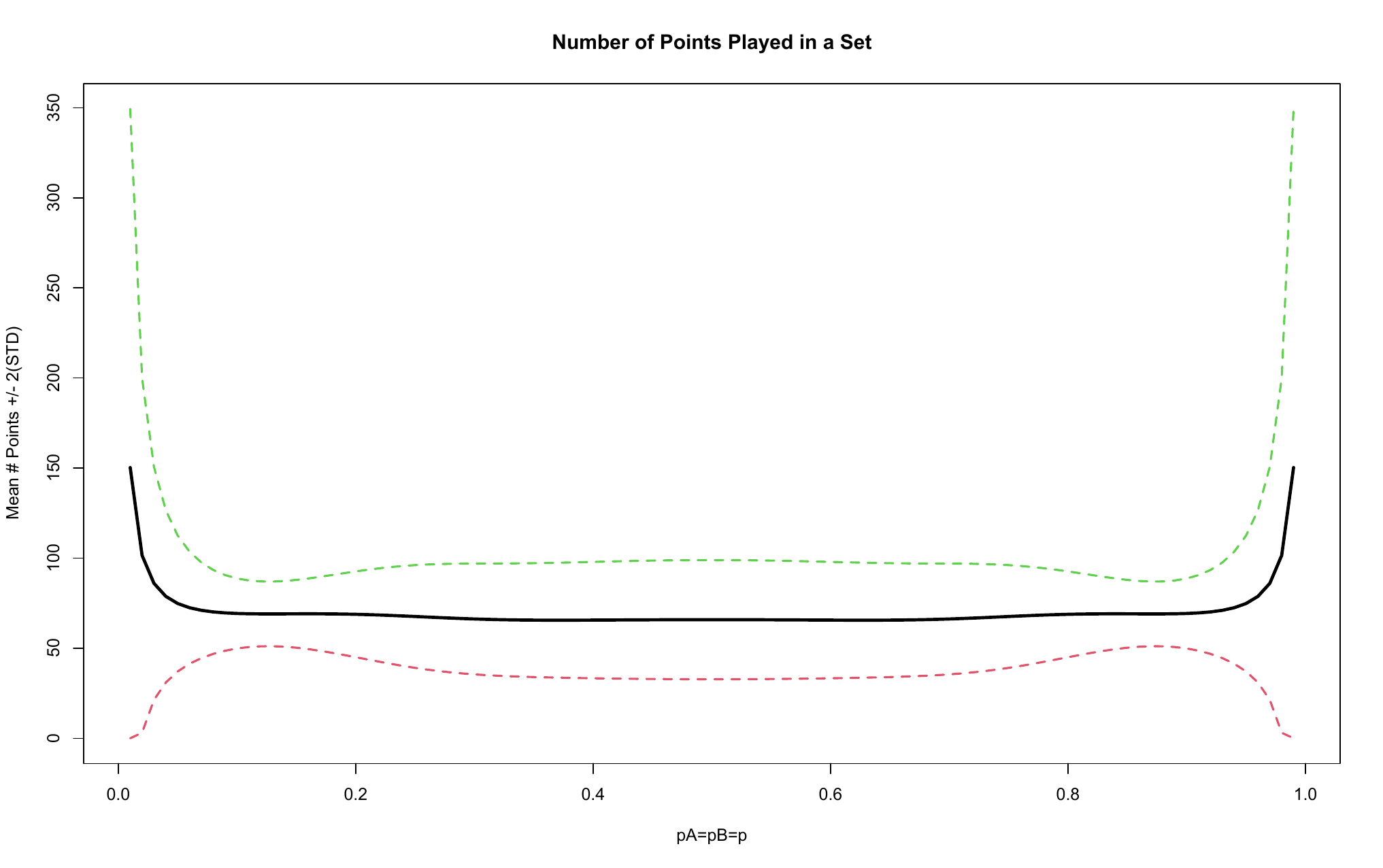} 
\includegraphics[width=\mymultB\textwidth,height=\mymultB\textwidth]{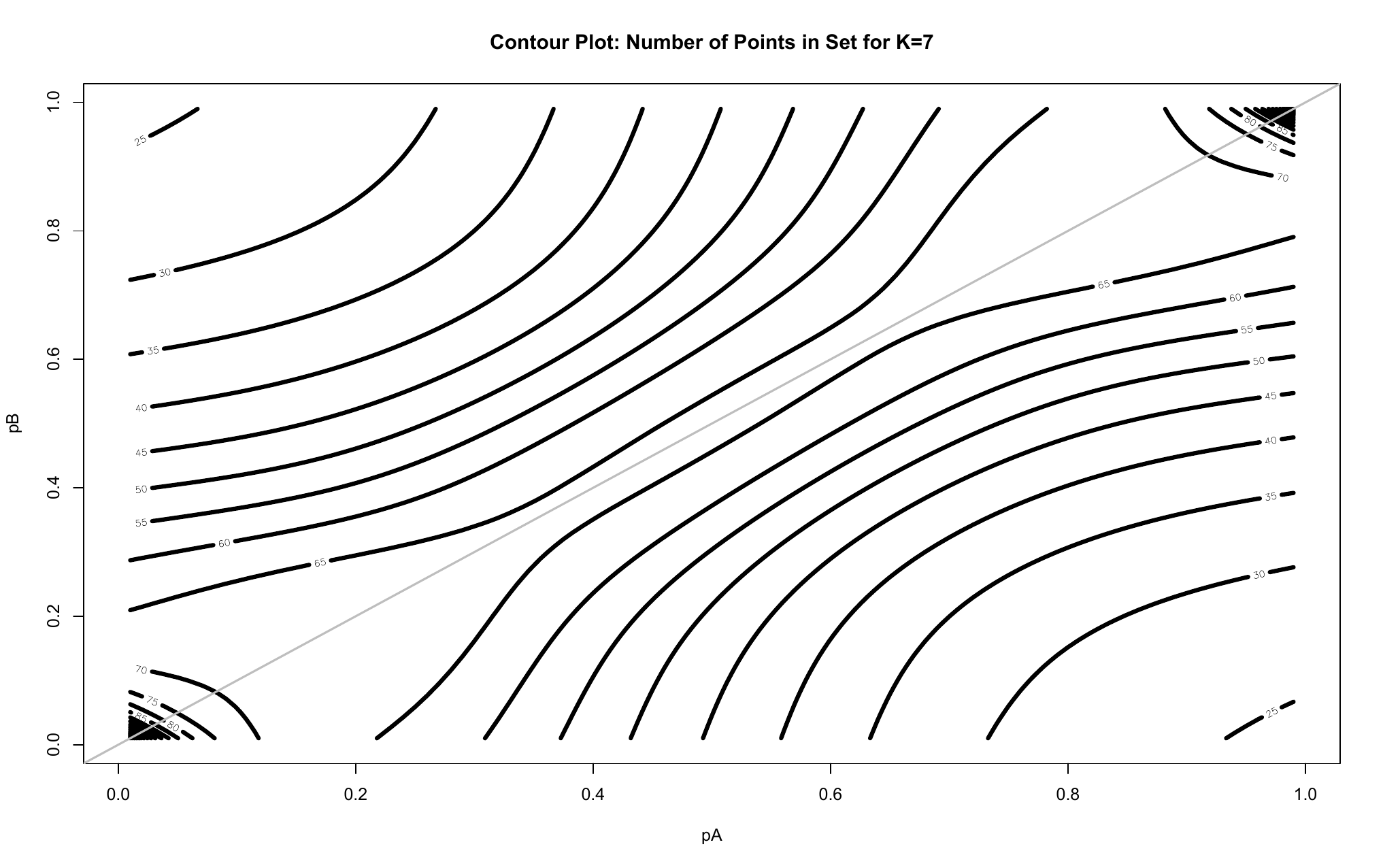}
\caption{Plot of mean number of points played in a set with $K=7$ with respect to $p = p_A = p_B$, together with a contour plot.}
\label{fig: mean plot points set pA=pB=p and contour}
\end{center}
\end{figure}


\BLUE{However, owing to the result that $\mu_{ST}(p_A,p_B,K)$ need not equal $\mu_{ST}(p_B,p_A,K)$, and similarly for the variance of $N_{ST}$, it turns out that $\mu_S(p_A,p_B,K)$ need not equal $\mu_S(p_B,p_A,K)$, and similarly for the variance of $N_S$. In essence, the effect of who serves first in the set manifests in the mean and variance of the number of points played.}

\section{Match}
\label{sec: Match}

\subsection{Match Probabilities}

Finally, we have reached a stage where we could compute the probability that the first server in a match will win the match. We consider the system in the US Open Tennis tournament, where for the men's section, it is a Best-of-Five-Sets match, with the ST for the first 4 sets having $K=7$, while for the fifth set, it has $K=10$. For the women's section, it is a Best-of-Three-Sets match, with the first two sets having $K=7$ for the set tie-breaker, while for the third set, we have $K = 10$.
Consider therefore a Best-of-$L$-Sets match with $L=2Q+1$ match, so $L$ is an odd integer, and with ST parameter $K_0$ for the first $(L-1)$, and $K_1$ for the $L$th set. \RED{We shall denote by $\theta_M(p_A,p_B,K_0,K_1,Q)$ the probability that the first server will win the match.}
For the men's section in the US Open, $Q = 2$, while for the women's section, $Q=1$.

\begin{theorem}
    For a Best-of-$L$ match with $L = 2Q +1$ and with set tie-breaker parameters $(K_0,K_1)$, 
    \begin{eqnarray*}
    \theta_M(p_A,p_B,K_0,K_1,Q) & = &
    \sum_{h=0}^{Q-1} \Pr\{\mbox{$(Q+1,h)$ score} | (p_A,p_B,K_0)\} + \\ && \Pr\{\mbox{$(Q+1,Q)$ score} | (p_A,p_B,K_0,K_1)\},    
    \end{eqnarray*}
where, for $h \in \{0,1,\ldots,Q-1\}$,
\begin{eqnarray*}
    \lefteqn{\Pr\{\mbox{$(Q+1,h)$ score} | (p_A,p_B,K_0)\} \equiv \theta_M(Q+1,h,p_A,p_B,K_0,K_1,Q) } \\
    & = & \left\{
    \sum_{j=0}^Q {{t_A(Q+h)} \choose j} {{t_B(Q+h)} \choose {Q-j}} 
    [\theta_S(p_A,p_B,K_0)]^j [1 - \theta_S(p_A,p_B,K_0)]^{t_A(Q+h) - j} \times \right. \\ && \left.
    [1 - \theta_S(p_B,p_A,K_0)]^{Q-j} [\theta_S(p_B,p_A,K_0)]^{t_B(Q+h) - (Q-j)} \right\} \times \\
    && [\theta_S(p_A,p_B,K_0)]^{J_A(Q+1+h)} [1 - \theta_S(p_B,p_A,K_0)]^{1 - J_A(Q+1+h)};
\end{eqnarray*}
and
\begin{eqnarray*}
    \lefteqn{\Pr\{\mbox{$(Q+1,Q)$ score} | (p_A,p_B,K_0,K_1)\} \equiv \theta_M(Q+1,Q,p_A,p_B,K_0,K_1,Q)   } \\
    & = & \left\{
    \sum_{j=0}^Q {{Q} \choose j}^2 [\theta_S(p_A,p_B,K_0)\theta_S(p_B,p_A,K_0)]^j \times \right. \\
    && \left.
     [(1 - \theta_S(p_A,p_B,K_0))(1 - \theta_S(p_B,p_A,K_0))]^{Q - j}  \right\} \times \\
    && [\theta_S(p_A,p_B,K_1)]^{J_A(2Q+1)} [1 - \theta_S(p_B,p_A,K_1)]^{1 - J_A(2Q+1)}.
\end{eqnarray*}
Alternatively,
\begin{eqnarray*}
\lefteqn{\theta_M(p_A,p_B,K_0,K_1,Q)   } \\ & = &
\Pr\{B(Q,\theta_S(p_A,p_B,K_0)) +  B(Q,1-\theta_S(p_B,p_A,K_0)) \ge Q + 1\} +  \\ &&
\Pr\{B(Q,\theta_S(p_A,p_B,K_0)) + B(Q,1-\theta_S(p_B,p_A,K_0)) = Q\} \theta_S(p_A,p_B,K_1).
\end{eqnarray*}
\BLUE{In addition, $\theta_M(p_A,p_B,K_0,K_1,Q) = \theta_M(q_B,q_A,K_0,K_1,Q)$, $\mu_M(p_A,p_B,K_0,K_1,Q) = \mu_M(q_A,q_B,K_0,K_1,Q)$ and $\sigma_M^2(p_A,p_B,K_0,K_1,Q) = \sigma_M^2(q_A,q_B,K_0,K_1,Q)$.}
\end{theorem}

\begin{proof}
 The proof is very similar to that for obtaining the set probability. The only difference is that for sets $1$ to $(L-1) = 2Q$, the set tie-breaker parameter is $K_0$, while if the match reaches the final set, the $L$th set, then the set tie-breaker parameter is $K_1$. Also, note that when player $B$ is serving, the relevant set probability is  $\theta_S(p_B,p_A,K)$, so player $A$ will win this set with probability $[1 - \theta_S(p_B,p_A,K)]$.

 The alternative form for $\theta_M(p_A,p_B,K,Q)$ is obtained using the same arguments as in previous propositions, while the final identities follow directly from the alternative form and since $\theta_S(p_A,p_B,K) = \theta_S(q_B,q_A,K)$, $\mu_S(p_A,p_B,K) = \mu_S(q_A,q_B,K)$, and $\sigma_S^2(p_A,p_B,K) = \sigma_S^2(q_A,q_B,K)$.
\end{proof}

For the match, we could actually simplify the expressions since we have found out that it does not matter who serves first in a set in terms of winning a set. Let $(S_A,S_B)$ be the number of sets won by player $A$ and player $B$, respectively, in a match. For a best of $L = 2Q+1$ match, the match ends whenever $\{S_A = Q+1\}$ or $\{S_B = Q+1\}$. It is now easy to see, because of the afore-mentioned property above, that the joint probability mass function (JPMF) of $(S_A,S_B)$ is
\begin{eqnarray*}
\lefteqn{   p_{(S_A,S_B)}(a,b) \equiv \Pr\{S_A=a,S_B=b|p_A,p_B,K_0,K_1,Q\} } \\ & = &
    \left\{
    \begin{array}{lll}
    {{a+b} \choose a} \theta_S(p_A,p_B,K_0)^a (1 - \theta_S(p_A,p_B,K_0))^b \ \mbox{if}\ (a \le Q, b \le Q) \\
    p_{(S_A,S_B)}(Q,b\} \theta_S(p_A,p_B,K_0) \ \mbox{if}\ (a = Q+1, b < Q) \\
    p_{(S_A,S_B)}(Q,Q\} \theta_S(p_A,p_B,K_1) \ \mbox{if}\ (a = Q+1, b = Q) \\
        p_{(S_A,S_B)}(a,Q\} (1-\theta_S(p_A,p_B,K_0)) \ \mbox{if}\ (a < Q, b = Q+1) \\
    p_{(S_A,S_B)}(Q,Q\} (1-\theta_S(p_A,p_B,K_1)) \ \mbox{if}\ (a = Q, b = Q+1) \\
0 \ \mbox{otherwise}
    \end{array}
    \right..
\end{eqnarray*}
From this JPMF, we could obtain the probability that the first server wins the match via
$$\theta_M(p_A,p_B,K_0,K_1,Q) = \sum_{b=0}^Q p_{(S_A,S_B)}(Q+1,b|p_A,p_B,K_0,K_1,Q).$$
Also, note that the JPMF of $(S_A,S_B)$ provides the probabilities of the different match final scores in terms of the number of sets won by the players, which takes values in the set 
$$\{(Q+1,0),(Q+1,1),\ldots,(Q+1,Q),(0,Q+1),(1,Q+1),\ldots,(Q,Q+1)\}.$$

Figure \ref{fig: contour match best of five and three} presents the contour plots of the first server winning the match, for a best-of-5-sets match and for a best-of-3-sets match, respectively. Observe again that when $p_A = p_B = p \in (0,1)$, the probability of winning for the first server, hence the first receiver, is 1/2, so that {\em when the players are evenly matched, there is no advantage to serving first or receiving first}. Thus, the tennis match system is a fair system.

In addition, as to be expected, the match system is more efficient than the set system in determining the better player. Also, a best-of-five-sets match is more efficient, at least in determining the better player, than a best-of-three-sets match. Of course, one may argue that the increase in efficiency of a best-of-five-sets match compared to a best-of-three-sets match must be viewed also with respect to other considerations, such as the length of the match, the toll it takes on the players, and the level of excitement the match will bring to the spectators and fans.

\begin{figure}[h]
    \centering
\includegraphics[width=\mymultB\textwidth,height=\mymultB\textwidth]{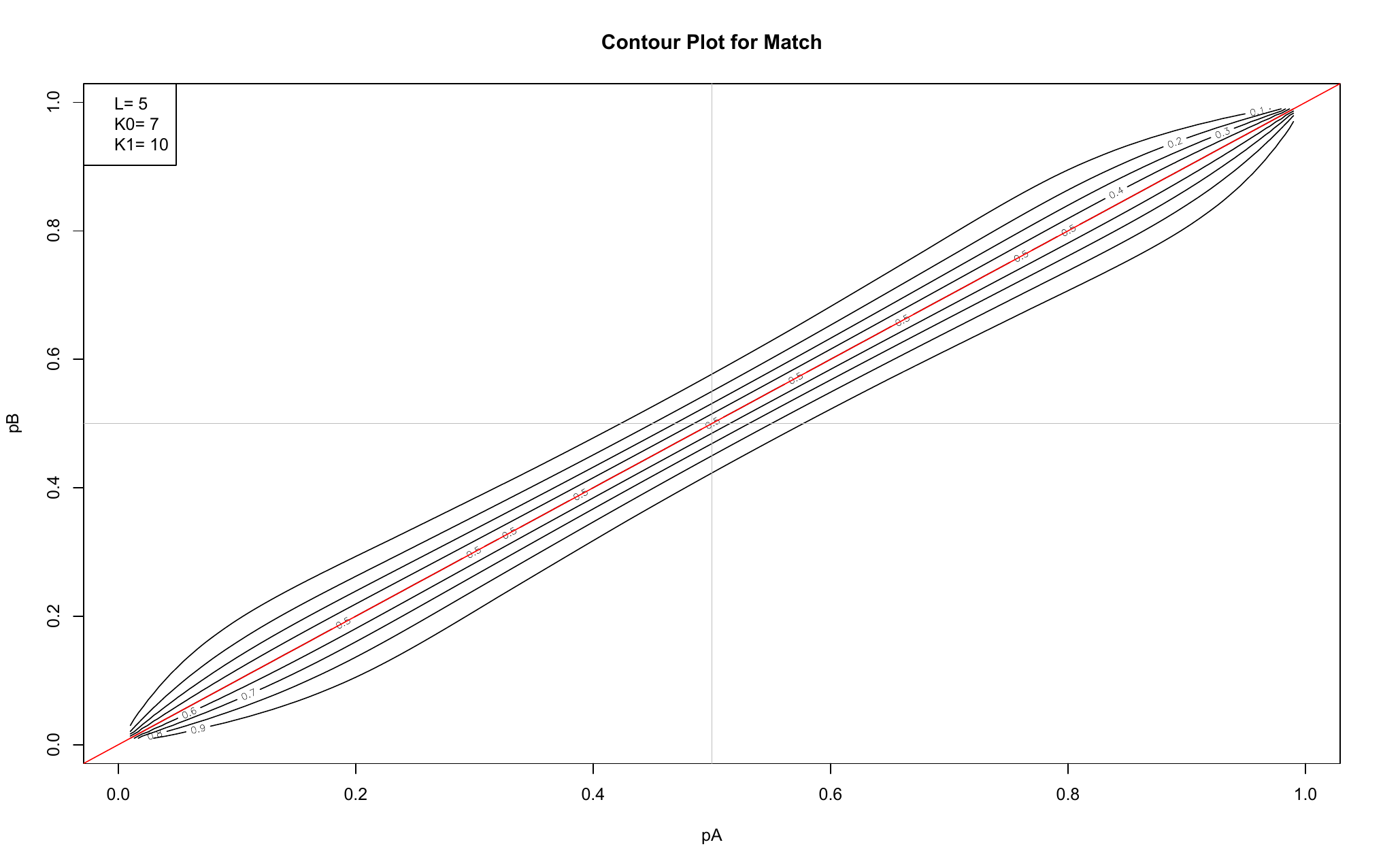} 
\includegraphics[width=\mymultB\textwidth,height=\mymultB\textwidth]{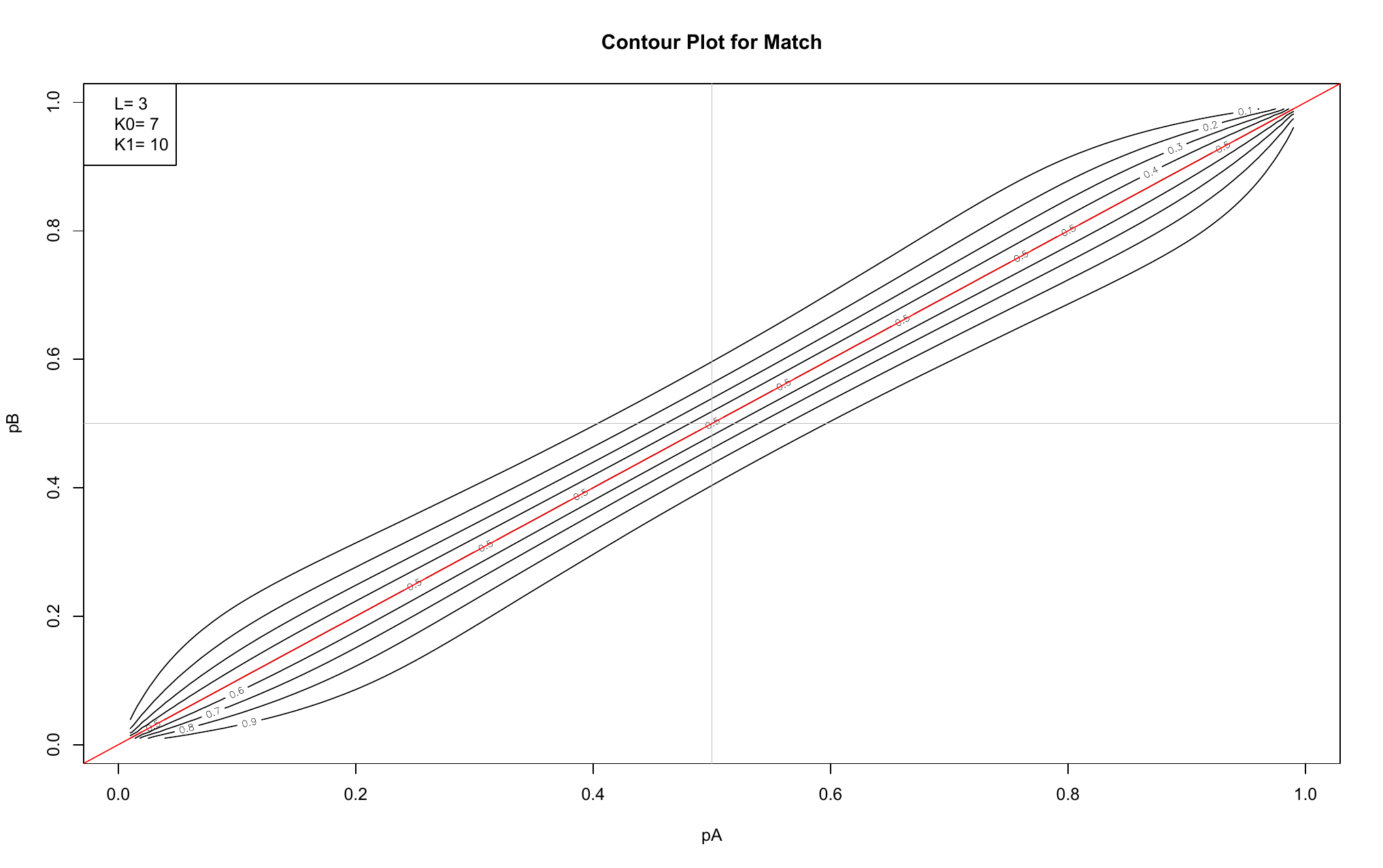}
    \caption{Contour plot of the probability of first server winning the {\bf best-of-five-sets} match (men's matches) and a {\bf best-of-three-sets} match (women's matches) when $K_0 = 7$ and $K_1 = 10$. $pA$ is the probability of first server winning point on his/her serve; while $pB$ is the probability of the first receiver winning point on his/her serve.}
    \label{fig: contour match best of five and three}
\end{figure}

\subsection{Number of Points in a Match}

Finally, we consider the number of points that will be played in a Best-of-$L$ match, with $L = 2Q+1$. \RED{Let us denote by $N_M$ the random variable denoting the total number of points that will be played for the match. Let $N_M(a,b)$ be the number of points in the match when the match score is $(S_A,S_B) = (a,b)$.} Then, we have the decomposition of $N_M$ given by
\begin{eqnarray*}
    N_M & = & \sum_{a = 0}^{Q-1} N_M(a,Q+1) I\{(S_A,S_B) = (a,Q+1)\} + 
    N_M(Q,Q+1) I\{(S_A,S_B) = (Q,Q+1)\} + \\
    && \sum_{b = 0}^{Q-1} N_M(Q+1,b) I\{(S_A,S_B) = (Q+1,b)\} + 
    N_M(\RED{Q+1,Q}) I\{(S_A,S_B) = (Q+1,Q)\}.
\end{eqnarray*}

We have the conditional means and conditional variances, \BLUE{assuming player $A$ is the first server in each of the sets,}
$$E\{N_M(a,Q+1)\} = (a+Q+1) \mu_S(p_A,p_B,K_0) \ \mbox{if} \ a < Q;$$
$$E\{N_M(Q,Q+1)\} = 2Q \mu_S(p_A,p_B,K_0) + \mu_S(p_A,p_B,K_1);$$
$$E\{N_M(Q+1,b)\} = (b+Q+1) \mu_S(p_A,p_B,K_0) \ \mbox{if} \ b < Q;$$
$$E\{N_M(Q+1,Q)\} = 2Q \mu_S(p_A,p_B,K_0) + \mu_S(p_A,p_B,K_1);$$
$$Var\{N_M(a,Q+1)\} = (a+Q+1) \sigma_S^2(p_A,p_B,K_0) \ \mbox{if} \ a < Q;$$
$$Var\{N_M(Q,Q+1)\} = 2Q \sigma_S^2(p_A,p_B,K_0) + \sigma_S^2(p_A,p_B,K_1);$$
$$Var\{N_M(Q+1,b)\} = (b+Q+1) \sigma_S^2(p_A,p_B,K_0) \ \mbox{if} \ b < Q;$$
$$Var\{N_M(Q+1,Q)\} = 2Q \sigma_S^2(p_A,p_B,K_0) + \sigma_S^2(p_A,p_B,K_1).$$
The conditional mean and variance of $N_M$, given $(S_A,S_B)$, is therefore
\begin{eqnarray*}
    E\{N_M|(S_A,S_B)\} & = & \sum_{a = 0}^{Q-1} E\{N_M(a,Q+1)\}  I\{(S_A,S_B) = (a,Q+1)\} + \\ && E\{N_M(Q,Q+1)\} I\{(S_A,S_B) = (Q,Q+1)\} + \\
    && \sum_{b = 0}^{Q-1} E\{N_M(Q+1,b)\} I\{(S_A,S_B) = (Q+1,b)\} + \\ && E\{N_M(Q+1,Q)\} I\{(S_A,S_B) = (Q+1,Q)\};
\end{eqnarray*}
\begin{eqnarray*}
    Var\{N_M|(S_A,S_B)\} & = & \sum_{a = 0}^{Q-1} Var\{N_M(a,Q+1)\}  I\{(S_A,S_B) = (a,Q+1)\} + \\ && Var\{N_M(Q,Q+1)\} I\{(S_A,S_B) = (Q,Q+1)\} + \\
    && \sum_{b = 0}^{Q-1} Var\{N_M(Q+1,b)\} I\{(S_A,S_B) = (Q+1,b)\} + \\ && Var\{N_M(Q+1,Q)\} I\{(S_A,S_B) = (Q+1,Q)\}.
\end{eqnarray*}
From these conditional means and variances, we are then able to obtain the mean and variance of $N_M$ via the iterated rules for the mean and variance (cf., Ross \cite{Ross1998}), where the conditioning is on $(S_A,S_B)$:
$$\mu_M(p_A,p_B,K_0,K_1,Q) = E\{N_M\} = E\{E[N_M|(S_A,S_B)]\};$$
\begin{eqnarray*}
\sigma_M^2(p_A,p_B,K_0,K_1,Q) & = & Var\{N_M\}  = E\{Var[N_M|(S_A,S_B)]\} + Var\{E[N_M|(S_A,S_B)]\}.    
\end{eqnarray*}
The standard deviation of $N_M(L)$ then obtains via $\sigma_M(p_A,p_B,K_0,K_1,Q) = +\sqrt{\sigma_M^2(p_A,p_B,K_0,K_1,Q)}.$

For the match, we could summarize the stochastic characteristics via Table \ref{tab: match statistics}. In this table, for $h = 0, 1, \ldots, Q,$
$$\theta_M(h) =  \theta_M(h,Q+1) + \theta_M(Q+1,h); \quad
\mu_M(h) = \mu_M(h,Q+1) = \mu_M(Q+1,h);$$  $$\sigma_M^2(h) = \sigma_M^2(h,Q+1) = \sigma_M^2(Q+1,h).$$

\begin{table}[h]
\begin{center}
\begin{tabular}{||c||c|c|c||c|c||} \hline
Loser's & \multicolumn{3}{c||}{Probability} & \multicolumn{2}{c||}{Points Played} \\ \cline{2-4} \cline{5-6}
Score & $A$ wins & $B$ wins & $A$ or $B$ wins & CMean & CVariance \\ \hline
$0$ & $\theta_M(Q+1,0)$ & $\theta_M(0,Q+1)$ & $\theta_M(0)$ & $\mu_M(0)$ & $\sigma_M^2(0)$ \\ 
$1$ & $\theta_M(Q+1,1)$ & $\theta_M(1,Q+1)$ &$\theta_M(1)$ & $\mu_M(1)$ & $\sigma_M^2(1)$ \\ 
\vdots & \vdots & \vdots & \vdots & \vdots & \vdots \\
$Q$ & $\theta_M(Q+1,Q)$ & $\theta_M(Q,Q+1)$ & $\theta_M(Q)$ & $\mu_M(Q)$ & $\sigma_M^2(Q)$ \\ \hline
Overall & $\theta_M$ & $1-\theta_M$ & 1 & $\mu_M$ & $\sigma_M^2$ \\ \hline\hline
\end{tabular}
\caption{Summary of Statistics for a Match given $(p_A,p_B,K_0,K_1,L=2Q+1)$, where player $A$ is assumed to serve first in each of the sets. In the table, $\theta_M(S_A,S_B,p_A,p_B,K_0,K_1,L)$ will be written as $\theta_M(S_A,S_B)$, and similarly for $\mu_M(S_L,p_A,p_B,K_0,K_1,L)$ and $\sigma_S(S_L,p_A,p_B,K_0,K_1,L)$.}
\label{tab: match statistics}
\end{center}
\end{table}

For a concrete illustration, Table \ref{tab: characteristics for match with (p_A,p_B,K_0,K_1,Q) = (.6,.55,7, 10,2)} presents the characteristics for a Best-of-5 Match when $(p_A,p_B) = (.60,.55)$. In particular, observe that the match system magnifies the chance of the better player winning the match. In this illustration, player $A$ has probability 0.60 of winning his service point, while player $B$ has probability 0.55 of winning his service point, but now for the whole match, player $A$ has probability of 0.795 of winning the match, while player $B$ only has a 0.205 chance of winning the match.

\begin{table}[h]
\begin{center}
\begin{tabular}{||c||c|c|c||c|c||} \hline
Loser's & \multicolumn{3}{c||}{Probability} & \multicolumn{2}{c||}{Points Played} \\ \cline{2-4} \cline{5-6}
Score in Set & $A$ wins & $B$ wins & $A$ or $B$ wins & CMean & CVariance \\ \hline\hline
      0& 0.300& 0.036& 0.336 & 193.056&  801.811 \\
       1 & 0.298 & 0.072 & 0.370 & 257.408 & 1069.082 \\
       2 & 0.197 & 0.097 & 0.293 & 322.488 &1378.232 \\ \hline
 Overall & 0.795 & 0.205 & 1.000 & 254.894 & 3700.152 \\ \hline\hline
\end{tabular}
\caption{Summary of Statistics for a Match given $(p_A,p_B,K_0,K_1,Q) = (.6,.55,7, 10,2)$, where player $A$ is assumed to serve first in each of the sets.}
\label{tab: characteristics for match with (p_A,p_B,K_0,K_1,Q) = (.6,.55,7, 10,2)}
\end{center}
\end{table}

\BLUE{However, the decomposition of $N_M$ used above is approximate since we are not keeping track of who serves first in each of the sets; instead we are assuming that player $A$ is always the first server in each of the sets. This may not be the case though since if there is a set-tiebreaker, then the first server on the set after the tie-breaker is the receiver in the set with the tie-breaker. Or, if the score $(g_A,g_B)$ in the current set is such that $g_A+g_B$ is odd, then the first server in the next set is the receiver in the current set. To finesse this we could average out the {\em two cases} where players $A$ and $B$ serves first in all the sets of the match, each with probability $1/2$. This leads to the approximate mean and variance of the number of points in the match given by
$$\mu_M \approx \frac{1}{2} \left[\mu_M(p_A,p_B,K_0,K_1,Q) + \mu_M(p_B,p_A,K_0,K_1,Q) \right];$$
$$\sigma_M^2 \approx \frac{1}{2} \left[ \sigma_M^2(p_A,p_B,K_0,K_1,Q) + \sigma_M^2(p_B,p_A,K_0,K_1,Q)\right] + \frac{1}{4}\left[\mu_M(p_A,p_B,K_0,K_1,Q) - \mu_M(p_B,p_A,K_0,K_1,Q)\right]^2. $$} 

Figure \ref{fig: mean plot points match pA=pB=p and contour plot} depicts a plot of the mean number of points in a Best-of-Five-Sets Match when the two players are of equal abilities, together with the two standard deviation bounds, as well as a contour plot. \BLUE{Here it is still assumed that player $A$ serves first in each of the sets.} Interestingly, when $p$ is in the interval $[.1,.9]$, the mean number of points does not fluctuate much, but past this interval, the mean number of points steeply increase, and this is because in each of the sets, tie-breakers have a high chance of happening, and in each of the set tie-breaker's tie-breaker, there could be very many points played (see Figure \ref{fig: Mean and Std Num Points Set Tie-Breaker}). Looking at the contour plot, observe the symmetrical property, with respect to the 45-degree line, and the saddle-shaped property, again with respect to the 45-degree line, of the mapping $(p_A,p_B) \mapsto \mu_M(p_A,p_B,K_0,K_1,Q)$. Compared to the contour plot for a set, depicted in Figure \ref{fig: mean plot points set pA=pB=p and contour}, this saddle has a longer seat portion.

\begin{figure}[h]
\begin{center}
\includegraphics[width=\mymultB\textwidth,height=\mymultB\textwidth]{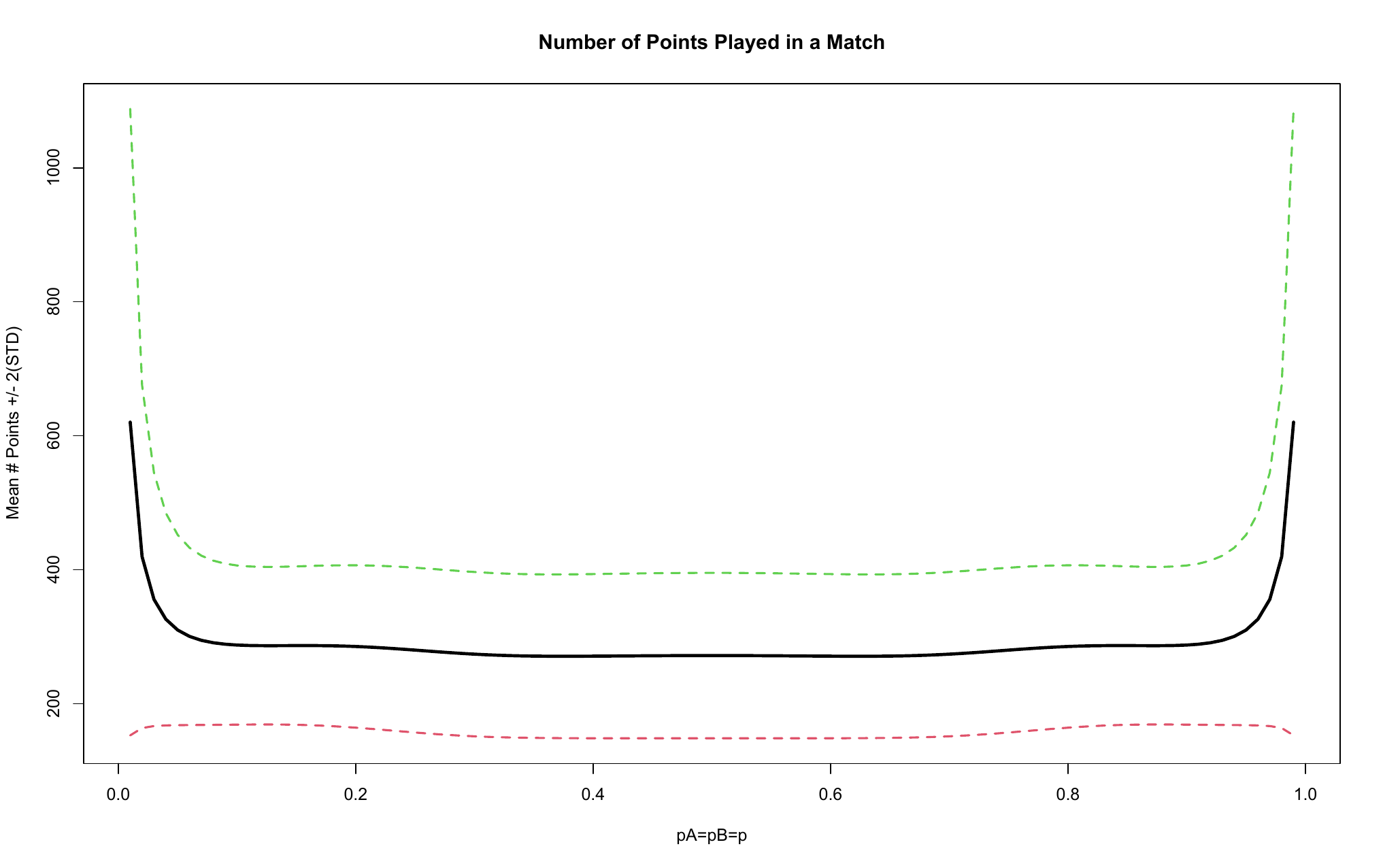} 
\includegraphics[width=\mymultB\textwidth,height=\mymultB\textwidth]{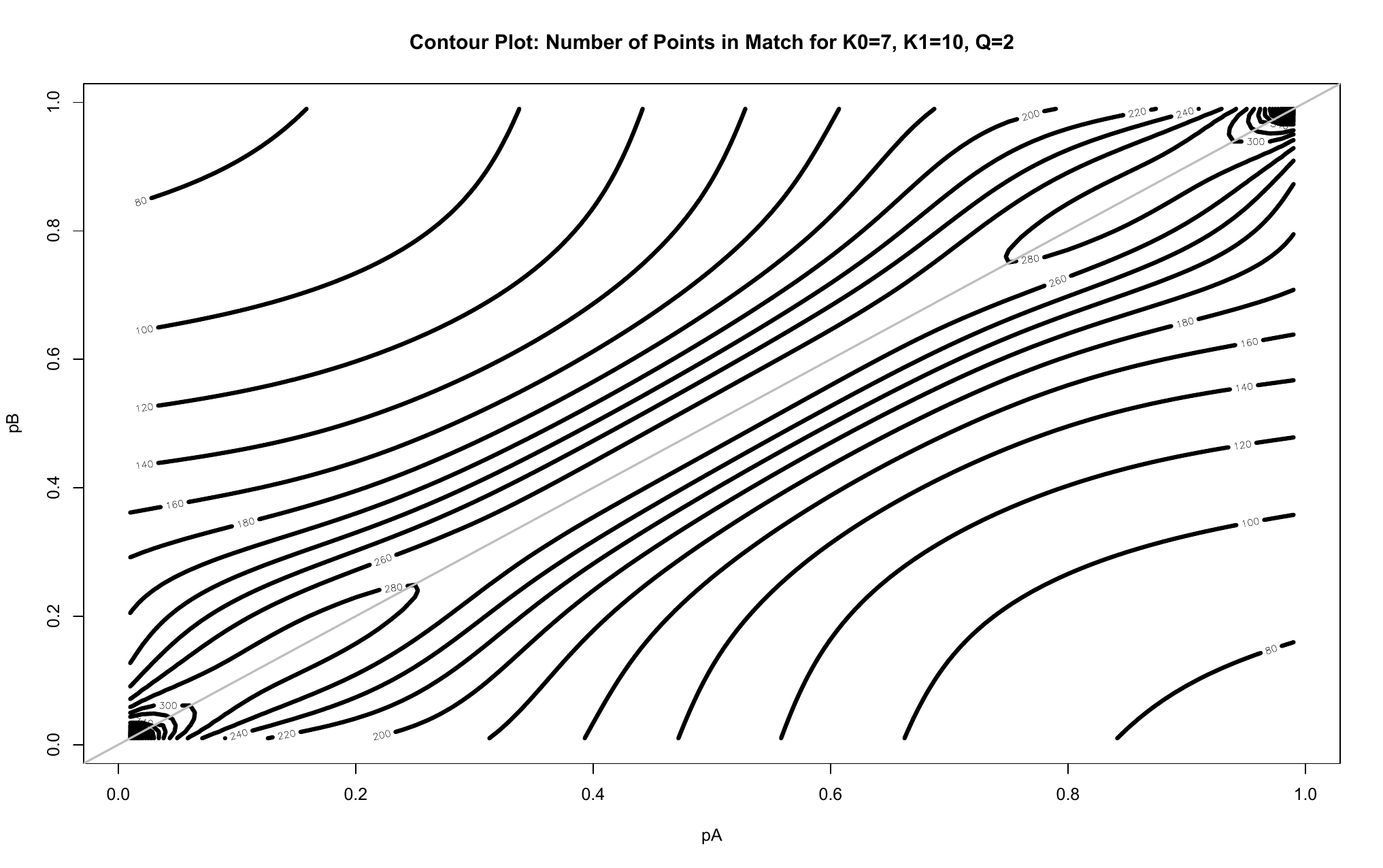}
\caption{Plot of mean number of points played in a Best-of-5 Match with $K_0=7$ and $K_1=10$ with respect to $p = p_A = p_B$, together with a contour plot.}
\label{fig: mean plot points match pA=pB=p and contour plot}
\end{center}
\end{figure}

\BLUE{

\subsection{Simulating Tennis Matches}

We also developed {\tt R} codes for performing simulations of a game tie-breaker, a game, a set-tiebreaker's tie-breaker, a set tie-breaker, a set, and finally a match. These simulations provide additional verifications of the theoretical results we obtained in earlier sections and subsections. It is highly satisfying to observe that the results of the simulation runs are very much in agreement with the theoretical results. We present the results for a game, set, and match, with 50000 simulation replications, for the case with $p_A=.60,p_B=.55,K_0=7,K_1=10,Q=2$. We consider two situations: first, when player $A$ serves first at the beginning of each stage; and second, when player $B$ serves first at the beginning of each stage. Table \ref{simulated and theoretical} summarizes the results. Observe the very strong agreement between the simulated and theoretical results.

}

\begin{table}[!htbp]
{\small 
\begin{center}
\begin{tabular}{||c||c||c|c||c|c||} \hline\hline
Stage & Score Profile & \multicolumn{2}{c||}{Simulated} &  \multicolumn{2}{c||}{Theoretical} \\  \cline{3-4} \cline{5-6}
& ($A$ score; $B$ score) & $A$ serves first & $B$ serves first & $A$ serves first & $B$ serves first \\ \hline\hline
& 4 -- 0 & .12824 & .09110 & .12960 & .09151\\
& 4 -- 1 & .20444 & .16662 & .20736 & .16471 \\
& 4 -- 2 & .20990 & .18536 & .20736 & .18530 \\
GAME & 0 -- 4 & .02696 & .04090 & .02560 & .04101 \\
& 1 -- 4 & .06152 & .08944 & .06144 & .09021 \\
& 2 -- 4 & .09280 & .12346 & .08507 & .12404 \\ 
& 3 --3 (GT) & .27614 & .30312 & .27648 & .30322 \\ \hline
& $\Pr\{\mbox{Server wins Game}\}$ & .73206 & .62194 & .73573 & .62315 \\ \hline
& MeanNumPoints & 6.48624 & 6.66822 & 6.48418 & 6.68092 \\ \hline
& StdDevNumPoints & 2.58379 & 2.70464 & 2.59003 & 2.72606 \\ \hline\hline
& 6 -- 0 & .02162 & .00420 & .02131 & .00447 \\
& 6 -- 1 & .09494 & .02850 & .09469 & .02829 \\
& 6 -- 2 & .09644 & .03296 & .09517 & .03357 \\
& 6 -- 3 & .20624 & .10286 & .20418 & .10272 \\
& 6 -- 4 & .10486 & .06044 & .10531 & .06255 \\
& 7 -- 5 & .06708 & .04130 & .06869 & .04079 \\
SET  &  0 -- 6 & .00472 & .02194 & .00447 & .02131 \\
($K=7$) & 1 -- 6 & .01180 & .04744 & .01199 & .04850 \\
& 2 -- 6 & .05126 & .14082 & .04987 & .14136 \\
& 3 -- 6 & .04324 & .10342 & .04356 & .10459 \\
& 4 -- 6 & .12084 & .20798 & .12170 & .20491 \\
& 5 -- 7 & .04082 & .06926 & .04079 & .06869 \\ 
& 6 -- 6 (ST) & .13614 & .13888 & .13825 & .13825 \\ \hline
& $\Pr\{\mbox{First Server wins Set}\}$ & .66946 & .32830 & .66948 & .33052 \\ \hline
& MeanNumPoints & 64.18378 & 64.95984 & 64.35217 & 64.87782 \\ \hline
& StdDevNumPoints & 16.27580 & 16.13323 & 16.34841 & 16.11616 \\ \hline\hline
& 3 -- 0 & .29814 & .03544 & .3001 & .0361\\
MATCH  & 3 -- 1 & .29964 & .07336 & .2975 & .0725 \\
($K_0=7,K_1=10$) & 3 -- 2 & .19694 & .09556 & .1972 & .0965 \\
($Q=2$) & 0 -- 3 & .03584 & .29860 & .0361 & .3001 \\
 & 1 -- 3 & .07184 & .29718 & .0725 & .2976 \\
& 2 -- 3 & .09760 & .19986 & .0965 & .1948 \\ \hline
& $\Pr\{\mbox{First Server wins Match}\}$ & .79472 & .20436 & .7948 & .2052 \\ \hline
& MeanNumPoints & 255.8984 & 256.3629 & 254.895 & 256.975 \\ \hline
& StdDevNumPoints & 62.6485 & 62.54117 & 60.829 & 60.934 \\ \hline\hline
\end{tabular}
\end{center}
}
\caption{Simulated and theoretical results for a Game, Set, and Match with $p_A=.60,p_B=.55,K_0=7,K_1=10,Q=2$ and 50000 replications. The overall probability is the probability of the first server winning the game, set, or match, while the probabilities for each of the score profiles are the probability of achieving that score profile.}
\label{simulated and theoretical}
\end{table}

\BLUE{Additional information are obtainable from the simulations. For instance, the theoretical distribution of the number of points played in a match is not easy to obtain in closed-form, but we are able to obtain an approximation of this distribution from the simulations. Figure \ref{fig: simulated dist of numpoints in match} presents the histogram of the number of points in a match when $p_A=.60,p_B=.55,K_0=7,K_1=10,Q=2$ and player $A$ serves first in the match.

\begin{figure}[h]
\begin{center}
\includegraphics[width=\mymultB\textwidth,height=\mymultB\textwidth]{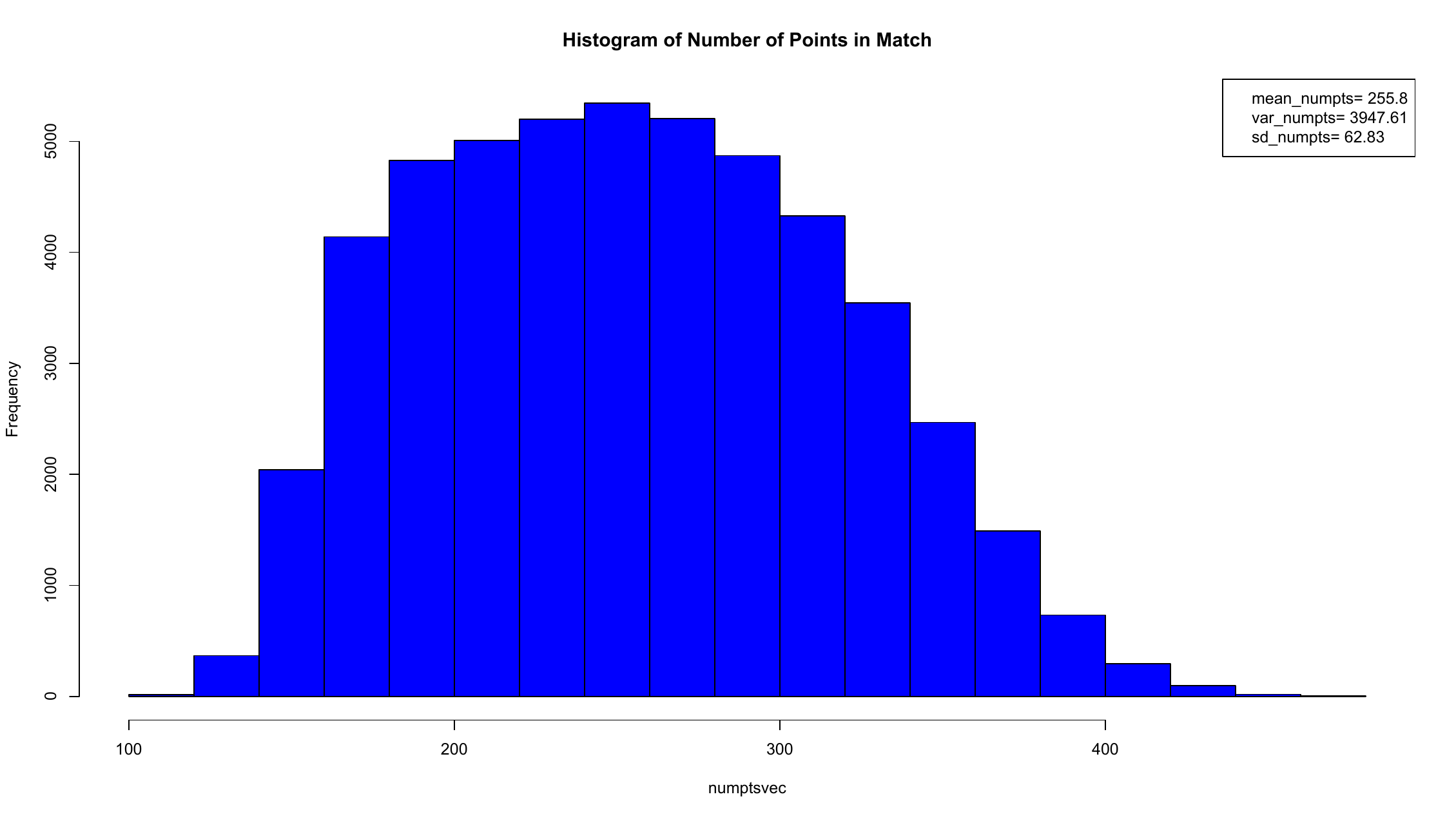}
\end{center}
\caption{Histogram of the number of points played in a match with $p_A=.6,p_B=.55,K_0=7,K_1=10,Q=2$, a best-of-five sets match, with player $A$ serving first, based on 50000 simulation replications.}
\label{fig: simulated dist of numpoints in match}
\end{figure}

}

\section{Efficiencies and Comparison of Systems}
\label{sec: efficiency and comparison}

\OLDPROOF{
\RED{maybe better to move the discussions of efficiency in a separate section (here).

Also,now thinking that maybe better efficiency measures are:

When there is one parameter, $p$, with $\theta(p)$, we could measure efficiency via, with respect to a prior $\pi(p)$:

$$Eff = \frac{\int_0^{1/2} (1/2 - \theta(p)) \pi(p) dp +
\int_{1/2}^1 (\theta(p)-1/2) \pi(p) dp}
{\int_0^{1/2} (1/2 - 0) \pi(p) dp +
\int_{1/2}^1 (1-1/2) \pi(p) dp}$$

while with two parameters $(p_A,p_B)$ with $\theta(p_A,p_B)$, to measure efficiency via:

$$Eff = \frac{\int_{p_A < p_B} (1/2 - \theta(p_A,p_B)) \pi(p_A,p_B) dp_A dp_B +
\int_{p_A>p_B} (\theta(p_A,p_B)-1/2) \pi(p_A,p_B) dp_A dp_B}
{\int_{p_A < p_B} (1/2 - 0) \pi(p_A,p_B) dp_A dp_B +
\int_{p_A>p_B} (1-1/2) \pi(p_A,p_B) dp_A dp_B}$$

In both cases, the ideal system will get a 100\% efficiency rating, while just a random system will get zero efficiency rating.

We could use the second measure to measure the efficiency of the set system and the match system.

Relative efficiency could then be measured as well by taking ratio.

We could compare the game system with Best-of-K systems.

We could compare the set and match systems relative to the STT system. We don't really have other systems to compare to, except Best-of-L systems.

}}

In essence, a tennis match is a statistical system to decide who is better between the two competing players. Indeed, Miles \cite{Miles1984} alluded to this aspect of scoring systems being statistical systems to determine the better player.  But how do we define who is the better player? One could argue that whoever wins the match is the better player, but this is a rather simplistic viewpoint. For one could have a biased coin whose probability of landing up `Head' on any given toss being 0.75, but when the coin is tossed 10 independent times, one could obtain 4 `Heads' and 6 `Tails' and conclude {\em wrongly} that the coin is biased in favor of `Tails'. Granted that the probability of the event of getting at least 6 tails out of 10 tosses for this coin is only 1.973\%, nonetheless there is a non-zero probability of concluding that the coin is loaded in favor of `Tails,' when in reality it is loaded in favor of `Heads'. Similarly, a better player in a tennis match could lose to the inferior player. But how do we define `better'?

Consider again the coin mentioned above. We say that it is loaded in favor of `Head' because its probability of landing up a `Head' in one toss is 0.75, which is greater than 0.50. But what is this value of 0.75? From the frequentist definition of probability, justified by the Strong Law of Large Numbers (SLLN), it is the limit of the proportion ($p_n$) of the occurrence of Heads after ($n$) tosses, as we keep tossing the coin ($n \rightarrow \infty$):
$$\lim_{n\rightarrow\infty} p_n = \lim_{n\rightarrow\infty}\left[\frac{\mbox{Number of Heads After $n$ Tosses}}{n}\right] = 0.75.$$
As such, the `fairness' or `loadedness' of a coin is defined with respect to a hypothetical experiment of tossing the coin indefinitely.

We adopt a similar idea to define who is a better player between two competing players. We could start by defining that player $A$ is better than player $B$ if the proportion of times that player $A$ wins the match over a hypothetical infinite number of matches between the two players is greater than $1/2$. However, instead of basing our definition on the outcome of a match, we instead focus on the more elemental outcome of a point played between the two players. But there is an asymmetry in tennis in that the player serving the point will usually have an advantage. Thus, we focus on the probability of each player winning the point on their serve, which are the two parameters $p_A$ and $p_B$. Thus, $p_A$ ($p_B$) could be viewed as the limiting proportion of player $A$ ($B$) winning the point on his/her serve if points are played indefinitely. We could now then say that player $A$ is better than player $B$ if $p_A > p_B$; player $A$ is inferior to player $B$ if $p_A < p_B$; and they are of equal abilities if $p_A = p_B$. A tennis match, when stripped of its other elements such as players competitiveness, fan excitement, and sponsor commercialization, is therefore just a statistical system to decide whether $p_A > p_B$ or $p_A < p_B$, statistical in the sense that the decision will be based on a finite, instead of an infinite, number of points played. Being a statistical system, analogously to the coin experiment, there is the possibility that player $A$ will lose {\em even} if it is the case that $p_A > p_B$, or player $A$ will win {\em even} if it is the case that $p_A < p_B$.
One could therefore ask how `good' these statistical systems are for determining who the better player is between two competitors, which is a question of the statistical {\em efficiencies} of systems. Miles \cite{Miles1984} formulated these questions in terms of statistical hypothesis testing, and thereby examined the efficiency of scoring systems in terms of the probabilities of statistical Type I and Type II errors.  

\subsection{One-Parameter Systems}

We begin by first considering one-parameter systems, such as the tennis game tie-breaker system, the tennis game system, or a Best-of-$K$ system, where there is a probability $p$ of winning the point, and we would like to determine the efficiency of a system whose goal is to determine if $p > 1/2$ or if $p < 1/2$. This is the case, for instance, with the tennis game tie-breaker system, where we have the  mapping $p \mapsto \theta_{GT}(p)$, or the tennis game system where we have the mapping $p \mapsto \theta_G(p)$. {\em Ideally}, we would like a system, called an oracle system (OS), such that when $p > 1/2$, then $\theta_{OS}(p) = 1$; when $p < 1/2$, then $\theta_{OS}(p) = 0$; and when $p = 1/2$, then $\theta_{OS}(1/2) = 1/2$. To measure the efficiency of a system, we could compare it with this oracle system over a distribution $\Pi(\cdot)$ of possible values of $p$.

\begin{definition}
    Let $\mbox{SYS}$ be a one-parameter system such that $\theta_{{SYS}}(p)$ is the probability that the player with probability $p$ of winning a point will win the game according to this system. The efficiency of this system with respect to a probability distribution $\Pi(\cdot)$ over the possible values of $p$ is
    \begin{eqnarray*}
    \mbox{Eff}({\mbox{SYS}};\Pi) & = &
    \frac{\int_0^{1/2} [(1/2) - \theta_{SYS}(p)] \Pi(dp) + \int_{1/2}^1 [\theta_{SYS}(p) - (1/2)] \Pi(dp)}{\int_0^{1/2} [(1/2) - 0] \Pi(dp) + \int_{1/2}^1 [1 - (1/2)] \Pi(dp)}
   \\
    & = & \int_0^{1/2} [1 - 2\theta_{SYS}(p)] \Pi(dp) + \int_{1/2}^1 [2\theta_{SYS}(p) - 1] \Pi(dp).
     \end{eqnarray*} 
Thus, $\mbox{Eff}(\mbox{OS};\Pi) = 1$; while a totally random system (RAN) with $\theta_{{RAN}}(p) = 1/2$ will have $\mbox{Eff}(\mbox{RAN};\Pi) = 0$.   
\end{definition}

The distribution $\Pi(\cdot)$ represents either our prior knowledge about the value of $p$, or it could also represent the mechanism that determines the value of $p$ at any given point. It could be a degenerate distribution, but a generally flexible class of distributions would be the beta family of probability distributions. For top professional tennis players, their $\Pi(\cdot)$ will generally be beta distributions with parameters $(\alpha,\beta)$ and with $ \alpha > \beta$. For complicated mappings $p \mapsto \theta(p)$, the efficiency value could be calculated using numerical integration techniques, which is what we did in our numerical examples.

To assess the efficiencies of the game tie-breaker and the game systems, we shall compare them with the more traditional Best-of-$K$ (Bof$K$) system, where $K = 2L+1$. Best-of-$K$ systems are used, for instance, in the Baseball World Series or the NBA Championship Series. It is a system where the winner is the first to get $L+1$ wins. If both players or teams reach $L$ wins, then the rubber match, which is the $K=2L+1$ match can be viewed as a tie-breaker. Thus, in baseball and basketball where it is a Best-of-7 system, $L=3$, so the team to achieve 4 wins is declared the winner. For a Best-of-$(2L+1)$ system, we have
$$\theta_{\mathrm{Bof}K}(p;K=2L+1) = \sum_{j=0}^L {{(L+1)+j-1} \choose L} p^{L+1} q^j.$$

Table \ref{tab: efficiencies of one-parameter systems} presents the efficiency measures of the tennis game tie-breaker system, tennis game system, and Best-of-$K=2L+1$ for $L=3,4,5$, when the distribution over the values of $p$ is a beta distribution with $(\alpha,\beta)$ parameter vector. Figure \ref{fig: comparison of game systems} depicts the probabilities for each of these systems.

\begin{table}[h]
\begin{center}
\begin{tabular}{||c||c|c|c|c|c|c||} \hline\hline
$\alpha$  &  0.5 & 0.5 & 1 & 1 & 2 & 3 \\ \hline
$\beta$   &  0.5 & 1 & 1 & 2 & 1 & 1 \\ \hline\hline
GameGT  & 0.7935 & 0.7738 & 0.6931 & 0.6931 & 0.6931 & 0.7500 \\
Game  &   0.8378 & 0.8214 & 0.7537 & 0.7537 & 0.7537 & 0.8046 \\
Best-Of-7 & 0.8188 & 0.8008 & 0.7266 & 0.7265 & 0.7265 & 0.7812 \\
Best-Of-9 & 0.8382 & 0.8217 & 0.7539 & 0.7539 & 0.7539 & 0.8051 \\
Best-Of-11 & 0.8524 & 0.8372 & 0.7744 & 0.7744 & 0.7744 & 0.8227 \\ \hline\hline
\end{tabular}
\caption{Efficiencies of one-parameter systems, under a beta$(\alpha,\beta)$ distribution over $p$.}
\label{tab: efficiencies of one-parameter systems}
\end{center}
\end{table}

\begin{figure}[h]
    \centering
    \includegraphics[width=\mymultB\textwidth,height=\mymultB\textwidth]{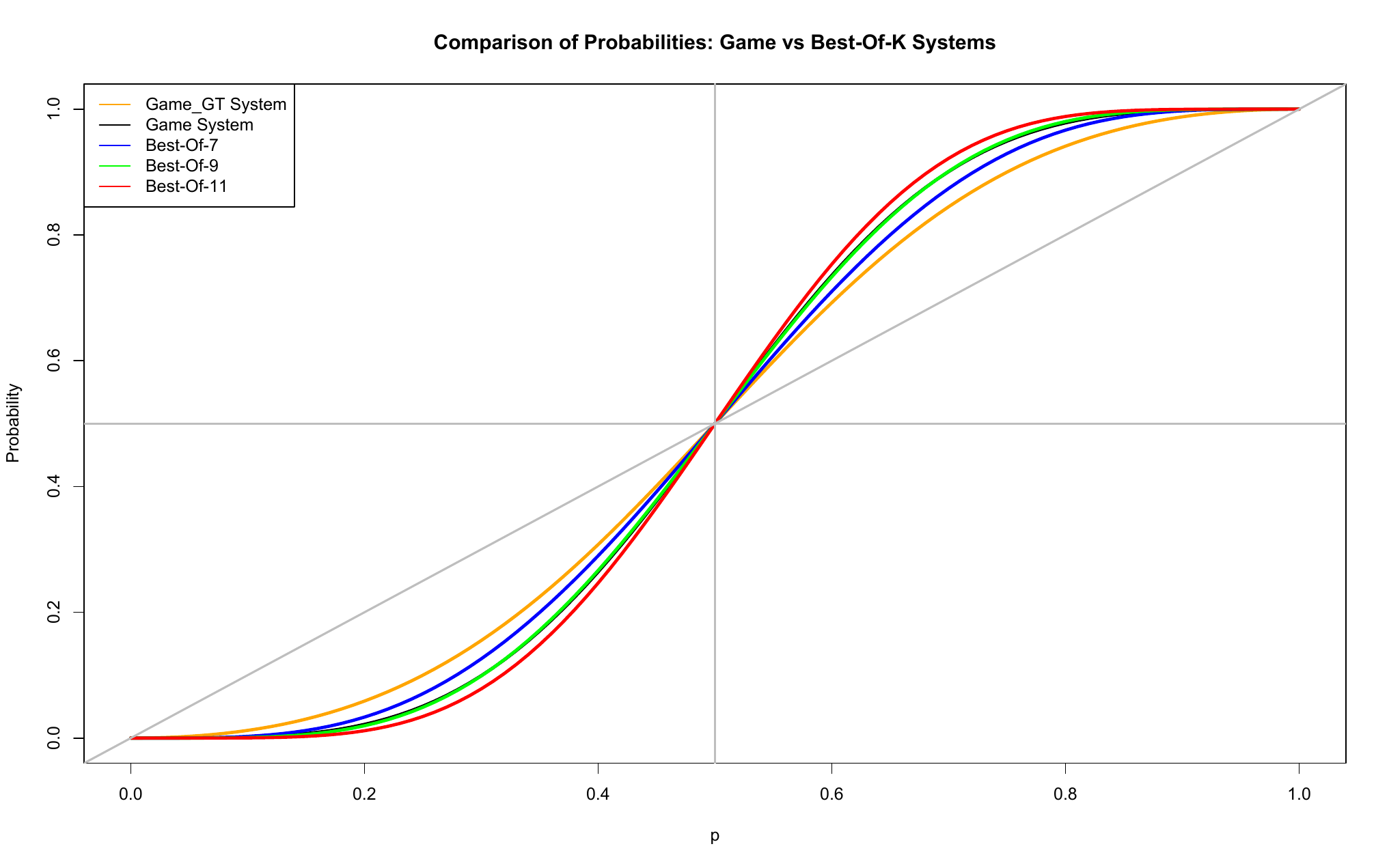}
    \caption{Plot of the probability of server winning the game with respect to server's probability of winning point under the game tennis tie-breaker system (ORANGE), game tennis system (BLACK) and with a best-of-7 (BLUE), best-of-9 (GREEN), and best-of-11 (RED) systems.}
    \label{fig: comparison of game systems}
\end{figure}

From these table and figure, observe that the Best-of-$K$ system that performs closest to the tennis game system is the Best-of-9 system. Best-of-7 system is inferior to the game system, though better than just the game tie-breaker system, and this is shown analytically in Proposition \ref{prop: Best_of-7 System inferior to Game System}. Best-of-11 system dominates the other four systems, though this will also require playing more points.

\begin{proposition}
\label{prop: Best_of-7 System inferior to Game System}
    Let $\theta_G(p)$ denote the probability that the server wins the game in the tennis system (win by at least two points), and the corresponding probability in the best-of-seven system is 
    $$\theta_{\mathrm{Bof7}}(p) = p^4 + 4p^4q + 10p^4q^2 + 20p^4q^3.$$
    The difference in the probabilities between these two systems is
    \[ 
    \theta_G(p) - \theta_{\mathrm{Bof7}}(p) = 20p^3q^3\left( \theta_{GT}(p) - p \right), \qquad \text{where } \theta_{GT}(p) = \frac{p^2}{p^2+q^2}.
    \]
   Thus,
    \[
    \theta_G(p)
    \begin{cases}
    > \theta_{\mathrm{Bof7}}(p), & p \in (1/2, 1),\\
    = \theta_{\mathrm{Bof7}}(p), & p \in \{0, 1/2, 1\},\\
    < \theta_{\mathrm{Bof7}}(p), & p \in (0, 1/2).\\
    \end{cases}
    \]
\end{proposition}
\begin{proof}
    By Theorem~\ref{thm: prob wins game}, $\theta_G(p) = p^4 + 4p^4 q + 10p^4 q^2 + 20p^3 q^3 \theta_{GT}(p).$ We also have the probability that the server wins the game in the best-of-seven system
    \[
        \theta_{\mathrm{Bof7}}(p) = \sum_{j=0}^3 {{4+j-1} \choose 3} p^4 q^j = p^4 + 4p^4q + 10p^4q^2 + 20p^4q^3.
    \]
    Subtracting gives $\theta_G(p) - \theta_{\mathrm{Bof7}}(p) = 20p^3q^3(\theta_{GT}(p)-p)$. Apparently, this difference is equal to zero when $p = 0$, $q = 1-p = 0$. 
    For $p \in (0,1)$, we have
    \[
    \theta_{GT}(p)-p = \frac{p^2}{p^2+q^2}-p
    = \frac{p^2-p(p^2+(1-p)^2)}{p^2+q^2}
    =  \frac{p(1-p)(2p-1)}{p^2+q^2}.
    \]
    Note that the sign of $\theta_{GT}(p)-p$ matches the sign of $(2p-1)$, establishing the three cases stated in the proposition.
\end{proof}

As alluded to above, we could compare the number of points played in these systems. For a Best-of-$K$ system  the number of points played is always bounded by $K$, in contrast to the tennis game system. Denoting the number of points needed to be played in a Best-of-$K$ system by $N_{BofK}$, we have for $K = 2L+1$,
$$\Pr\{N_{BofK} = n\} = {{n-1} \choose L} \left[p^{L+1} q^{n-1-L} + q^{L+1} p^{n-1-L}\right]$$
for $n = L+1,L+2,\ldots,K$; 0 otherwise.
Again, the mean ($\mu_{BofK}$), variance ($\sigma_{BofK}^2$), and standard deviation ($\sigma_{BofK}$) of $N_{BofK}$ could be obtained from this PMF. For a Best-of-9 system and under $p =.5$, the plot of this PMF is the right panel in Figure \ref{fig: PMF Num Points Game} (the left panel is the PMF for the tennis game system). Its mean is 7.5390, its variance is 1.4828, and its standard deviation is 1.2177. It is interesting to note that even though the number of points in a Best-of-9 system is bounded by 9, whereas for the tennis system there is no upper bound, on average, when $p = .5$, the tennis system requires fewer number of points to be played, though as to be expected, the tennis system has more variability.

\begin{figure}[!htb]
\begin{center}
\includegraphics[width=\mymultB\textwidth,height=\mymultB\textwidth]{PMF_NumPoints_Game_p5.pdf} 
\includegraphics[width=\mymultB\textwidth,height=\mymultB\textwidth]{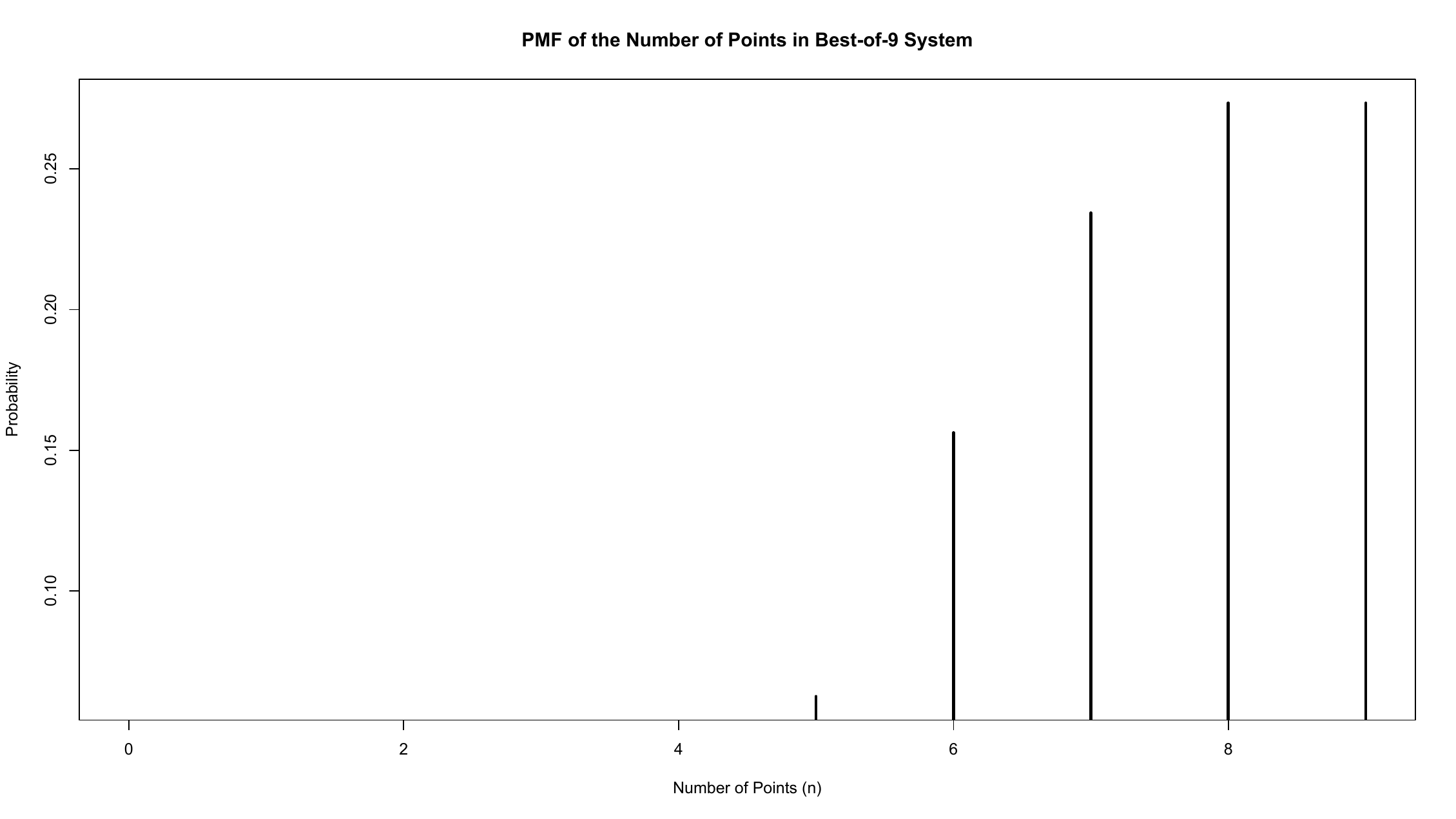}
\caption{Probability distribution of the Number of Points Needed to end a Game when $p = 0.50$ under the tennis game system (top) and for a best-of-9 system (right).}
\label{fig: PMF Num Points Game}
\end{center}
\end{figure}

In Figure \ref{fig: comparison of expected game length}, we plotted the means and standard deviations of the number of points required to end the game under the tennis Game system, a Best-of-7 system, a Best-of-9 system, and a Best-of-11 system, for different values of $p$. Comparing the Best-of-9 system with the tennis Game system, which are the two systems with comparable probabilistic performances, the former always have higher mean game length compared to the latter system. With respect to the variability, observe that the tennis Game system usually have higher standard deviation than the best-of-$K$ system for $K \in \{7,9,11\}$. This is intuitively expected since in the tennis Game system, the number of points needed to end the game is unbounded, whereas for a best-of-$K$ system, the game's length is bounded above by $K$ points. Of course, it should be noted that the afore-mentioned {\em intuitive expectation} could be misleading, since it does {\em not} imply that unboundedness of the range of a variable means that its variance is larger than that of a bounded variable.

\begin{figure}[h]
    \centering
    \includegraphics[width=\mymultB\textwidth,height=\mymultB\textwidth]{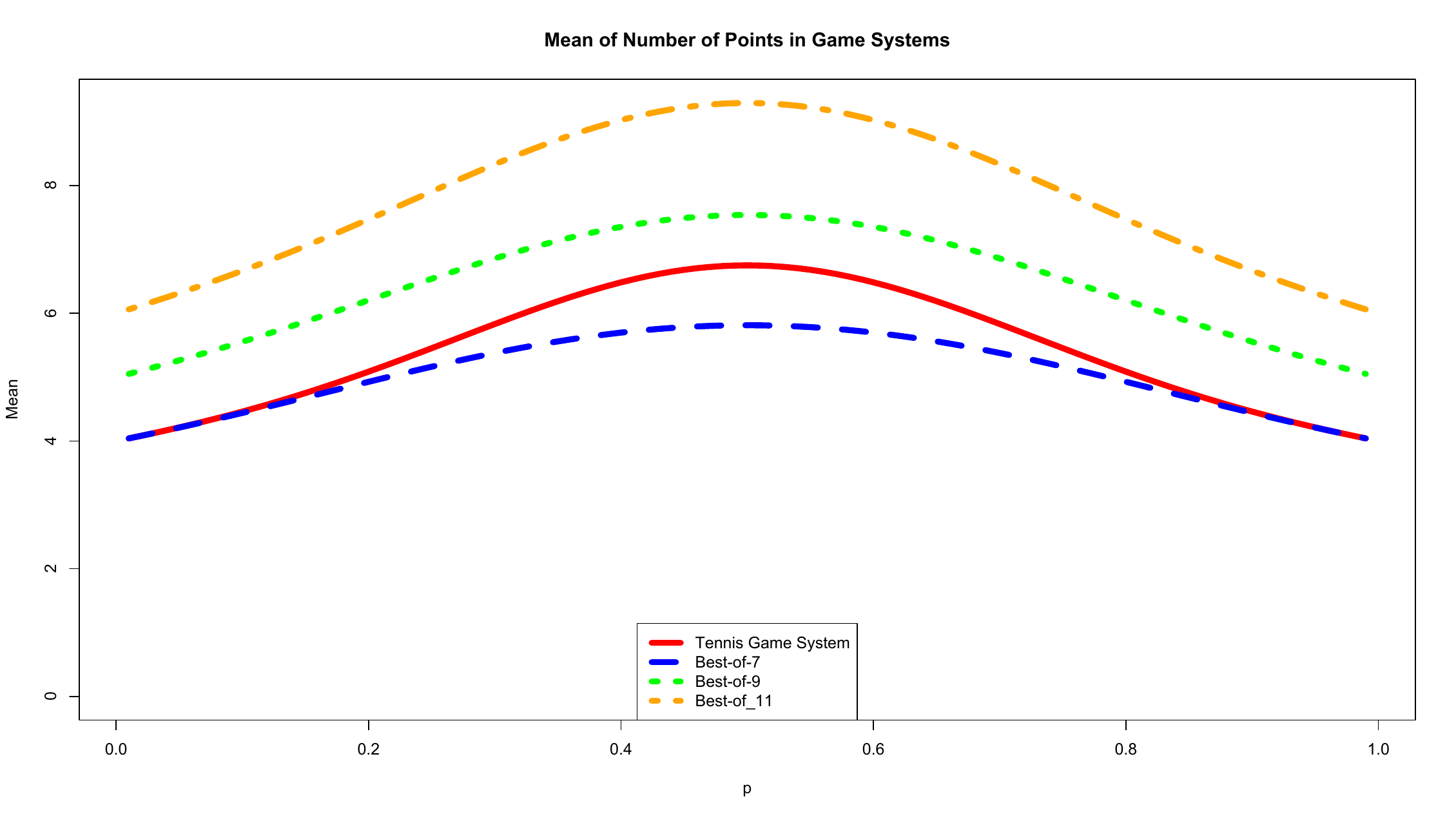} 
        \includegraphics[width=\mymultB\textwidth,height=\mymultB\textwidth]{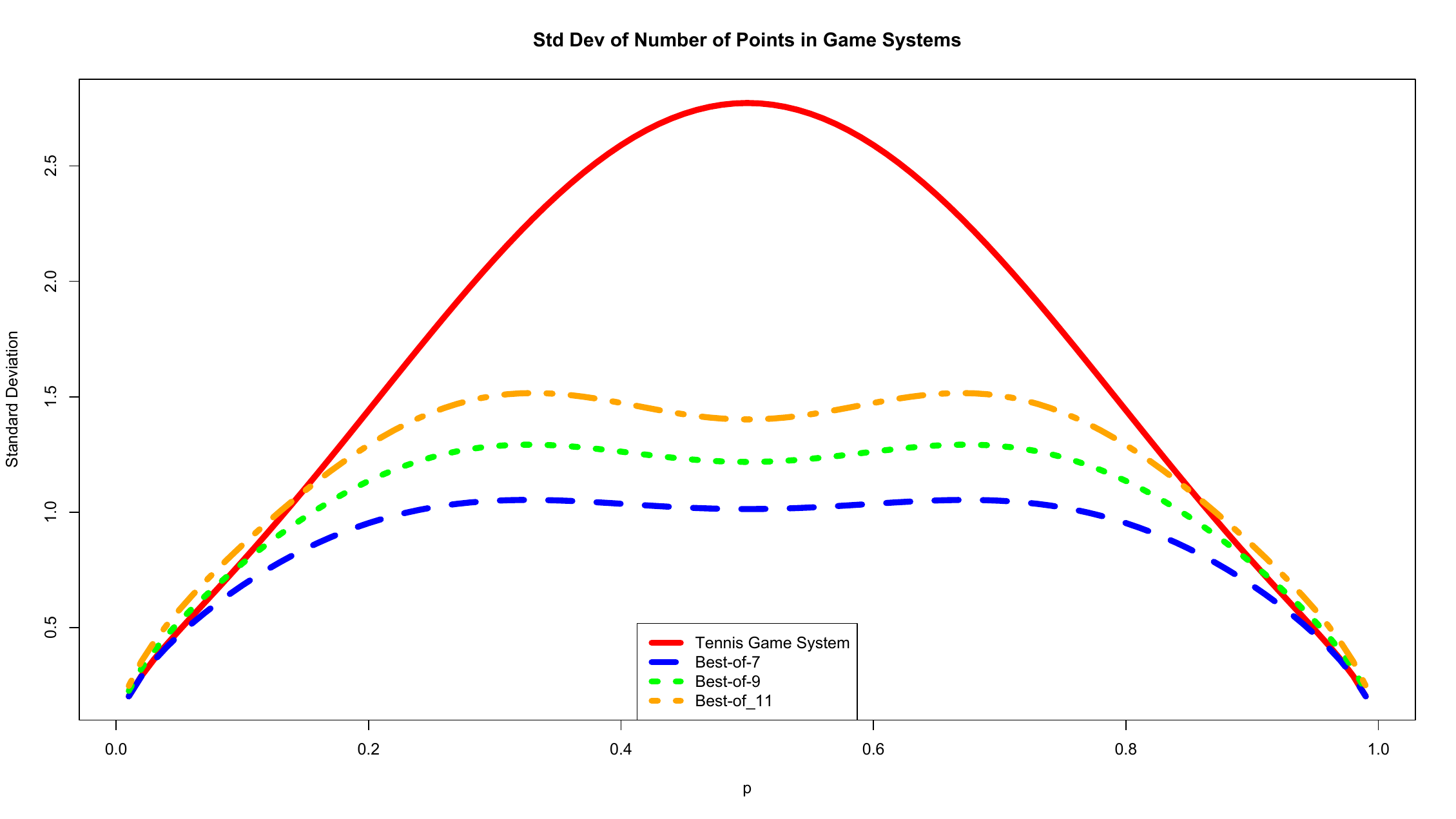}
    \caption{Plot of the comparison of the means (top) and standard deviations (bottom) of the number of points played in a Game for the tennis game system (RED) and the best-of-7 (BLUE), best-of-9 (GREEN), and best-of-11 (ORANGE) systems.}
    \label{fig: comparison of expected game length}
\end{figure}

\subsection{Two-Parameter Systems}
\label{subsec: two-parameter systems}

Next, we consider the efficiencies of two-parameter systems, such as the set tie-breaker's tie-breaker (STT), set tie-breaker (ST), set (S), and match (M) systems, which now involve the two parameters: $p_A$ and $p_B$. Thus, we now have a mapping: $(p_A,p_B) \mapsto \theta_{SYS}(p_A,p_B)$. How do we measure the efficiency of such a system with respect to a joint distribution $\Pi(p_A,p_B)$ over the region $[0,1] \times [0,1]$? Recall that $\theta_{SYS}(p_A,p_B)$ is the probability that player $A$ wins the STT, ST, S, or M, depending on which stage in the tennis match we are considering. 
We extend our definition of efficiency for  one-parameter systems to two-parameter systems. Given a $(p_A,p_B)$, we had said that the ideal or oracle system should have $\theta_{OR}(p_A,p_B) = 0, 1/2, 1$ depending on whether $p_A < p_B, p_A = p_B, p_A > p_B$, respectively. As such, we could measure the efficiency of a system SYS by comparing it with the oracle system.

\begin{definition}
    Let $\mbox{SYS}$ be a two-parameter system such that $\theta_{SYS}(p_A,p_B)$ is the probability that player $A$ (\RED{$B$}) with probability $p_A$ ($p_B$) of winning his/her service point will win the `match' according to this system. The efficiency of this system with respect to a joint probability distribution $\Pi(\cdot,\cdot)$ over $[0,1] \times [0,1]$ is
    \begin{eqnarray*}
    \mbox{Eff}(\mbox{SYS};\Pi) 
    & = & \int_{\{p_A < p_B\}} [1 - 2\theta_{SYS}(p_A,p_B)] \Pi(dp_A,dp_B) + \\ && 
    \int_{\{p_A > p_B\}} [2\theta_{SYS}(p_A,p_B) - 1] \Pi(dp_A,dp_B).
     \end{eqnarray*} 
    %
\end{definition}

For purposes of comparing these tennis systems to a non-tennis system we consider a Best-of-$K$ system, with $K=2L+1$, where the first player to win $(L+1)$ \RED{games} wins the match. The players will be assumed to alternate serving the games, and each game will have the usual game tie-breaker. Three systems will be considered when a tied-score after $2L$ games ensues.

For the first system (Bof$K$1 or Bof$K$:SG, with `SG' indicating `sudden game'), {\em one} game will be played, and who serves this deciding game will be determined at random, say, via a coin flip.
For the second system (Bof$K$2 or Bof$K$:STTG), an STT-type tie-breaker will ensue, except that {\em games} are played instead of {\em points}, and the player to first get an advantage of two games wins the match. 
For the third system (Bof$K$3 or Bof$K$:STTP), an STT-type tie-breaker will ensue, except that {\em points} are played instead of {\em games}, and the player to first get an advantage of two points wins the match. 

The respective probabilities of player $A$, the first server, of winning the match, under these three systems are as follows:

\begin{eqnarray*}
\lefteqn{ \theta_{BofK1}(p_A,p_B)  =  
 \Pr\{B(L,\theta_G(p_A)) + B(L,1-\theta_G(p_B)) \ge L+1\} + } \\ &&
 \Pr\{B(L,\theta_G(p_A)) + B(L,1-\theta_G(p_B)) = L\}
 \frac{1}{2} \left[\theta_G(p_A) + (1-\theta_G(p_B))\right];
\end{eqnarray*}
\begin{eqnarray*}
 \lefteqn{ \theta_{BofK2}(p_A,p_B)  =  
 \Pr\{B(L,\theta_G(p_A)) + B(L,1-\theta_G(p_B)) \ge L+1\} + }\\ &&
 \Pr\{B(L,\theta_G(p_A)) + B(L,1-\theta_G(p_B)) = L\}
 \theta_{STT}(\theta_G(p_A),\theta_G(p_B));
\end{eqnarray*}
\begin{eqnarray*}
\lefteqn{ \theta_{BofK3}(p_A,p_B)  =  
 \Pr\{B(L,\theta_G(p_A)) + B(L,1-\theta_G(p_B)) \ge L+1\} + } \\ &&
 \Pr\{B(L,\theta_G(p_A)) + B(L,1-\theta_G(p_B)) = L\}
 \theta_{STT}(p_A,p_B).
\end{eqnarray*}

Figure \ref{fig: prob wrt pB of systems} presents the plots of the probabilities of player A winning when $p_A=.6$ (top panel) and $p_A=.8$ (bottom panel), and as $p_B$ varies, when $L = \RED{22}$, that is, for Best-of-\RED{45} systems. Observe that when $p_A = .80$, the probability under Best-of-$(2L+1)$:1 as $p_B$ approaches 1 does not converge to 0. This is so since when $p_A$ and $p_B$ are both high, then there is a non-ignorable probability that the tie-breaker will ensue. However, in the first system, {\em only} one game is played, with the server chosen randomly, and so this single game becomes most critical. As such, this Best-of-$(2L+1):SG$ system is not a good system, especially when the two players are excellent servers.

\begin{figure}[h]
\begin{center}
\includegraphics[width=\mymultB\textwidth,height=\mymultB\textwidth]{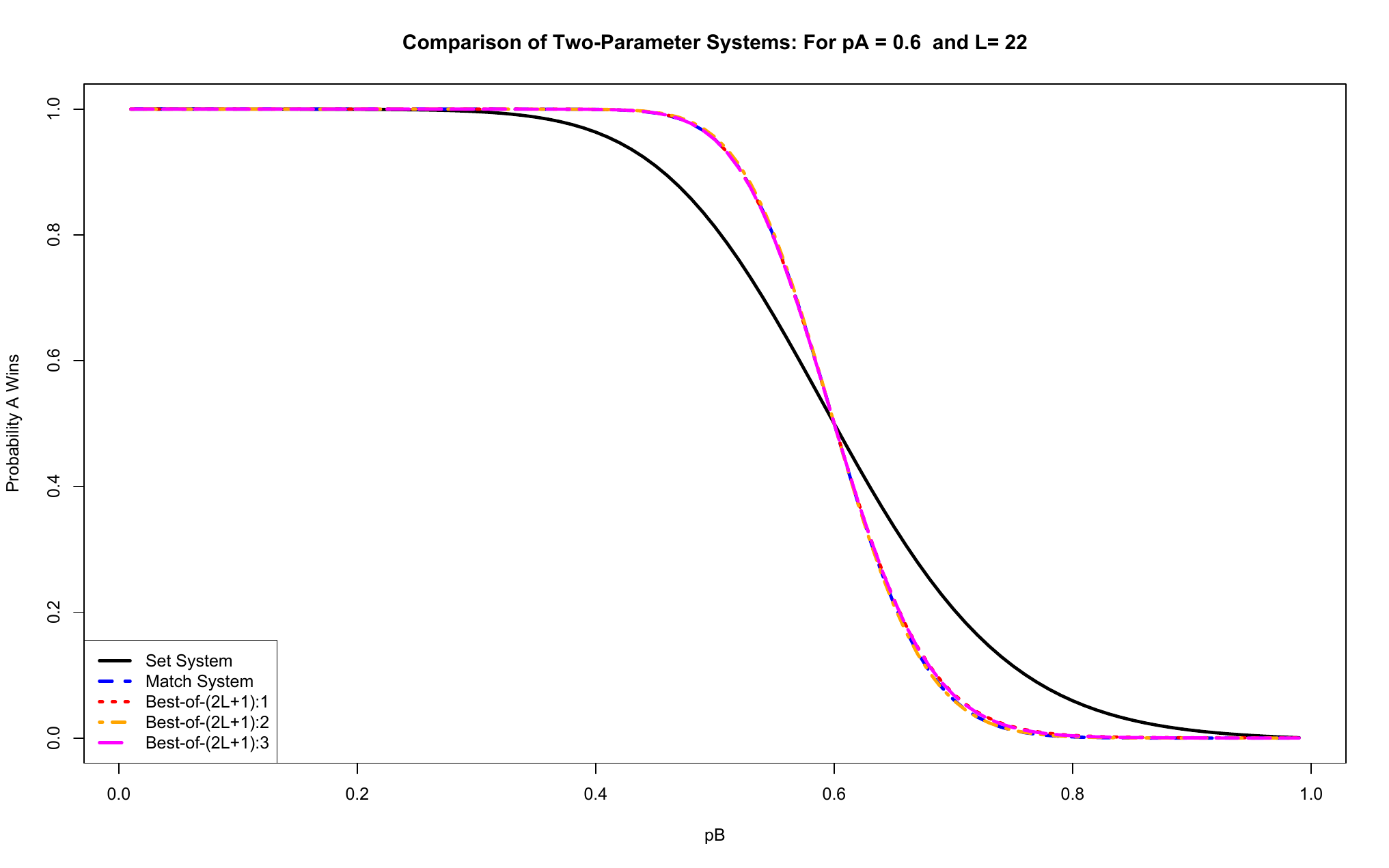} 
\includegraphics[width=\mymultB\textwidth,height=\mymultB\textwidth]{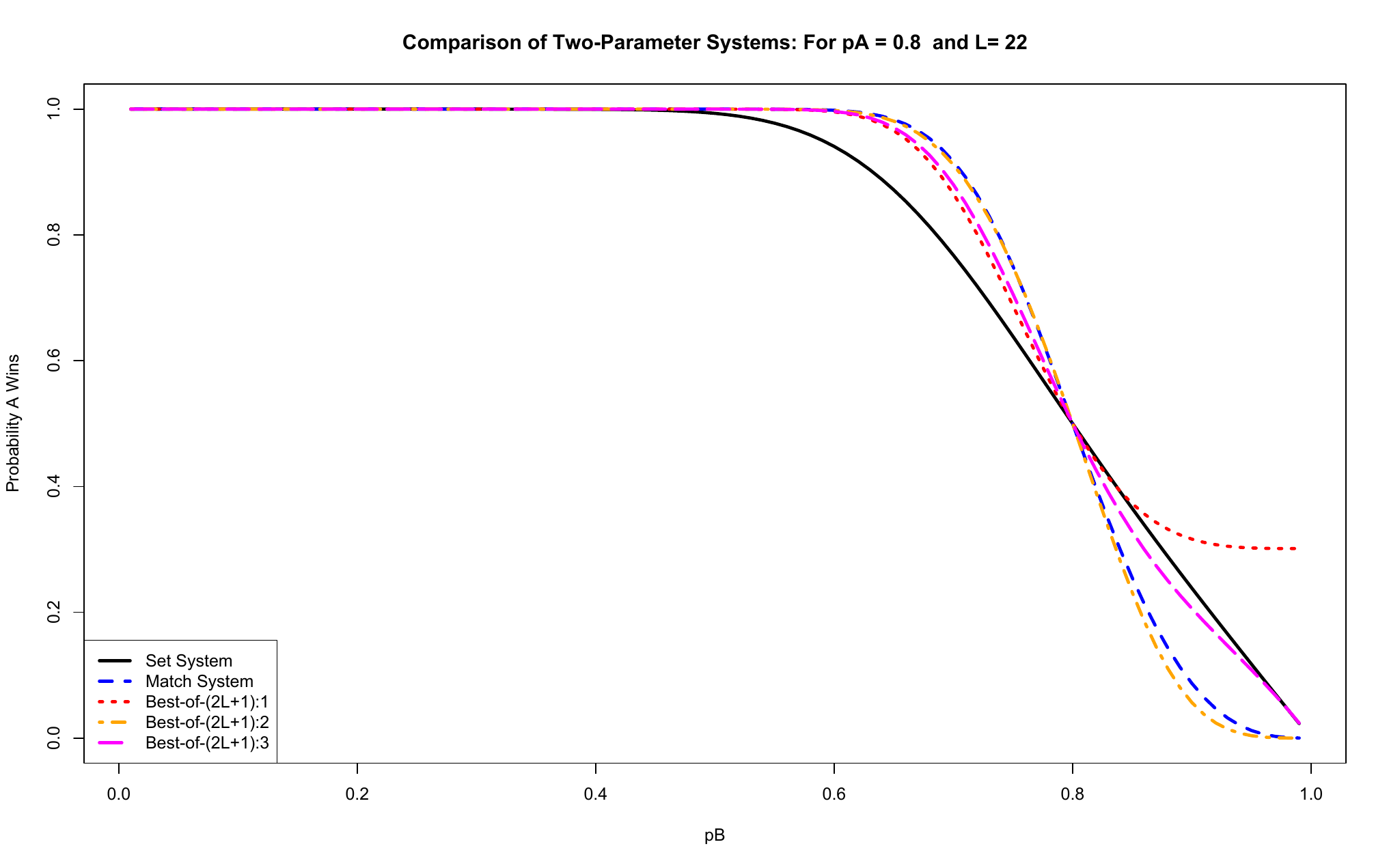}
\caption{Plots of the probability of player $A$ winning the match under the Tennis Best-of-Five-Sets System and the Best-of-$(2L+1)$-Games System under three type of tie-breakers for $L=22$ and for $pA \in \{.6,.8\}$ as $pB$ ranges over $(0,1)$.}
\label{fig: prob wrt pB of systems}
\end{center}
\end{figure}

For two joint distributions on $(p_A,p_B)$, Figure \ref{fig: efficiency two-parameter systems} provides the efficiency plots for these Best-of-$(2L+1)$ systems, together with that for the tennis Set and Match systems, as a function of $L$. The two panel plots are for two joint distributions $\Pi$ on $(p_A,p_B)$, which are products of two beta distributions with parameters (i) $(\alpha_1=2,\beta_1=1,\alpha_2=2,\beta_2=1)$ and (ii) $(\alpha_1=3,\beta_1=2,\alpha_2=2.5,\beta_2=2)$. Observe that among the Best-of-$(2L+1)$ systems, the second system, where the tie-breaker \RED{ends when either player becomes ahead by two games}, has the best efficiency. In fact, even with just $L=15$, so a Best-of-31-Games system, when  $(\alpha_1=2,\beta_1=1,\alpha_2=2,\beta_2=1)$, the \RED{second system performs almost the same as with the tennis Match system.} Of course, this system may also take a long time to finish, especially if the two players are excellent servers, as will be seen later. \RED{The first system is demonstrably worse, and it will require $L=115$, when $(\alpha_1=2,\beta_1=1,\alpha_2=2,\beta_2=1)$, so a Best-of-231-Games, to obtain approximately the same performance as the tennis Match system (efficiency of this Best-of-231-Games System 1 is 0.8755).} The third system, where the tie-breaker is the first to have an advantage of two {\em points}, will require $L=51$, so a Best-of-103-Games, to have an almost equivalent performance with the Match system (its efficiency is 0.8754). This will have a shorter duration compared to the second system. One may ask: What is the impact of the parameters for the joint distribution on $(p_A,p_B)$ and why are the efficiencies  quite different? For the parameter set $(\alpha_1=2,\beta_1=1,\alpha_2=2,\beta_2=1)$, there is a high chance of getting extremely high values of both $p_A$ and $p_B$, and this impacts drastically the performance of Best-of-$(2L+1)$:SG since the tie-breaker stage will usually happen. This will not be the case when $(\alpha_1=3,\beta_1=2,\alpha_2=2.5,\beta_2=2)$,
since extremely high-values of $p_A$ and $p_B$ will be less likely, hence the Best-of-$(2L+1)$-Games match will not often reach the tie-breaker stage.

\begin{figure}[h]
\begin{center}
\includegraphics[width=\mymultB\textwidth,height=\mymultB\textwidth]{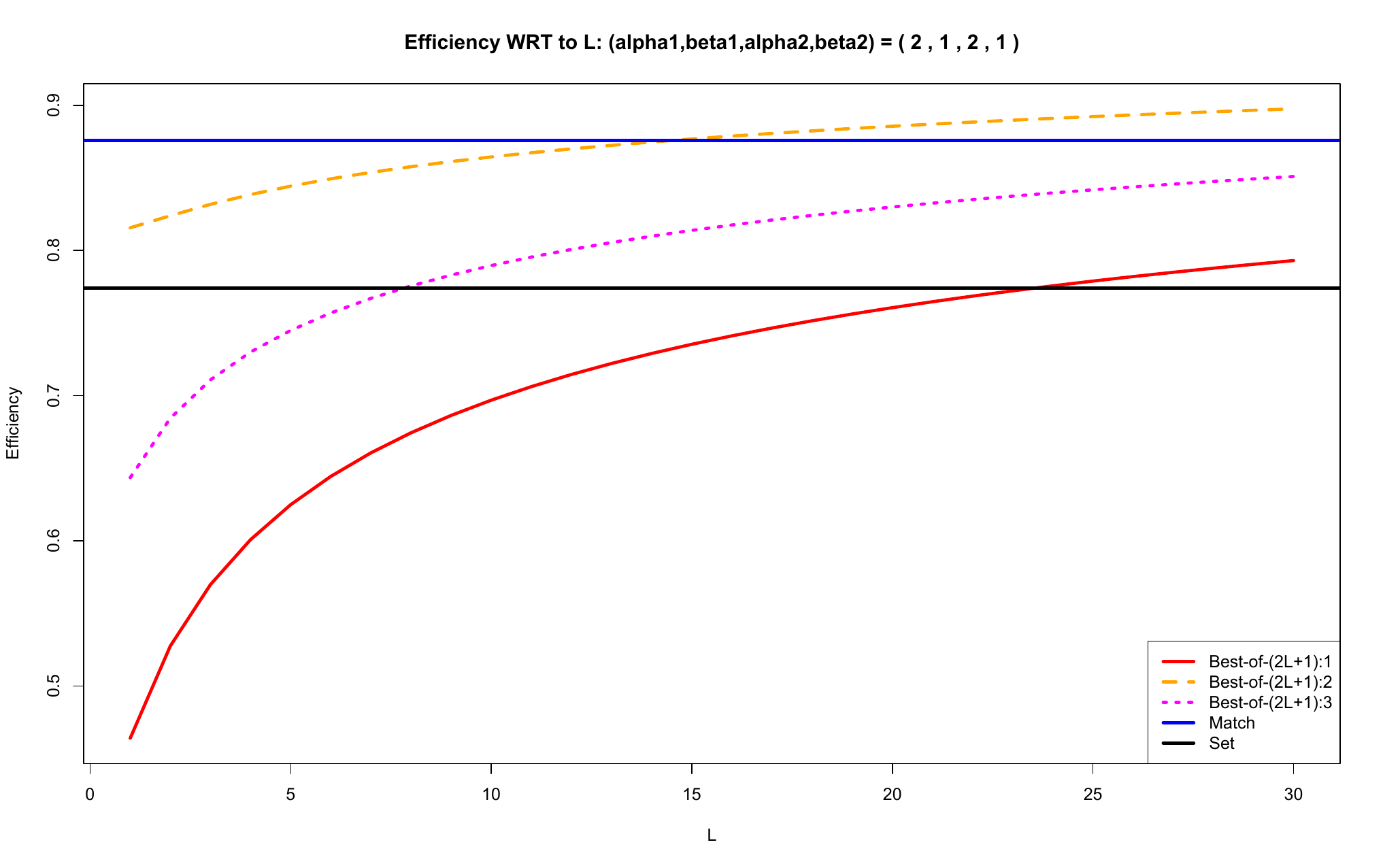} 
\includegraphics[width=\mymultB\textwidth,height=\mymultB\textwidth]{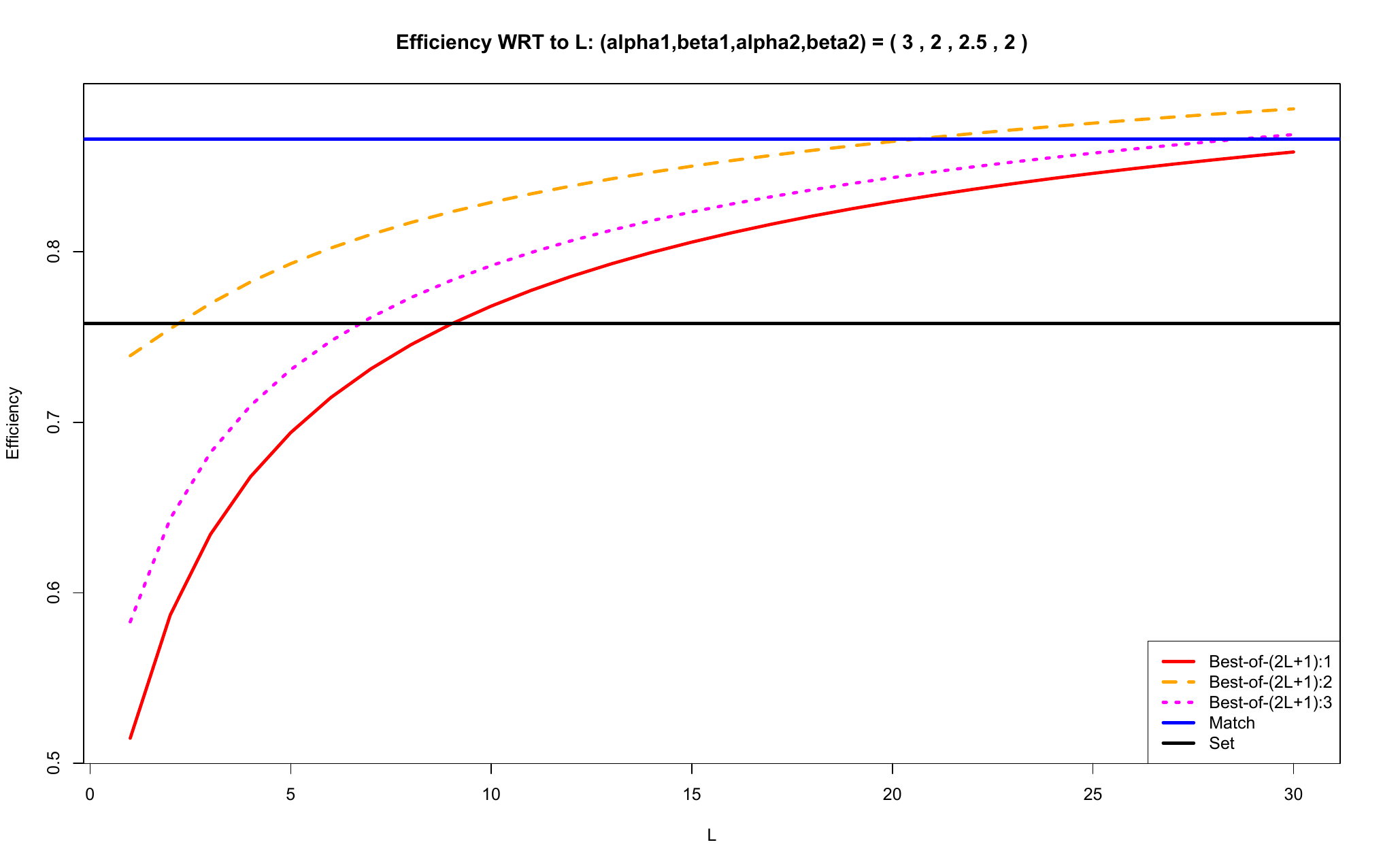} 
\caption{Efficiency plots for Two-Parameter Systems and two sets of parameters for the joint distribution of $(p_A,p_B)$.}
\label{fig: efficiency two-parameter systems}
\end{center}
\end{figure}

Table \ref{tab: efficiencies of two parameter systems} provides a summary of the efficiencies of different systems and for chosen $L$-values for the Best-of-$(2L+1)$-Games systems. \RED{Examining this table and Figure \ref{fig: efficiency two-parameter systems}, the most comparable Best-of-$(2L+1)$ system to the Best-of-Five-Sets tennis Match system is the Best-of-45-Games System 2, that is, with $L=22$, with a STT-type tie-breaker of {\em games} where the winner is the first to have an advantage of two games and with the service games either alternating between the two players, or using the sequence ABBAABBAA....} Since the performance of these different systems depends on the specific value of $(p_A,p_B)$, the efficiency measure is therefore affected as well by the choice of $\Pi(\cdot,\cdot)$, which determines the distribution of the $(p_A,p_B)$. For example, if $\Pi = Beta(3,2) \otimes Beta(2.5,2)$, the Best-of-$(2L+1)$ systems that are comparable to the Best-of-5-Sets Match system are: $L = 34$ for Bof$K$:1 with efficiency of 0.8666; $L=22$ for Bof$K$:2; and $L=29$ for Bof$K$:3.

\begin{table}[h]
\begin{center}
\begin{tabular}{|l|c|c|} \hline\hline
Two-Parameter System & \multicolumn{2}{c|}{Efficiency} \\ \cline{2-3}
& $Be(2,1) \otimes Be(2,1)$ & $Be(3,2) \otimes Be(2.5,2)$ \\ \hline\hline
Set Tie-Breaker's Tie-Breaker  & 0.5603 & 0.4601  \\ \hline
Set Tie-Breaker ($K=7$) & 0.6666 & 0.5928 \\ 
Set Tie-Breaker ($K=8$) & 0.6808 &  0.6103 \\ 
Set Tie-Breaker ($K=9$) & 0.6934 & 0.6257 \\ 
Set Tie-Breaker ($K=10$) & 0.7046 & 0.6394 \\ \hline
Set ($K=7$) & 0.7741 & 0.7578 \\
Set ($K=8$) & 0.7778 & 0.7603 \\
Set ($K = 9$) & 0.7814 & 0.7627 \\
Set ($K =10$) & 0.7847 & 0.7648 \\ \hline
Match ($K_0=7,K_1=7,Q=2$) & 0.8743 & 0.8651 \\
Match ($K_0=7,K_1=10,Q=2$) & 0.8757 & 0.8660 \\ \hline
Best-Of-$K1$ with $L = 5$ & 0.6249 & 0.6939 \\ 
Best-Of-$K1$ with $L = 15$ & 0.7353 & 0.8056 \\
Best-Of-$K1$ with $L = 22$ & 0.7685 & 0.8366 \\
Best-Of-$K1$ with $L = 29$ & 0.7904 & 0.8562 \\ \hline
Best-Of-$K2$ with $L = 5$ & 0.8443 & 0.7930 \\
Best-Of-$K2$ with $L = 15$ & 0.8768 & 0.8502 \\
Best-Of-$K2$ with $L = 22$ & 0.8884 & 0.8693 \\
Best-Of-$K2$ with $L = 29$ & 0.8965 & 0.8823 \\ \hline
Best-Of-$K3$ with $L = 5$ & 0.7449& 0.7308 \\
Best-Of-$K3$ with $L = 15$ & 0.8138 & 0.8233 \\
Best-Of-$K3$ with $L = 22$ & 0.8350 & 0.8497 \\
Best-Of-$K3$ with $L = 29$ & 0.8492 & 0.8667 \\
\hline
Best-Of-$60$ with STP ($K=10$) & 0.8613 & 0.8735 \\
\hline\hline
\end{tabular}
\caption{Efficiencies of the different two-parameter systems for a joint distribution on $(p_A,p_B)$ given by $\Pi(p_A,p_B) = Beta(2,1) \otimes Beta(2,1)$ and $\Pi(p_A,p_B) = Beta(3,2) \otimes Beta(2.5,2)$.}
\label{tab: efficiencies of two parameter systems}
\end{center}
\end{table}

\RED{From this table we observe that as $K$ increases in the ST system, the system's efficiency also increases. This consequently makes the efficiencies of the set and match systems to increase as $K_1$ increases from say $K_1=7$ to $K_1=10$. Figure \ref{fig: effect of K on ST} plots the probability of player $A$ winning the set tie-breaker as $K$ changes and for four sets of values of $(p_A,p_B)$: $(.60,.55)$, $(.55, .60)$, $(.80, .60)$, and $(0.30, 0.60)$. When $p_A > p_B$, then the probability of the first server (player $A$) winning the set-tiebreaker increases when $K$ increases; while if $p_A < p_B$, then it decreases with increasing $K$. This issue is not unimportant since recall that in the final set of major tennis tournaments (US Open, Wimbledon, Australian Open, French Open), the value of $K$ becomes 10 on the fifth set, whereas in the first four sets, the value of $K$ is 7. So, there is an impact on the choice of the value of $K$ to use in the set tie-breaker in that it increases the chances of the better player winning the set-tiebreaker, hence the set and the match.}

\BLUE{We also decided to consider a Best-of-60-Games system where the tie-breaker is a set-tiebreaker with $K=10$. The motivation for adding this system in the comparison is because the maximum number of games that could be played in a tennis best-of-five-sets match system is 60 games. This system differs from the Best-of-$61$-Games system 3 system only in the tie-breaker system when each player wins 30 games.}

\begin{figure}[h]
\begin{center}
\includegraphics[width=\mymultB\textwidth,height=\mymultB\textwidth]{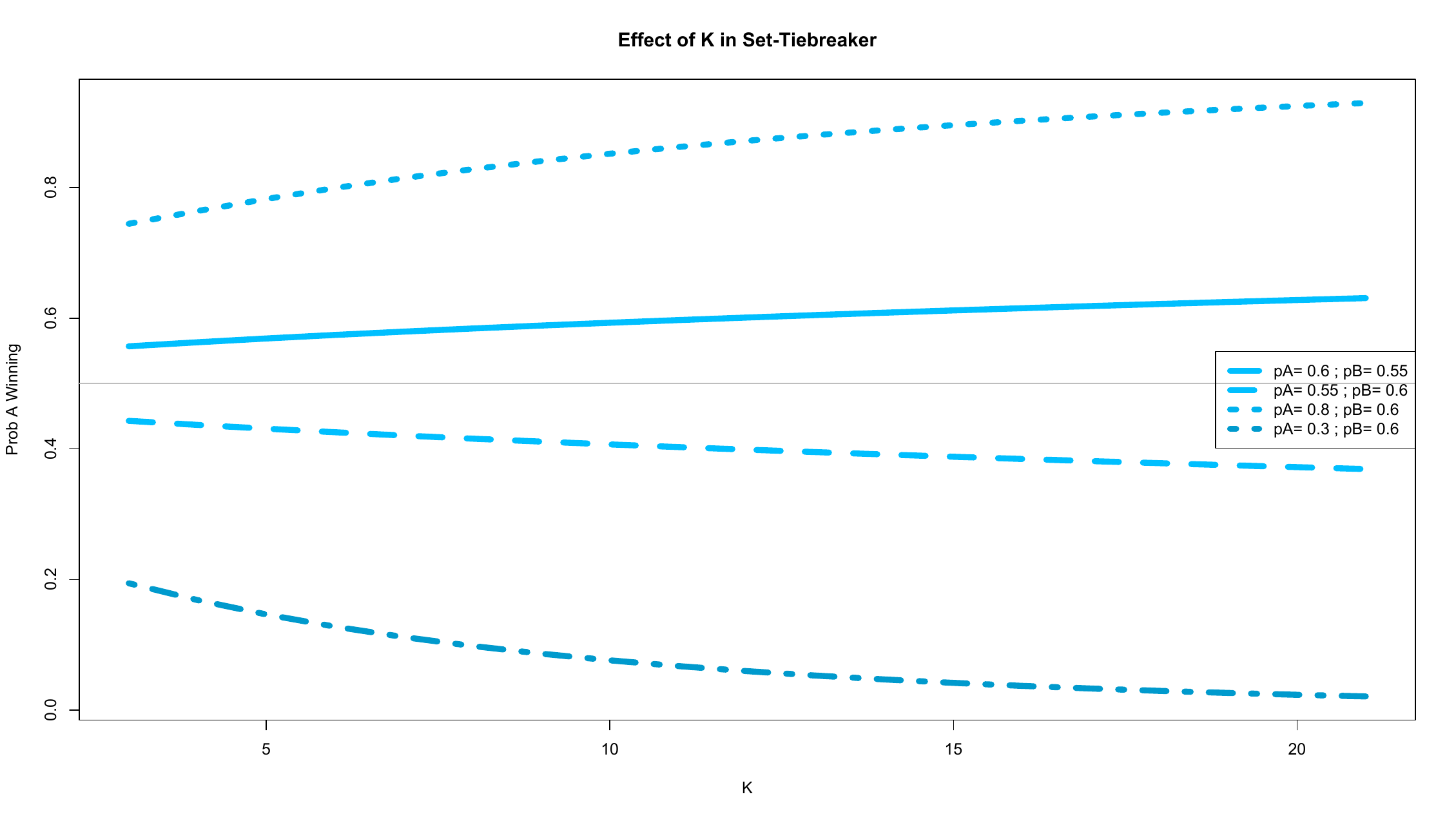}
\caption{Effect of $K$ in the Set Tie-Breaker. Plotted is the probability of player $A$, the first server, winning the tie-breaker.}
\label{fig: effect of K on ST}
\end{center}
\end{figure}

\subsection{Number of Points in Best-of-$(2L+1)$-Games Systems}

In comparing different match systems, it is also important to compare the number of points played to end the match. We have already obtained the mean and variance for the number of points in a Best-of-$(2Q+1)$-Sets, the tennis Match system. \RED{We consider in this subsection the mean and variance of the number of points played in the three Best-of-$(2L+1)$-Games systems.} We shall assume that player $A$ serves first, then they alternate serving the games.

Let $G_A$ ($G_B$) be the number of games won by player $A$ ($B$) in these systems. Then, we have the decompositions, 
%
\begin{eqnarray*}
 \{\mbox{$A$\ wins}\} & = & \left[\bigcup_{b=0}^\RED{L-1}\{G_A=L+1,G_B=b\}\right] 
 \cup 
 \left[\{G_A=G_B=L\} \cap \{\mbox{$A$ wins tie-breaker}\}\right]
 \end{eqnarray*}
 and
 \begin{eqnarray*}
 N & = & \left[\sum_{b=0}^\RED{L-1} I\{G_A=L+1,G_B=b\} N(L+1,b)\right] + 
 \left[\sum_{a=0}^\RED{L-1} I\{G_A=a,G_B=L+1\} N(a,L+1)\right] + \\ && I\{G_A=G_B=L\} (N(L,L) + N_{TB}),
 \end{eqnarray*}
 where $N$ is the total number of points played in the match; $N(a,b)$ is the number of points played if $G_A=a,G_B=b$; and $N_{TB}$ is the number of points played in the tie-breaker, which will depend on the tie-breaker system. With $t_A(n)$ ($t_B(n)$) being the number of games served by $A$ ($B$) during the first $n$ games, and with $J_A(n)$ indicating if $A$ is serving the $n$th game, then the JPMF of $(G_A,G_B)$, denoted by $p_{(G_A,G_B)}(a,b;p_A,p_B,L)$, is given by
 \begin{eqnarray*}
  \lefteqn{ p_{(G_A,G_B)}(a,b;p_A,p_B,L) \equiv p_{(G_A,G_B)}(a,b) \equiv p(a,b) = } \\ &&
  \left\{
  \begin{array}{l}
  \Pr\{B(t_A(a+b),\theta_G(p_A)) + B(t_B(a+b),1-\theta_G(p_B)) = a\} \times \\ I\{a \le L, b \le L\} \\
   \Pr\{B(t_A(L+b),\theta_G(p_A)) + B(t_B(L+b),1-\theta_G(p_B)) = L\} \times \\  \theta_G(p_A)^{J_A(L+b+1)} (1-\theta_G(p_B))^{1-J_A(L+b+1)}I\{a = L+1, b < L\} \\
      \Pr\{B(t_A(L+a),1-\theta_G(p_A)) + B(t_B(L+a),\theta_G(p_B)) = L\} \times \\  (1-\theta_G(p_A))^{J_A(L+a+1)} \theta_G(p_B)^{1-J_A(L+a+1)}I\{a < L, b = L+1\}
  \end{array}
  \right..
 \end{eqnarray*}
Since
$$N(a,b) = \sum_{i=1}^{t_A(a+b)} N_{GAi} + \sum_{j=1}^{t_B(a+b)} N_{GBj}$$
where $N_{GAi}$ ($N_{GBj}$) is the number of points played in the $i$th ($j$th) $A$ ($B$) game, and these random variables are independent, then by applying Wald's identity,
$$\mu(a,b) = E[N(a,b)] = t_A(a+b) \mu_G(p_A) + t_B(a+b) \mu_G(p_B);$$
$$\sigma^2(a,b) = Var[N(a,b)] = t_A(a+b) \sigma_G^2(p_A) + t_B(a+b) \sigma_G^2(p_B).$$

For the three tie-breaker systems, we have the following probabilities, means, and variances.

\begin{center}
\begin{tabular}{c} \hline\hline
{Tie-Breaker System 1 (``SG'')} \\ \hline
$ \theta_{TB1} = \Pr\{\mbox{$A$\ wins TB1}\}  =  \frac{1}{2}[\theta_G(p_A) + (1-\theta_G(p_B)] $ \\
$ \mu_{TB1} = E[N_{TB1}]  =  \RED{\frac{1}{2}[\mu_G(p_A) + \mu_G(p_B)]} $ \\
$ \sigma_{TB1}^2 = Var[N_{TB1}]  =  \RED{\frac{1}{2}[\sigma_G^2(p_A) + \sigma_G^2(p_B)] +
\frac{1}{4}[\mu_G(p_A) - \mu_G(p_B)]^2}$ \\ \hline\hline
{Tie-Breaker System 2 (``STTG'')} \\ \hline
$ \theta_{TB2} =  \Pr\{\mbox{$A$\ wins TB2}\}  =  \theta_{STT} (\theta_G(p_A),\theta_G(p_B)) $ \\
$ \mu_{TB2} = E[N_{TB2}]  =  \RED{\mu_{STTG}(p_A,p_B)} $ \\
$ \sigma_{TB2}^2 = Var[N_{TB2}]  = \RED{\sigma_{STTG}^2(p_A,p_B)} $ \\ \hline\hline
{Tie-Breaker System 3 (``STTP'')} \\ \hline
$ \theta_{TB3} =  \Pr\{\mbox{$A$\ wins TB3}\}  =  \theta_{STT}(p_A,p_B) $ \\
$ \mu_{TB3} = E[N_{TB3}]  =  \mu_{STT}(p_A,p_B) $ \\
$ \sigma_{TB3}^2 = Var[N_{TB3}]  =  \sigma_{STT}^2(p_A,p_B) $ \\ \hline\hline
\end{tabular}
\end{center}
%
where the expressions for the STTG system are given below:
\RED{
$$\nu(p_A,p_B) = \frac{1}{\theta_G(p_A) (1-\theta_G(p_B)) + (1 - \theta_G(p_A)) \theta_G(p_B)};$$
$$\mu_{STTG}(p_A,p_B) = \nu(p_A,p_B) (\mu_G(p_A)+\mu_G(p_B));$$
\begin{eqnarray*}
& \sigma_{STTG}^2(p_A,p_B)  = \nu(p_A,p_B) (\sigma_G^2(p_A) + \sigma_G^2(p_B)) + & \\ &  \nu(p_A,p_B)(\nu(p_A,p_B)-1) (\mu_G(p_A)+\mu_G(p_B))^2. &
\end{eqnarray*}
}
\RED{To establish these formulas for the STTG, simply note that the number of games served by player $A$ and by player $B$ will always be identical when the tie-breaker ends, and the number of games served by player $A$ (hence also by player $B$) has a geometric distribution with success probability being 
$$\theta_G(p_A) (1-\theta_G(p_B)) + (1 - \theta_G(p_A)) \theta_G(p_B)= \frac{1}{\nu(p_A,p_B)},$$
which is the probability of player $A$ holding his/her service game and breaking player $B$'s service game {\em or} player $A$ getting broken and player $B$ holding his/her service game. Then the results follows by the iterated rules for mean and variance and an application of Wald's identity.}

Finally, by again using the iterated rules for the mean and variance, with the conditioning on $(G_A,G_B)$, we obtain the probability of $A$ winning, and the mean and variance of $N$ as follows:

\begin{eqnarray*}
    \theta(p_A,p_B,L) & = & \Pr\{\mbox{$A$\ wins}\} = \sum_{b=0}^\RED{L-1} p(L+1,b) 
    + p(L,L) \theta_{TB}; \\
    \mu(p_A,p_B,L) & = & E\{N|(p_A,p_B,L)\} \\
    & = & \sum_{b=0}^\RED{L-1} p(L+1,b) \mu(L+1,b) + \sum_{a=0}^\RED{L-1} p(a,L+1) \mu(a,L+1) + \\ &&  p(L,L) [\mu(L,L) + \mu_{TB}]; 
   \end{eqnarray*}
   \begin{eqnarray*}
\sigma^2(p_A,p_B,L) &  = & Var\{N|(p_A,p_B,L)\} = E[Var(N|(G_A,G_B))] + Var[E(N|(G_A,G_B)]  \\
    & = & \left\{\sum_{b=0}^\RED{L-1} p(L+1,b) \sigma^2(L+1,b) + \sum_{a=0}^\RED{L-1} p(a,L+1) \sigma^2(a,L+1) + \right. \\ &&  \left. p(L,L) [\sigma^2(L,L) + \sigma_{TB}^2] \right\} +  \left\{ \left[ \sum_{b=0}^\RED{L-1} p(L+1,b) \mu(L+1,b)^2 + \right.\right. \\ && \left.\left.\sum_{a=0}^\RED{L-1} p(a,L+1) \mu(a,L+1)^2  +   p(L,L) [\mu(L,L) + \mu_{TB}]^2 - \right.\right. \\ && \left.\left. 
    \mu(p_A,p_B,L)^2 \right]\right\}; \\
     \sigma(p_A,p_B,L)  & = & StdDev(N) = +\sqrt{\sigma^2(p_A,p_B,L)}.
\end{eqnarray*}
Note that $\theta(p_A,p_B,L)$ is an alternative formula for obtaining the probability that $A$ wins, and it yields the same value as the $\theta_{BofK}(p_A,p_B,L)$ formulas earlier in this subsection.

With these expressions for the mean and variance of the number of points played in Best-of-$(2L+1)$-Games systems, we could refine our comparisons with the tennis Best-of-$(2Q+1)$-Sets system. However, since we have noted that System 2 of the Best-of-$(2L+1)$-Games system appears to be the most suitable system to compare to the tennis Match system, we only compare below the tennis Match system and the Best-of-$(2L+1)$-Games system 2.

Table \ref{tab: comparison of tennis match system and BOKG2 systems} provides comparisons for specific $(p_A,p_B)$-pairs with respect to the probability of player $A$ winning, the mean number of points played, and the standard deviation of the number of points played. Observe that whenever $p_A=p_B$, the probability of $A$ winning is always $1/2$, but the number of points played is very different. When $p_A=p_B$ and close to $0.5$, the Best-of-$(2L+1)$-Games system tends to have a smaller number of points played compared to the tennis Match system. However, when $p_A=p_B$ and close to either $0$ or to $1$, then the Best-of-$(2L+1)$-Games system requires more points, \RED{in fact many more points,} to be played. The reason for this is that the match will most likely go into a tie-breaker, but the tie-breaker will then require many points to finish, since the requirement is to be ahead by two games. Neither system appears to dominate the other, but with the Best-of-$(2L+1)$-Games system possibly taking a much longer time (in terms of points played) especially when two players are highly excellent servers ($p_A$ and $p_B$ close to 1) or both are very bad servers ($p_A$ and $p_B$ close to 0), though of course the latter will not be the case for these top-notch tennis tournaments. In particular, observe that when $p_A=p_B=.90$, for the Best-of-$45$-Games system, the mean number of points is 3094 with standard deviation of 3076. This is so since $\theta_G(p_A) = \theta_G(p_B) =  .9985$, which are the probabilities governing the games and the STT-type games tie-breaker, whereas for the tennis match system the mean is 287 and standard deviation of 59. For this value of .90 for both $p_A$ and $p_B$, the probability for each player winning the match under either system is 0.50 as they are of equal abilities.

\begin{table}[h]
\begin{center}
\begin{tabular}{||l||c|c||c||c|c|c|c||c||} \hline\hline
Match &\multicolumn{2}{c||}{Player's}& Best-of-5& \multicolumn{4}{c||}{Best-of-$(2L+1)$-Games System 2} & Best-of-60 \\
\cline{2-3} \cline{5-8}
   Property  &  $p_A$ & $p_B$  & Sets & $L=5$ & $L=15$ & $L=22$ & $L=29$ & Games with STP(10) \\ \hline\hline
PrA wins & 0.5 & 0.5 &  0.5000  &   0.5000   &   0.5000  &    0.5000  &    0.5000 & 0.5000 \\
MeanPts  &   0.5 & 0.5 & 271.8082  &  67.7109  &  188.6962  &  276.5922  &  365.5683 & 377.3182 \\
StdPts  &    0.5 & 0.5 & 61.7407 &   21.1487  &   28.2093  &   32.8004  &   36.9742 & 35.6255 \\ \hline
PrA wins & 0.5 & 0.6 &  0.0483  &   0.1798  &    0.0762  &    0.0443  &    0.0264 & 0.0255 \\
MeanPts   &  0.5 & 0.6 & 222.9703  &  63.3857 &  170.2394  & 245.7167  & 321.0804 & 331.6710 \\
StdPts   &   0.5 & 0.6 & 53.9818 &   19.6496  &  27.6063  &  33.3907  &  38.6388 & 38.8967 \\ \hline
PrA wins & 0.6 & 0.5 &  0.9517 &    0.8202  &    0.9238  &    0.9557   &   0.9736 & 0.9745 \\
MeanPts  &   0.6 & 0.5 & 220.5927 &   62.7244 &  169.1378 &  244.4906 &  319.7853 & 330.3688 \\
StdPts   &   0.6 & 0.5 & 53.7999  &  19.8563  &  27.7637  &  33.4785  &  38.6724 & 38.9186 \\ \hline
PrA wins & 0.8 & 0.6 &  0.9982  &  0.9621  &   0.9929   &  0.9979   &  0.9994 & 0.9993 \\
MeanPts  &   0.8 & 0.6 & 177.4933 &  60.8815  & 146.6769 &  211.2299 &  276.2519 & 285.5100 \\
StdPts    &  0.8 & 0.6 & 36.0239  & 28.9254  &  21.4157  &  23.0037  &  25.6342 & 25.8299 \\ \hline
PrA wins & 0.9 & 0.8 &   0.9130 &   0.9391  &   0.9412  &   0.9435   &  0.9462 & 0.8564 \\
MeanPts &    0.9 & 0.8 & 253.8929 & 414.0690 &  433.4630 &  457.1944 &  487.6213 & 292.0742 \\
StdPts  &    0.9 & 0.8 &  57.2131 & 405.8694  & 392.0642  & 377.5522  & 361.0946 & 16.7038 \\ \hline
PrA wins & 0.9 & 0.9  & 0.5000 &    0.5000   &   0.5000  &    0.5000  &    0.5000 & 0.5000 \\
MeanPts  &   0.9 & 0.9 & 287.4960 & 3085.6010 & 3089.1297 & 3093.7618 & 3100.1070 & 286.6689 \\
StdPts  &    0.9 & 0.9 & 59.4101 & 3080.6199 & 3078.3865 & 3075.5323 & 3071.7220 & 11.8486 \\ \hline\hline
\end{tabular}
\caption{Comparison of the probabilities of player $A$ winning, the mean and standard deviation of the number of points of the tennis match system (Best-of-5-Sets) versus Best-of-$(2L+1)$-Games System with an STT-Type tie-breaker of games (system 2) and a Best-of-60-Games system with a 10-point tie-breaker system, for different pairs of values of $(p_A,p_B)$. {\bf Legend:} `PrA wins' is the probability that player $A$ wins the match; `MeanPts' (`StdPts') is the mean (standard deviation) of the number of points played in the match.}
\label{tab: comparison of tennis match system and BOKG2 systems}
\end{center}
\end{table}

Another way of comparing these systems is by looking at the difference or ratio in the probabilities of player $A$ winning the match; the ratio of the mean of the number of points played in each system; and the ratio of the standard deviation of the number of points played in each system. To do so, we present filled contour plots (obtained by using the {\tt filled.contour} object in {\tt R} \cite{R}) over the region $(p_A,p_B) \in (0,1) \times (0,1)$. We only present the plots associated with the case $L=22$, which appears to be the value that makes the tennis Match system and the Best-of-$(2L+1)$-Games system 2 comparable from the efficiency comparisons and from Table \ref{tab: comparison of tennis match system and BOKG2 systems}. Figure \ref{fig: filled contour plots of match vs BOKG - probabilities} provides the contour plots for the difference and for the logarithm of the ratio of the probabilities of player $A$ winning the match under the Best-of-5-Sets tennis match system and the Best-of-$(2L+1)$-Games system 2, while Figure \ref{fig: filled contour plots of match vs BOKG - points} provides the contour plots of the ratio of the means and standard deviations for these match systems. 

\begin{figure}[h]
\begin{center}
\includegraphics[width=\mymultA\textwidth,height=\mymultA\textwidth]{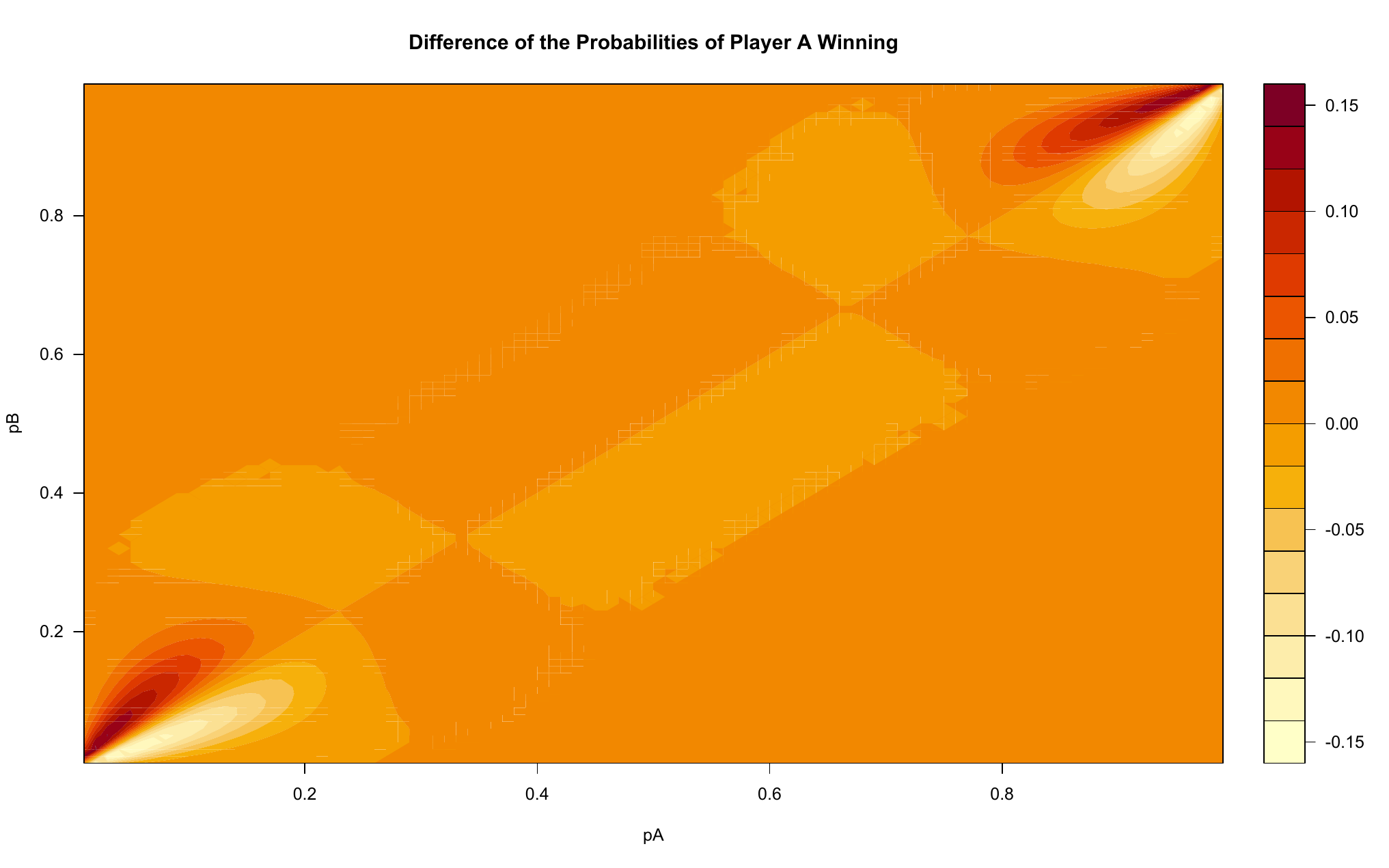} 
\includegraphics[width=\mymultA\textwidth,height=\mymultA\textwidth]{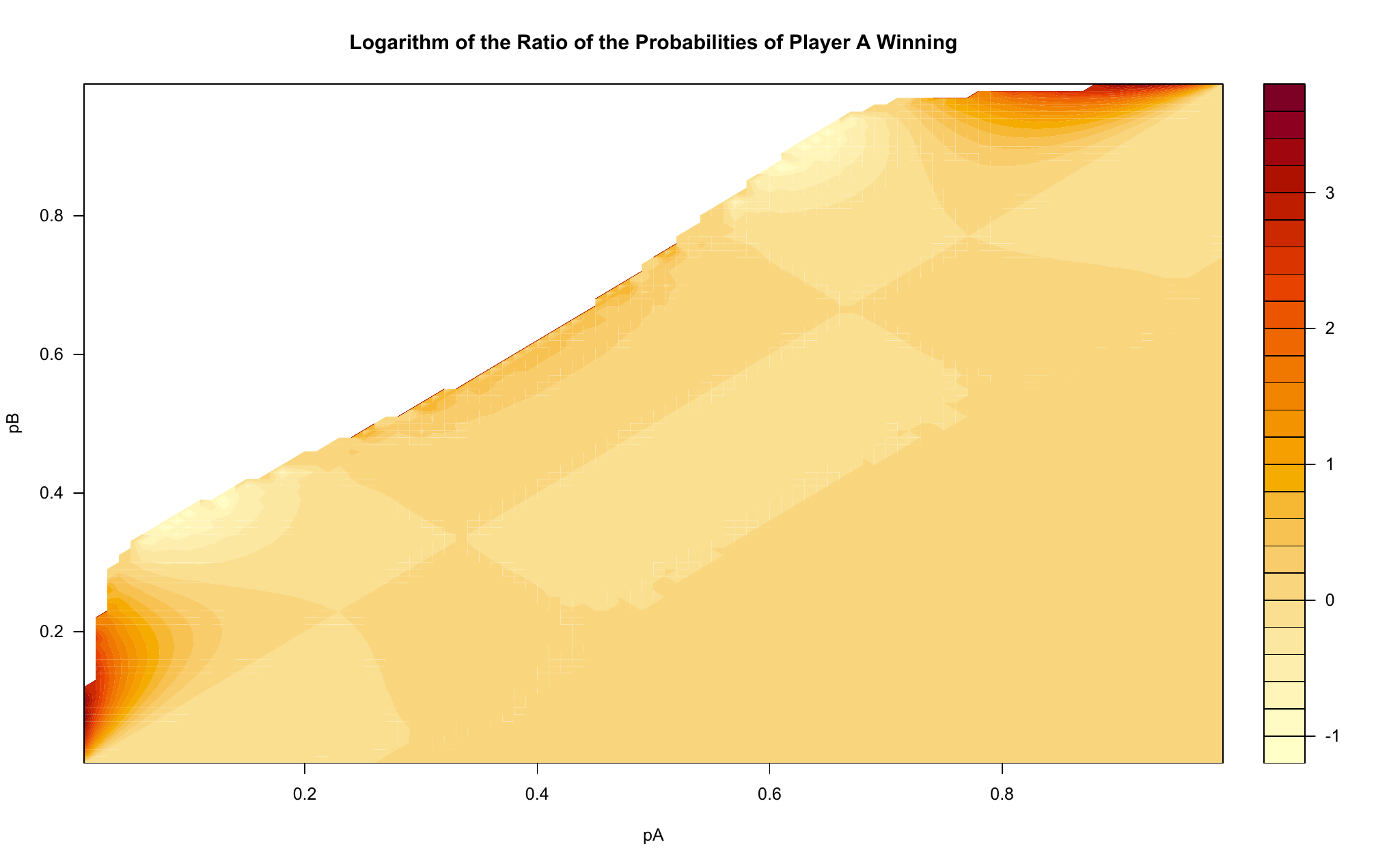}
\caption{Filled contour plots of the difference in the probability of player $A$ winning, and for the logarithm of the ratio of the probability of player $A$ winning for the tennis best-of-5-sets system and the best-of-$(2L+1)$-games system with an STT-type tiebreaker of games with $L=22$.}
\label{fig: filled contour plots of match vs BOKG - probabilities}
\end{center}
\end{figure}

\begin{figure}[h]
\begin{center}
\includegraphics[width=\mymultA\textwidth,height=\mymultA\textwidth]{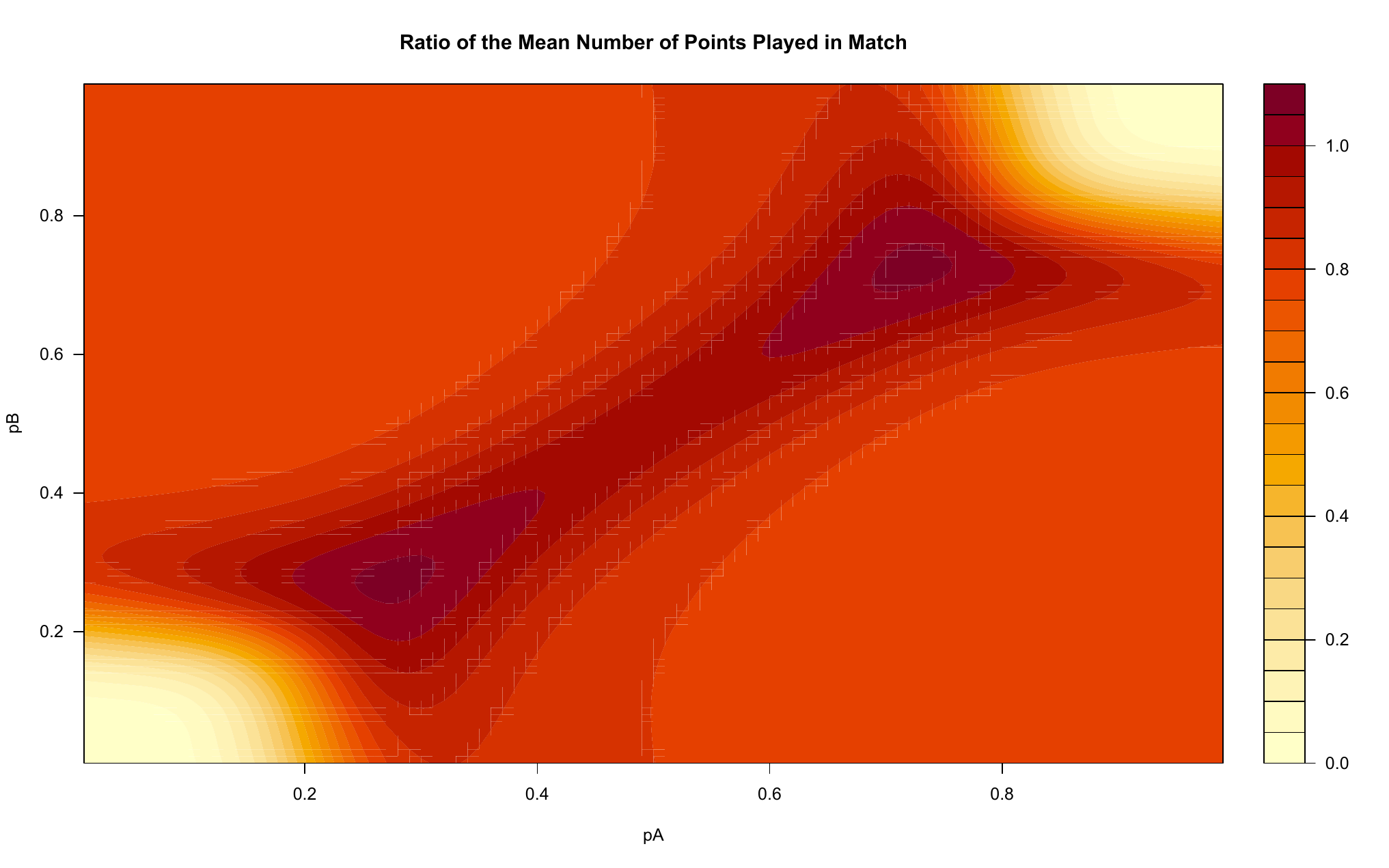} 
\includegraphics[width=\mymultA\textwidth,height=\mymultA\textwidth]{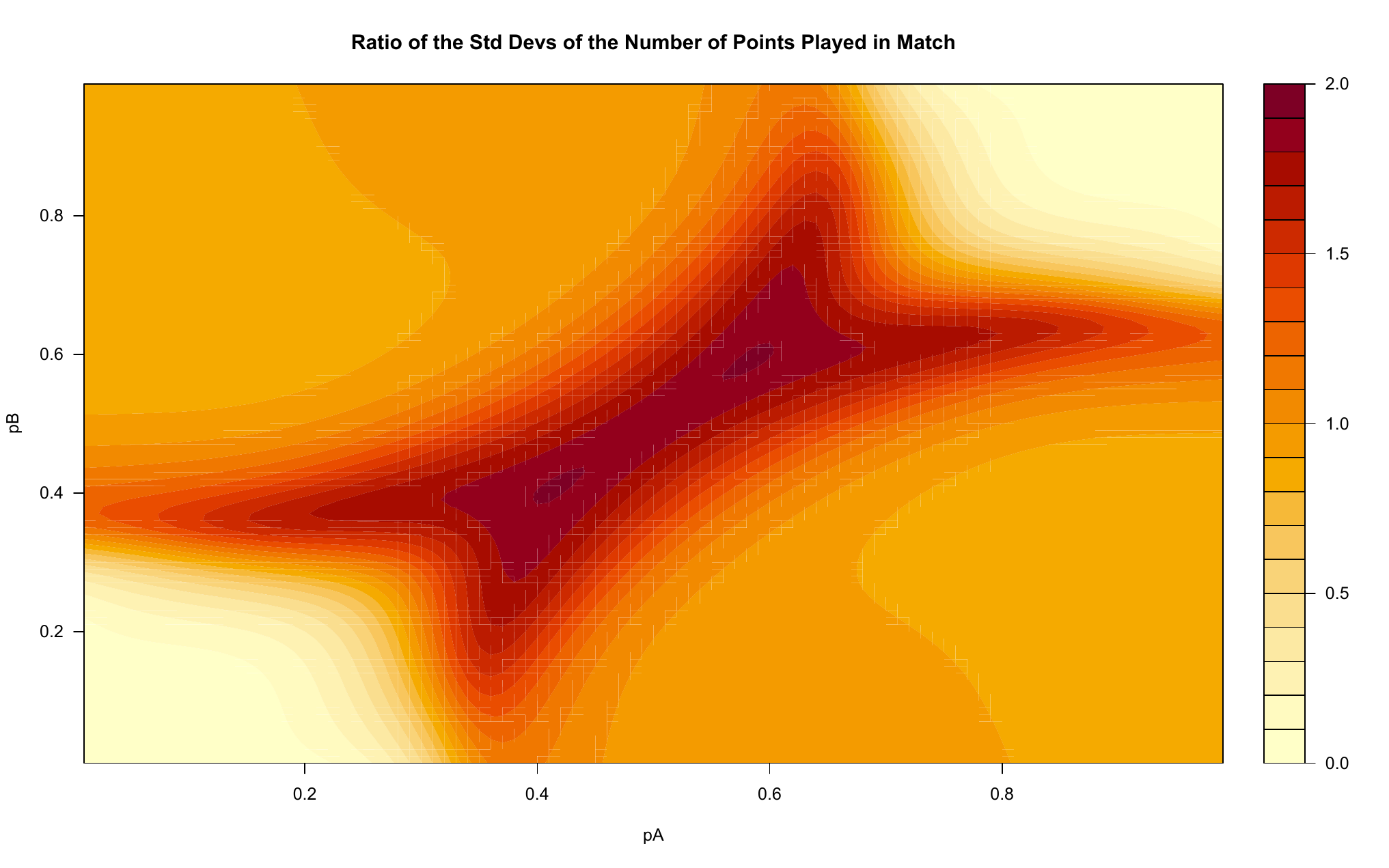}
\caption{Filled contour plots of the ratio of the mean number of points played, and the ratio of the standard deviation of the number of points played in the match for the tennis best-of-5-sets system and the best-of-$(2L+1)$-games system with an STT-type tiebreaker of games with $L=22$.}
\label{fig: filled contour plots of match vs BOKG - points}
\end{center}
\end{figure}

\RED{Based on considering the probability of the better player winning the match and the number of points to be played in the match, it appears that the Best-of-Five-Sets tennis match system is preferable to the Best-of-45-Games with the STTG tie-breaker system. \BLUE{On the other hand, the Best-of-60-Games system 4 may have some claim to being competitive to the tennis match system since the probabilities are not so different, though this system tends to require more points to be played compared to the tennis match system, but at a lower rate than that for the Bof$K$:2 system, in particular, see the result under $(p_A=.90, p_B=.80)$ and $(p_A=.90,p_B=.90)$.} However, one could argue that the tennis match system may still provide more excitement to spectators, since a player could lose the first two sets with scores of 6-0 and 6-0, but still win the match with scores of 7-6 (5), 7-6 (6), and 7-6 (10) in the last three sets! (Of course, we recognize that such an outcome will {\em not} at all be probabilistically consistent with the idealized assumptions of this paper.)}

\section{Concluding Remarks}
\label{sec: Concluding Remarks}

This work focused on the probabilistic or stochastic aspects of tennis scoring systems, from the point, to the game and its tie-breaker, to the set and its tie-breaker and also the tie-breaker's tie-breaker, and to the match itself. A {\em simplistic and idealized} model is postulated where outcomes of points played are independent, and the success probability of a point server remains the same for the duration of the match, an assumption that has been made in other papers such as Carter and Crews \cite{Carter_Crews_1974} and Newton and Keller \cite{Newton_Keller_2005}. Clearly, these identical server success probability per player and the independence assumption are unrealistic, as in a real tennis match, there are many factors that could dynamically alter elements of the match such as crowd support (e.g., a player with a diaspora following such as for the Filipina sensation Alexandra Eala \cite{Eala_diaspora_2026}), momentum change (see Jackson and Mosurski \cite{JacksonMosurski1997} for a discussion of the so-called {\em psychological momentum}), injury, missed calls, etc.. However, as Klaasen and Magnus \cite{Klaasen_Magnus_2001} pointed out, the IID-ness assumption may still provide an acceptable approximation and serve as a vehicle to examine the properties of scoring systems in tennis.

The basic probabilistic problem addressed in this work is as follows: Given that player $A$ (player $B$) has probability $p_A$ ($p_B$) of winning the point on his/her serve, what is the probability that he/she wins a service game; he/she wins a set; and he/she wins the match. This problem offers a great pedagogical opportunity for probability calculations, as well as obtaining other characteristics, such as the mean and standard deviation of the number of points that will be played in a game and its tie-breaker, in a set and its tie-breaker, and also in the whole match. We also addressed intriguing questions such as whether the probability that a tennis match will end equals one; whether the first server in a set or the match has an advantage, even under the hypothetical situation of equal abilities ($p_A=p_B$) of the two players; and the impact of the number of points needed to win a set tie-breaker.

Tennis scoring systems are technically statistical decision systems to determine the better player. There could be other systems that are possible, such as best-of-$K$-games systems. In any statistical decision system, erroneous decisions could occur, since the decisions are simply based on a finite number of points played. Thus, in tennis scoring systems, an inferior player could still win the match. The quality of a scoring system is therefore measured by its ability to determine the better player. In this paper, we also therefore proposed a system efficiency measure to enable the comparison of competing systems. In essence, the efficiency measure is based on comparing its performance with an oracle system and in considerations of the possible values of the two parameters $(p_A,p_B)$. Based on these efficiency measures, we compared tennis scoring systems with best-of-$K$-games systems, both in terms of the probability of the better player winning, as well as in terms of the duration of the match, measured in terms of the number of points that will be played in the match. An interesting, surprising result is that a best-of-45-games system, with a tie-breaker system where the first player to gain an advantage of two games ({\em games} instead of {\em points}), is comparable to the best-of-five-sets tennis system in terms of the probability of determining the better player. However, such a system may also require many points to be played, especially when the two players are excellent servers, so that the tennis match system may still be preferable, aside from the fact that tennis match system in Grand Slam Tournaments of Best-of-Five-Sets for Men or Best-of-Three-Sets for Women are historical with a very rich tradition and also generates considerable excitement for the spectators. In addition, it endows the game of tennis with a truly unique scoring system. 

Though this work focused on the probabilistic aspects, there are also a myriad of statistical questions that arise, which we hope to address in separate works. Foremost among these questions is the problem of rating or ranking players based on the outcomes of head-to-head matches. It has not escaped the attention of the authors that the probability formulas in this paper could be used in constructing rating or ranking systems of tennis players. Another set of problems is of an inferential nature. For example, suppose that we had observed the scores of a best-of-five-sets match. What then is the best estimates of $p_A$ and $p_B$, and how do we measure the uncertainty in the resulting estimates? How does one construct a confidence region for $(p_A,p_B)$? Or, if one is of a Bayesian temperament, given a joint prior distribution for $(p_A,p_B)$, what is the posterior distribution of $(p_A,p_B)$ based on the result of a best-of-five-sets match? How could one form a credible region for $(p_A,p_B)$? How do we compute the posterior probability that player $A$ is better than player $B$? Perhaps, a more basic statistical question is to decide, based on the outcomes of the points in a match, whether the assumptions of unchanging probabilities per point, as well as the independence of the outcomes of points, are actually tenable. Klaasen and Magnus \cite{Klaasen_Magnus_2001} looked into this question and based on empirical data, concluded that this assumption is untenable, though they added that the deviations from {IID}-ness are small, hence the IID hypothesis still provides a good approximation in many cases. A more ambitious project, on the other hand, is to consider the whole tournament and all the games in the tournament, for instance, all 127 matches in the US Open Tennis Tournament. Given the outcomes of all these matches, could one develop a probabilistic or a generative model that could reasonably model the whole tournament? Such a generative model may, for instance, be used to simulate the whole tournament (see Newton and Aslam \cite{NewtonAslam2009} which proposed a Markov Chain model that could be used to simulate tennis tournaments)! These are interesting questions for future investigations. We envision that these future studies may benefit from the probabilistic analysis and results provided in this  paper.

\bibliographystyle{plainnat}
\bibliography{DoubleDelight}

\end{document}